\documentclass[11pt
,reqno]{amsart}

\usepackage{amsmath, amssymb}
\usepackage{stmaryrd}
\usepackage{enumitem}
\usepackage[all]{xy}
\usepackage{extarrows}
\usepackage{mathtools}
\usepackage{dynkin-diagrams}
\usepackage{caption}
\usepackage[super]{nth}
\usepackage{csquotes}
\usepackage{soul}
\usepackage[french, english]{babel}

\usepackage{hyperref}

\DeclarePairedDelimiter\floor{\lfloor}{\rfloor}

\makeatletter
\newenvironment{abstracts}{%
	\ifx\maketitle\relax
	\ClassWarning{\@classname}{Abstract should precede
		\protect\maketitle\space in AMS document classes; reported}%
	\fi
	\global\setbox\abstractbox=\vtop \bgroup
	\normalfont\Small
	\list{}{\labelwidth\z@
		\leftmargin3pc \rightmargin\leftmargin
		\listparindent\normalparindent \itemindent\z@
		\parsep\z@ \@plus\p@
		
		\itemsep\medskipamount
	}%
}{%
	\endlist\egroup
	\ifx\@setabstract\relax \@setabstracta \fi
}

\newcommand{\abstractin}[1]{%
	\otherlanguage{#1}%
	\item[\hskip\labelsep\scshape\abstractname.]%
}
\makeatother

\theoremstyle{theorem}
\newtheorem{theorem}{Theorem}[section]
\newtheorem{corollary}[theorem]{Corollary}
\newtheorem{lemma}[theorem]{Lemma}
\newtheorem{prop}[theorem]{Proposition}
\newtheorem*{question}{Question}

\theoremstyle{definition}
\newtheorem{definition}[theorem]{Definition}

\theoremstyle{remark}
\newtheorem{example}[theorem]{Example}
\newtheorem{remark}[theorem]{Remark}

\theoremstyle{theorem}
\newtheorem{maintheorem}{Theorem}

\newtheorem{mainprop}[maintheorem]{Proposition}

\def\id{\mbox{Id} }
\def\et{\quad\mbox{and}\quad}

\DeclareMathOperator{\GL}{GL}
\DeclareMathOperator{\SL}{SL}

\DeclareMathOperator{\pr}{pr}

\DeclareMathOperator{\rank}{rank}

\DeclareMathOperator{\Aut}{Aut}
\DeclareMathOperator{\SAut}{SAut}
\DeclareMathOperator{\alg}{alg}

\DeclareMathOperator{\red}{red}
\DeclareMathOperator{\kk}{\textbf{k}}
\DeclareMathOperator{\C}{\mathbb{C}}
\DeclareMathOperator{\R}{\mathbb{R}}
\DeclareMathOperator{\Q}{\mathbb{Q}}
\DeclareMathOperator{\Z}{\mathbb{Z}}
\DeclareMathOperator{\N}{\mathbb{N}}

\DeclareMathOperator{\mm}{\mathfrak{m}}
\DeclareMathOperator{\nn}{\mathfrak{n}}
\DeclareMathOperator{\OO}{\mathcal{O}}

\DeclareMathOperator{\CH}{CH}
\DeclareMathOperator{\End}{End}
\DeclareMathOperator{\aff}{aff}

\renewcommand{\id}{\mathrm{id}}

\def\RR{\mathbb{R}}
\def\PP{\mathbb{P}}
\def\AA{\mathbb{A}}
\def\ZZ{\mathbb{Z}}
\def\CC{\mathbb{C}}
\def\QQ{\mathbb{Q}}
\def\GG{\mathbb{G}}

\def\c{\textrm{comp}}

\renewcommand{\k}{\textbf{k}}
\DeclareMathOperator{\g}{\mathfrak{g}}
\DeclareMathOperator{\Spec}{Spec}

\newcommand{\set}[2]{\{\,#1 \ | \ #2\,\}}
\newcommand{\Bigset}[2]{\left\{\,#1 \ \Big| \ #2\,\right\}}

\begin{document}
\title[Existence of Embeddings into Algebraic Groups]
{Existence of Embeddings of Smooth 
Varieties into Linear Algebraic Groups}
\author{Peter Feller \and Immanuel van Santen}
\address{Peter Feller, ETH Z\"urich, Department of Mathematics, R\"amistrasse 101, CH-8092 Z\"urich, Switzerland}
\email{peter.feller@math.ch}
\address{Immanuel van Santen, University of Basel, Department of Mathematics and Computer Science, Spiegelgasse $1$, CH-$4051$ Basel, Switzerland}
\email{immanuel.van.santen@math.ch}

\begin{abstracts}
%
%
	\abstractin{english}
	We prove that every smooth affine variety of dimension $d$ embeds into every simple algebraic group of dimension at least $2d+2$. We do this by establishing the existence of embeddings of smooth affine varieties into the total space of certain principal bundles. For the latter we employ and build upon parametric transversality results for flexible affine varieties due to Kaliman. By adapting a Chow-group-based argument due to Bloch, Murthy, and Szpiro, we show that our result is optimal up to a possible improvement of the bound to $2d+1$. 
	
	In order to study the limits of our embedding method, we use rational homology group calculations of homogeneous spaces and we establish a domination result for rational homology of complex 
	smooth varieties. 
\end{abstracts}

\maketitle

\tableofcontents

\vspace{-0.3cm}
\addtocontents{toc}{\protect\setcounter{tocdepth}{1}}

\section{Introduction}

In this text, \emph{varieties} are understood to be (reduced) algebraic varieties
over a fixed algebraically closed field $\kk$ 
of characteristic zero endowed with the Zariski topology. We will focus on \emph{affine} varieties---closed subvarieties of the affine space $\AA^n$. A closed embedding, \emph{embedding} for short, 
$f\colon Z\to X$ of an affine variety $Z$ into an affine variety $X$ is a morphism such that
$f(Z)$ is closed in $X$ and $f$ induces an isomorphism $Z \simeq f(Z)$ of varieties.

A focus of this text lies on embeddings into the underlying varieties of affine algebraic groups. Recall that
an affine algebraic group, an \emph{algebraic group} for short, is a closed subgroup of the general linear
group $\GL_k$ for some positive integer
$k$.
An algebraic group is \emph{simple} if it has no non-trivial connected normal subgroup.
We prove the following embedding theorem.

\begin{maintheorem}
	[{Theorem~\ref{thm.simple}}]
	\label{thmintro:main}
	Let $G$ be the underlying affine variety of a simple algebraic group  
	and $Z$ be a smooth affine variety. If $\dim G > 2 \dim Z + 1$,
	then $Z$ admits an embedding into $G$.
\end{maintheorem}

In case $\dim G$ is even, the dimension assumption on $\dim Z$ in terms of $\dim G$ from Theorem~\ref{thmintro:main} is optimal; while in case $\dim G$ is odd, the dimension assumption can at best be relaxed by one, that is from $\dim G > 2 \dim Z + 1$ to $\dim G \geq 2 \dim Z + 1$.
Indeed, we have the following.

\begin{mainprop}
	[{Corollary~\ref{cor.non-emeddability}}]
	\label{cor:optim}
	Let $G$ be the underlying affine variety of an algebraic group of dimension $n\geq 1$.
	Then, for every integer
	$d \geq \frac{n}{2}$
	there exists a smooth irreducible affine variety $Z$ of dimension $d$ that does not admit
	an embedding into~$G$.
\end{mainprop}

Theorem~\ref{thmintro:main} fits well 
in the context of classical embedding theorems in different categories. 
We provide this context in the next subsection and an 
outline of the proof of Theorem~\ref{thmintro:main} in the subsection after that.

Before that
, we discuss a domination result for the rational homology of smooth varieties, which we believe to be of independent interest.
The connection to Theorem~\ref{thmintro:main} comes from an application that explains 
one crucial obstacle to weakening the dimension assumption to $\dim G \geq 2 \dim Z + 1$ 
in our proof of Theorem~\ref{thmintro:main}; see Proposition~\ref{prop:intronomapsZtoG/H} in the outline of the proof of Theorem~\ref{thmintro:main} below. For this domination result we work over the field of complex numbers, and 
{rational} homology groups $H_\ast(\cdot; \QQ)$
are taken with respect to the Euclidean topology.

\begin{maintheorem}
\label{mthm:surjectiononhomology} Let $f\colon X\to Y$ be a proper 
surjective morphism between complex $n$-dimensional smooth varieties. Then 
the induced map on $k$-th rational homology $H_k(X;\Q)\to H_k(Y;\Q)$ is a surjection for all 
integers $k \geq 0$.
\end{maintheorem}
We will formulate a version of Theorem~\ref{mthm:surjectiononhomology} (see Theorem~\ref{thm:THMCgeneral}) in the category of complex manifolds that can be understood as a generalization of Gurjar's Theorem~\cite{Gu1980Topology-of-affine} (see Remark~\ref{rmk:Gurjar}).
We prove Theorem~\ref{mthm:surjectiononhomology} via a version of
Hopf's Theorem on the Umkehrungshomomorphismus for non-compact topological manifolds; see Appendix~\ref{AppendixA}.

\subsection*{Context: embedding theorems in various settings}

\subsubsection*{Holme-Kaliman-Srinivas embedding theorem}
When considering affine varieties as closed subvarieties of the affine space $\AA^n$, it is natural to wonder about their minimal embedding dimension in affine space. It turns out that
every smooth affine variety $Z$ embeds into $\AA^n$ for $n\geq 2\dim Z +1$; see Holme~\cite{Ho1975Embedding-obstruct}, 
Kaliman~\cite{Ka1991Extensions-of-isom
}, and Srinivas~\cite{Sr1991On-the-embedding-d}.
This can be understood as an analog of the following classical result in differential topology.

\subsubsection*{Whitney embedding theorem}The weak Whitney 
embedding theorem states that every closed smooth manifold $M$ can be embedded into $\R^{n}$ for $n\geq 2\dim M+1$~\cite{Wh1936Differentiable-man}. 
The fact that Whitney's result also holds in case $n=2\dim M$ is known as the strong Whitney embedding theorem, based on the so-called Whitney trick \cite{Wh1944The-self-intersect}.
Furthermore, if $M$ is a closed smooth manifold such that $\dim M$ is not a power of 2, then Haefliger-Hirsch~\cite{HaHi1963On-the-existence-a} proved that $M$ embeds into $\R^{2\dim M-1}$. In contrast, the real projective 
space of dimension $2^k$ for $k \geq 0$ yields a $2^k$-dimensional smooth manifold that
does not embed into $\RR^{2 \cdot 2^k-1}$ \cite{Pe1957Some-non-embedding}.

\subsubsection*{Holomorphic embeddings of Stein manifolds}
Focusing on $\kk=\C$ (hence $\AA^n=\C^n$), it is natural to compare the Holme-Kaliman-Srinivas 
result with the holomorphic setup.
It is 
known that every Stein manifold $M$ of dimension at least $2$ can be holomorphically embedded into $\C^{n}$ for $n> \frac{3}{2}\dim M$; see Eliashberg-Gromov~\cite{ElGr1992Embeddings-of-Stei} and Sch\"urmann~\cite{Sc1997Embeddings-of-Stei}.
Examples of Forster show that this dimension condition is optimal~\cite{Fo1970Plongements-des-va}.

Focusing on more general targets, Andrist, Forsterni$\check{c}$, Ritter, and Wold
proved that for every Stein manifold $X$ that satisfies the  (volume) density property and every
Stein manifold $M$ such that $\dim X \geq 2 \dim M + 1$, there exists a holomorphic embedding 
of $M$ into $X$~\cite{AnFoRi2016Proper-holomorphic}.
In particular, if $G$ is a characterless algebraic group,
then $G$ satisfies the density property by Donzelli-Dvorsky-Kaliman \cite[Theorem~A]{DoDvKa2010Algebraic-density-} or $G$ is isomorphic to $\CC$. Hence,
every smooth affine variety $Z$ with $2 \dim Z + 1 \leq \dim G$ admits a holomorphic
embedding into $G$.
As far as the authors know, it remains open whether a 
dimension improvement \`a la Eliashberg-Gromov is possible.

\subsubsection*{Embeddings into projective varieties}
Comparing with the projective setting,
a further analog of the weak Whitney embedding theorem
states that every smooth projective variety $Z$ embeds into $\PP^n$ provided
$n \geq 2 \dim Z + 1$; see Lluis~\cite{Ll1955Sur-limmersion-des}.

While the Holme-Kaliman-Srinivas embedding result concerning affine spaces generalizes to
some, possibly all affine algebraic groups,
the embedding result due to Lluis concerning projective spaces cannot generalize to 
projective algebraic groups, better known as abelian varieties. In fact,
each rational map $Z \dashrightarrow A$ from a rationally connected variety $Z$ into 
an abelian variety $A$ is constant; see \cite[Corollary to Theorem 4, Ch.~II]{La1983Abelian-varieties}.

\subsubsection*{Optimality of the dimension condition for algebraic embeddings}
As seen above, in many categories, $d$-dimensional objects embed into the standard space of dimension $2d$, e.g.~the strong Whitney embedding theorem, or even lower like 
in the case of the Eliashberg-Gromov result. 
In contrast, even the analog of the strong Whitney embedding theorem is known to fail for affine
varieties.
Indeed, by a result of Bloch-Murthy-Szpiro~\cite{BlMuSz1989Zero-cycles-and-th}, for every $d \geq 1$ 
there exists a $d$-dimensional smooth affine variety that does not embed into $\AA^{2d}$. In fact, their argument (based on Chow group calculations) suffices to also yield Proposition~\ref{cor:optim}, as we will see in Section~\ref{sec:non-embed}.

Incidentally, in the Lluis embedding theorem, the dimension bound is optimal in
the sense that for every $d \geq 1$ there 
is a smooth projective variety of dimension $d$ that does not
admit an embedding into $\PP^{2d}$; see
Horrocks-Mumford~\cite{HoMu1973A-rank-2-vector-bu} and Van de Ven~\cite{Va1975On-the-embedding-o}.

\subsection*{Proof strategy: an embedding method and its limits}

\subsubsection*{Proof strategy of the Holme-Kaliman-Srinivas theorem and an 
approach to more general targets} We recall the basic idea behind the Holme-Kaliman-Srinivas embedding theorem, which uses the same method as the proofs of 
the weak Whitney embedding theorem and the Lluis embedding theorem.
To show that every smooth affine variety $Z$ embeds into $X=\AA^{2\dim Z+1}$, one starts from an arbitrary embedding $Z\subseteq\AA^m$ for some large integer 
$m \gg 2\dim Z+1$, and shows that the composition of the inclusion $Z \subseteq \AA^m$ with a generic linear projection $\AA^m\to \AA^{2\dim Z+1}$ is still an embedding.

For more general targets $X$, one looses the availability of (many) projections from $\AA^m$ to $X$.
In contrast with the above strategy, instead, we consider a morphism $\pi \colon X\to\AA^{\dim Z}$ and a finite morphism
$Z\to \AA^{\dim Z}$ (guaranteed to exist by Noether normalization) in order to build our embedding $Z\to X$ as a factorization of $Z \to \AA^{\dim Z}$ through $\pi$.  
This approach is similar to the setup of Eliashberg-Gromov and their notion of relative embedding using their `background map'; see~\cite[Section~2]{ElGr1992Embeddings-of-Stei}.
A strength of this approach lies in the following fact:
checking that a morphism $f\colon Z\to X$ is an embedding (i.e.~a proper injective morphism with everywhere injective differential), reduces to checking that $f$ is injective and has everywhere injective differential, since any morphism that can be composed with another yielding a finite (in particular proper) morphism is proper. Sloppily speaking, one gets properness `for free'.

\subsubsection*{Outline of the proof of Theorem~\ref{thmintro:main}} More concretely, our approach to prove Theorem~\ref{thmintro:main} can be understood in two steps. 
Step one 
involves finding 
a specific subvariety of a simple algebraic group using classical algebraic group theory. 
Using parametric transversality results, in step two 
we promote finite maps with target the base space of a principal bundle to 
embeddings into the total space.
Here the total space is the subvariety constructed in 
step one. 
These two steps will be treated in detail in Sections~\ref{sec:appl:groups} and~\ref{sec:EmbIntoPriBund}, respectively. We provide a short outline, where we fix a smooth affine variety $Z$ and a simple algebraic group $G$ with $\dim G > 2\dim Z + 1$. 

\textbf{Step one.} 
We find a closed codimension one subvariety $X\subset G$ isomorphic to $\AA^{\dim Z}\times H$,
where $H$ is a characterless closed subgroup of $G$. This will be achieved using a well-chosen maximal parabolic subgroup in $G$ and constitutes the bulk of Section~\ref{sec:appl:groups}. 
It turns out that $G$ itself cannot be a product of the form $\AA^m \times H$ for any variety $H$ underlying an algebraic group and $m>0$; hence, the $X$ we found has the largest possible dimension.

\textbf{Link between the two steps.} 
We note that step one reduces the proof of Theorem~\ref{thmintro:main} to finding an embedding of $Z$ into 
$\AA^{\dim Z}\times H$.
We set up a principal bundle together with a finite morphism from $Z$ into the base. For the latter, denoting by $\GG_a$ the underlying additive algebraic 
group of the
ground field $\kk$,
we consider the principal $\GG_a$-bundle $\rho\colon \AA^{\dim Z}\times H\to \AA^{\dim Z}\times H/U$, where $U$ is a closed subgroup of $H$ that is isomorphic to $\GG_a$. 
Using Noether normalization, one has a finite morphism $Z\to\AA^{\dim Z}$, which yields a morphism $r\colon Z\to \AA^{\dim Z} \times H/U$ by composing with
a section of the projection $\eta \colon \AA^{\dim Z}\times H/U \to \AA^{\dim Z}$ to the first factor.
Writing $X\coloneqq \AA^{\dim Z}\times H$ and 
$Q \coloneqq \AA^{\dim Z}\times H/U$, 
we have the following commutative diagram
\begin{equation}\label{eq:thm2.5asdiag}
	\begin{gathered}
	\xymatrix{
			&  X \ar[dr]^{\pi}\ar[d]^{\rho}
			& \\
			Z\ar[r]^{r}
			&Q\ar[r]^-{\eta}&
			\AA^{\dim Z}\, \quad.
		}
	\end{gathered}
\end{equation}

\textbf{Step two.} 
We consider the following setup generalizing~\eqref{eq:thm2.5asdiag}. This constitutes our embedding method mentioned earlier.
Consider a principal $\GG_a$-bundle $\rho\colon X\to Q$, 
where $X$ is a smooth irreducible affine variety of dimension at least $2\dim Z+1$,
and a finite morphism $Z\to \AA^{\dim Z}$ that is the composite of morphisms $r\colon Z\to Q$
and $\eta\colon Q\to \AA^{\dim Z}$ such that the following holds.
The composition $\pi\coloneqq \eta\circ\rho\colon X\to \AA^{\dim Z}$
is a smooth morphism such that there are sufficiently many automorphisms of $X$
that fix $\pi$ (see Definition~\ref{def.sufficiently_transitively}).
Given this setup, we
show that there exists an embedding of $Z$ into $X$ (see Theorem~\ref{thm.Product}).
This is done in Section~\ref{sec:EmbIntoPriBund} building on notions and results due to Kaliman~\cite{Ka2020Extensions-of-isom}. 
Next, we explain in broad strokes how we build such an embedding.

Note first that $\rho \colon X \to Q$ restricts to a trivial $\GG_a$-bundle over any affine subvariety of $Q$.  Hence, there exists a morphism 
\[
	f_0 \colon Z \to \rho^{-1}(r(Z)) \simeq r(Z) \times \GG_a \subset X
\]
such that $\rho \circ f_0 = r$. Then we use a generic automorphism $\varphi$ of $X$ that fixes 
$\pi$ to construct an `improved'  morphism $f_1 \colon Z \to X$ with 
$\rho \circ f_1 = \rho \circ \varphi \circ f_0$. `Improved' means 
that $f_1$ and its differential are `more injective' than $f_0$ and its differential, respectively.
After finitely many, say $k$, such `improvements',
we get an injective morphism $f_k \colon Z \to X$ with everywhere injective differential. Note that by construction we have that
$\pi \circ f_k = \eta \circ r \colon Z \to \AA^{\dim Z}$ is finite. This shows the properness of $f_k$,
and thus $f_k$ is an embedding of $Z$ into $X$.



\subsubsection*{The case of small dimensions and other cases}
While, in general, we do not know how to weaken the dimension assumption to the optimal $\dim G\geq 2\dim Z+1$ in Theorem~\ref{thmintro:main}, we are able to treat the case 
$\dim G \leq 8$: every smooth affine variety $Z$ embeds 
in every characterless algebraic group $G$ of dimension $\leq 8$ if $2 \dim Z + 1 \leq \dim G$; 
see Proposition~\ref{prop.lowdimension}.

From the method of the proof it is clear that Theorem~\ref{thmintro:main} generalizes to products of 
a simple algebraic group with affine spaces (Theorem~\ref{thm.simple}) and to products of a 
semisimple algebraic group with affine spaces 
but with a stronger dimension assumption (Theorem~\ref{thm.semisimple}). 
In case the dimension of the affine space in the product is big enough, we get in fact 
the embedding result with 
the optimal dimension assumption; see Corollary~\ref{cor.product}. In particular, we give a new proof of
the Holme-Kaliman-Srinivas embedding theorem; see Remark~\ref{rem:reproveHKS}.

Our embedding method also yields that if a smooth affine
variety $Z$ embeds into a smooth affine variety $X$ with $\dim X \geq 2 \dim Z + 1$,
then $Z$ embeds into the target of every finite \'etale surjection 
from $X$, whenever $X$
has sufficiently many automorphisms; see Corollary~\ref{cor:finite_etale}.
In particular, Theorem~\ref{thmintro:main} generalizes to homogeneous spaces
of simple algebraic groups with finite stabilizer; see Proposition~\ref{prop.characterization_suff_trans_algebraic_groups}.

\subsubsection*{Limits of the method and relation to Theorem~\ref{mthm:surjectiononhomology}}
We end the introduction by coming back to a statement from earlier: the seemingly unrelated Theorem~\ref{mthm:surjectiononhomology} explains a major obstacle to treating the case $\dim G=2\dim Z+1$. We explain this in terms of the above short two step outline.
In fact, in step one we find $\pi\colon X\to \AA^{\dim Z}$ by restricting 
the natural projection $p \colon G\to G/H$ for some closed subgroup $H$ to 
$X \subseteq G$, i.e.~$\pi \coloneqq p |_X \colon X\to p(X) \subseteq G/H$. 
 However, by the dimension assumption that we need for step two, if we were to follow that strategy, we would have to choose $X\subseteq G$ of full dimension. Hence, assuming w.l.o.g. that $G$ is irreducible, we would have to choose $X=G$ and would have to replace $\AA^{\dim Z}$ with a homogeneous space 
$G/H$ of dimension $\dim Z$ in diagram~\eqref{eq:thm2.5asdiag}.
For the embedding method from step two to work for 
$G/H$ in place of $\AA^{\dim Z}$ in diagram~\eqref{eq:thm2.5asdiag}, we need in particular a finite morphism from $Z$ to $G/H$; compare Theorem~\ref{thm.Product}. However, there exist $Z$ such that no finite morphism from $Z$ to $G/H$ exists.
Concretely, working over $\C$, rational homology calculations for homogeneous spaces 
(see Proposition~\ref{prop.homotopy_non-vanishing}) and Theorem~\ref{mthm:surjectiononhomology} 
yield the following result.

\begin{mainprop}[Proposition~\ref{prop.no-finite-morph}] \label{prop:intronomapsZtoG/H}
Let $Z$ be a simply-connected complex smooth 
algebraic variety with the rational homology of a point. If $G/H$ is a $\dim Z$-dimensional complex homogeneous space of a complex simple algebraic group $G$, then there is no proper surjective
morphism from $Z$ to $G/H$. 
\end{mainprop}
%
And indeed, we do not know whether such $Z$ embed into simple algebraic groups of dimension $2\dim Z+1$. Concretely, the authors cannot answer the following 
question, even over $\CC$ and for contractible $Z$.

\begin{question}
Does every $7$-dimensional smooth affine variety embed into $\SL_4$?
\end{question}

Addendum: In a new arXiv preprint, Kaliman has answered this question affirmatively~\cite[Theorem 1.1]{Ka2021Holme-type-theorem}. In fact, more generally, he proves that, 
if $G$ is a semisimple algebraic group such that its Lie algebra is a product of Lie algebras
of special linear groups, then every smooth affine variety $Z$ with $2 \dim Z + 1 \leq \dim G$
admits an embedding into $G$.

\subsection*{Acknowledgements} We thank J\'er\'emy Blanc, Adrien Dubouloz, Stefan Friedl, Matthias Nagel, Patrick Orson, Pierre-Marie Poloni, and Paula Tru{\"o}l for helpful conversations.
Moreover, we would like to thank the anonymous referee for the detailed and helpful comments. We are in particular grateful to them for pointing out an error in our original argument for Proposition~\ref{prop.sufficiently_transitive_on_fibers_quotient}.

PF gratefully acknowledges support by the SNSF Grant 181199.

\medskip 
\noindent

\addtocontents{toc}{\protect\setcounter{tocdepth}{2}}

\section{Embeddings into the total space of a principal bundle}\label{sec:EmbIntoPriBund}

For the main result in this section the following definition will be useful:

\begin{definition}
	\label{def.sufficiently_transitively}
	Let $X$ be a variety. A subgroup $G$ of the group
	of algebraic automorphisms $\Aut(X)$ acts 
	\emph{sufficiently transitively on $X$} if the natural action on $X$ is $2$-transitive
	and the natural action on $(TX)^\circ$ is transitive, where $(TX)^\circ$ denotes 
	the complement of the zero-section in the total space $TX$ of the tangent bundle of $X$.
\end{definition}

Let us recall the definition of an algebraic subgroup of an automorphism group which goes back to Ramanujam~\cite{Ra1964A-note-on-automorp}.

\begin{definition}
	\label{def.alg_subgroup}
	Let $X$ be a variety. A subgroup $H \subset \Aut(X)$ is called \emph{algebraic subgroup}
	if there exists an algebraic group $G$ and a faithful algebraic action $\rho \colon  G \times X \to X$
	such that $H$ is the image of the homomorphism 
	$f_\rho \colon G \to \Aut(X)$ induced by $\rho$. 
\end{definition}

\begin{remark}	
	\label{rem.alg_subgroup}
	Note that the algebraic group $G$ in Definition~\ref{def.alg_subgroup} 
	is uniquely determined by $H$ in the following sense: 
	if $G'$ is another algebraic group with a faithful algebraic action $\rho'$
	on $X$ such that $f_{\rho'}(G') = H$, then there exists an isomorphism of algebraic groups
	$\sigma \colon G' \to G$ such that $f_{\rho'} = f_\rho \circ \sigma$ \cite[Theorem~9]{KrReSa2019Is-the-Affine-Spac}. This allows us to identify $G$ and $H$.
\end{remark}

Moreover, we will use the following subgroups of the automorphism group of a variety:
\begin{definition}
	Let $X$ be a variety. Then $\Aut^{\alg}(X)$ 
	denotes the subgroup of $\Aut(X)$ that is generated
	by all \emph{connected} algebraic subgroups of $\Aut(X)$.
	
	If $X$ comes equipped with a morphism $\pi \colon X \to P$, then $\Aut_P(X)$ denotes
	the subgroup of $\Aut(X)$ that consists of the $\sigma \in \Aut(X)$ with $\pi \circ \sigma = \pi$.
	We define $\Aut^{\alg}_P(X)$ as the subgroup of $\Aut_P(X)$ that is generated by all
	connected algebraic subgroups of $\Aut(X)$ that lie in $\Aut_P(X)$. 
\end{definition}

The main result to construct embeddings in this article is the following theorem.
Note that $\GG_a$ denotes the underlying additive algebraic group of the 
ground field \textbf{$\k$}. The proof of the theorem is contained in Subsection~\ref{subsec:Proof_embedd_into_tot_space_principal_bundle}.

\begin{theorem}
	\label{thm.Product}
	Let $X$ be a smooth irreducible affine variety such that:
	\begin{enumerate}[label=\alph*), leftmargin=*]
		\item \label{thm.Product_a} 
				 There is a principal $\GG_a$-bundle $\rho \colon X \to Q$;
		\item \label{thm.Product_b} 
				There is a smooth morphism $\pi \colon X \to P$ such that
				$\Aut_P^{\alg}(X)$ acts sufficiently transitively on each fiber of $\pi$;
		\item \label{thm.Product_c} 
				There is a morphism $\eta \colon Q \to P$ that satisfies $\eta \circ \rho = \pi$.
	\end{enumerate}
	If there exists a smooth affine variety $Z$ such that $\dim X \geq 2 \dim Z + 1$ and
	\begin{enumerate}[label=\alph*), leftmargin=*]
		\setcounter{enumi}{3}
		\item \label{thm.Product_d} there exists a morphism $r \colon Z \to Q$ such that 
		$\eta \circ r \colon Z \to P$
		is finite and surjective,
	\end{enumerate}
	then there exists an embedding of $Z$ into $X$. 
\end{theorem}

Part of Theorem~\ref{thm.Product} can be illustrated by the following diagram
\[
	\xymatrix{
		Z \ar[rr]^{\textrm{$\exists$ embedding}} & & X \ar[dr]^-{\pi}\ar[d]^-{\rho}
			\\
		& Z \ar[r]^-{r} & Q\ar[r]^-{\eta}&
		P \ar@{}[r]_-{.} &
	}
\]

\begin{remark}
	\label{rem.thm.Product}
	Let $X$ be a smooth affine irreducible variety
	and assume that conditions~\ref{thm.Product_a},~\ref{thm.Product_b}, \ref{thm.Product_c} of Theorem~\ref{thm.Product} are satisfied. If $Z$ is a smooth affine variety with
	$\dim X \geq 2 \dim Z + 1$, $P = \AA^{\dim Z}$, and $\eta \colon Q \to P$ has a section 
	$s \colon P \to Q$, then condition~\ref{thm.Product_d}
	is also satisfied. Indeed, in this case there exists a finite morphism $p \colon Z \to \AA^{\dim Z}$ 
	due to Noether's Normalization Theorem and one can choose $r \coloneqq s \circ p \colon Z \to Q$.
\end{remark}

\subsection{Transversality results}
This subsection essentially amounts to collecting and rephrasing
some material from~\cite{Ka2020Extensions-of-isom} 
that we need for the proof of Theorem~\ref{thm.Product}.

\begin{definition}
	\label{def.big_enough}
	Let $X \to P$ be a smooth morphism of smooth irreducible varieties and 
	let $\mathcal{H} = (H_1, \ldots, H_s)$ be a tuple of connected algebraic
	subgroups $H_1, \ldots, H_s \subset \Aut_P(X)$.
	 Then $\mathcal{H}$ is
	
	\begin{enumerate}[leftmargin=*]
		\item \emph{big enough for proper intersection}, if for 
		every morphism $f \colon Y \to X$ and every locally closed
		subvariety $Z$ in $X$ 
		there is an open subset $U \subset H_1 \times \cdots \times H_s$
		such that for every $(h_1, \ldots, h_s) \in  U$ we have
		\[
			\label{Eq.enough for proper intersection}
			\tag{PI}
			\dim Y \times_X h_1 \cdots h_s \cdot Z  \leq \dim Y \times_P Z + \dim P - \dim X \, .
		\]
		
		\item \emph{big enough for smoothness} if
		there exists an open dense
		subset $U \subset H_1 \times \cdots \times H_s$ such that the morphism
		\begin{align*}
			\Phi_{\mathcal{H}} \colon H_1 \times \cdots \times H_s \times X & \to X \times_P X \, , \\
			((h_1, \ldots, h_s), x) & \mapsto (h_1 \cdots h_s \cdot x, x)
		\end{align*}
		is smooth on $U \times X$.
	\end{enumerate}
\end{definition}

\begin{prop}
	\label{prop.Big_enough_for_smoothness}
	Let $X \to P$ be a smooth morphism of smooth irreducible varieties and
	let $\mathcal{H} = (H_1, \ldots, H_s)$ be a tuple of connected algebraic
	subgroups $H_1, \ldots, H_s$ in $\Aut_P(X)$. Then:
	
	\begin{enumerate}[leftmargin=*]
		\item  \label{prop.Big_enough_for_smoothness1} 
		If $\mathcal{H}$ is big enough for smoothness, 
		then $\mathcal{H}$ is big enough for proper intersection.
		\item \label{prop.Big_enough_for_smoothness2}
		If $\mathcal{H}$ is big enough for smoothness and 
		$H_0, H_{s+1} \subset \Aut_P(Y)$ are two connected algebraic subgroups, then
		$(H_0, H_1, \ldots, H_s, H_{s+1})$ is big enough for smoothness.
	\end{enumerate}
\end{prop}

\begin{proof}
	\eqref{prop.Big_enough_for_smoothness1}: The proof closely follows~\cite[Theorem~1.4]{Ka2020Extensions-of-isom}. By assumption, there is an open dense
	subset $U \subset H_1 \times \cdots \times H_s$ such that 
	$\Phi_{\mathcal{H}} |_{U \times X} \colon U \times X \to X \times_P X$ is smooth.
	Let $f \colon Y \to X$ be a morphism and let $Z$ be a locally closed subvariety of $X$.
	Let $W$ be the fiber product of $Y \times_P Z \to X \times_P X$ and $\Phi_\mathcal{H}|_{U \times X}$:
	\[
		\xymatrix@=10pt{
			W \ar[d] \ar[rrr] &&& Y \times_P Z \ar[d] \\
			U \times X \ar[rrr]^-{\Phi_{\mathcal{H}} |_{U \times X}} &&& X \times_P X \, .
		}
	\]
	By generic flatness~\cite[Theorem 10.84]{GoWe2010Algebraic-geometry}, we may shrink $U$ and
	assume that $\pi \colon W \to U \times X \to U$ is flat.
	Take $h = (h_1, \ldots, h_s) \in U$. Then
	\[
		\pi^{-1}(h)_{\red} \longrightarrow  (Y \times_X h_1 \cdots h_s \cdot Z)_{\red} \, , \quad
		((h, x), (y, z)) \to (y, h_1 \cdots h_s \cdot z)
	\]
	is an isomorphism, since $(h_1 \cdots h_s \cdot x, x) = (f(y), z)$ for each
	$((h, x), (y, z)) \in \pi^{-1}(h)_{\red}$. If $\pi^{-1}(h)$ is empty, then~\eqref{Eq.enough for proper intersection} from Definition~\ref{def.big_enough} is satisfied
	(as by convention $\dim \varnothing = -\infty$) 
	and thus we may assume that
	$\pi^{-1}(h)$ is non-empty and we get 
	$\dim \pi^{-1}(h) \leq \dim W - \dim U$ by the flatness of $\pi$.  By the smoothness of $\Phi_\mathcal{H} |_{U \times X}$
	and the pullback diagram above, $W \to Y \times_P Z$ is smooth 
	since smoothness is preserved under pullbacks. In particular, 
	$\dim W \leq \dim Y \times_P Z + \dim U \times X - \dim X \times_P X$. In total we get
	\begin{align*}
		\dim Y \times_X h_1 \cdots h_s \cdot Z 
		&= \dim \pi^{-1}(h) \\
		&\leq \dim Y \times_P Z + \dim X - \dim X \times_P X \\
		&= \dim Y \times_P Z + \dim P - \dim X \, ,
	\end{align*}
	since $\dim X \times_P X = 2 \dim X - \dim P$, which in turn 
	follows from the smoothness of $X \to P$ and the irreducibility of $X$, $P$.
	
	\eqref{prop.Big_enough_for_smoothness2}:
	This follows directly from~\cite[Remark 1.8]{Ka2020Extensions-of-isom}.
\end{proof}

\begin{prop}[{\cite[Proposition 1.7]{Ka2020Extensions-of-isom}}]
	\label{prop.transitivity_and_big_enough_for_smoothness}
	Let $\kappa \colon X \to P$ be a smooth morphism of smooth irreducible varieties 
	and let a subgroup $G \subset \Aut_P(X)$ be generated by a family
	$\mathcal{G}$ of connected algebraic subgroups of $\Aut_P(X)$ which is
	closed under conjugation by $G$. Moreover, assume that $G$ acts transitively on
	each fiber of $\kappa$.

	Then there exist $H_1, \ldots, H_s \in \mathcal{G}$ such that
	$(H_1, \ldots, H_s)$ is big enough for smoothness.\qed
\end{prop}

\subsection{Sufficiently transitive group actions}

For a variety $X$ we denote by $\SAut(X)$ the subgroup
of $\Aut(X)$ that is generated by all unipotent algebraic subgroups; in particular, 
$\SAut(X)\subseteq \Aut^{\alg}(G)$. 
Transitivity of the natural action of $\SAut(X)$ on $X$ 
implies $m$-transitivity for all $m$ and 
that one can prescribe the tangent map of 
an automorphism of $X$ at a finite number of fixed points:

\begin{theorem}[{\cite[Theorem~0.1, Theorem~4.14 and Remark~4.16]{ArFlKa2013Flexible-varieties}}]
	\label{thm.transitivity_and_describing_jets}
	Let $X$ be an irreducible smooth affine variety of dimension at least $2$.
	If $\SAut(X)$ acts transitively on $X$, then:
	\begin{enumerate}[leftmargin=*]
		\item \label{thm.flexible1} $\SAut(X)$ acts $m$-transitively on $X$ for each $m \geq 1$;
		\item \label{thm.flexible2} for every finite subset $Z \subset X$ and
		every collection $\beta_z \in \SL(T_z X)$, $z \in Z$, there is an automorphism
		$\varphi \in \SAut(X)$ that fixes $Z$ pointwise such that
		the differential satisfies $\textrm{d}_z \varphi = \beta_z$ for all $z \in Z$.\qed
	\end{enumerate}
	\end{theorem}

\begin{example}
	\label{exa.Enough_transitively}
	Let $F$ be an irreducible smooth affine variety of dimension $\geq 2$ such that
	$\SAut(F)$ acts transitively on it. Then,
	Theorem~\ref{thm.transitivity_and_describing_jets} implies that 
	$\SAut(F)$ acts sufficiently transitively on $F$; see Definition~\ref{def.sufficiently_transitively}.
\end{example}

\begin{example}
	\label{exa.Enough_transitivity_alg_group}
	If $G$ is a connected 
	characterless algebraic group, then the group 
	$\Aut^{\alg}(G)$ acts sufficiently transitively
	on $G$. Indeed, such a $G$ is generated by its unipotent subgroups (see e.g.~\cite[Lemma~1.1]{Po2011On-the-Makar-Liman}) and thus $G \subseteq \SAut(G)$. In particular,
	$\SAut(G)$ acts transitively on $G$. 
	Now, if $\dim G = 0$, then the statement is trivial. If $\dim G = 1$,
	then $G$ is isomorphic to $\GG_a$ and thus $\Aut^{\alg}(G) = \Aut(\AA^1)$; 
	hence, the statement is also clear. If $\dim G \geq 2$,
	then the statement follows from Example~\ref{exa.Enough_transitively}.
\end{example}

Incidentally, the above example characterizes algebraic groups $G$ such that 
$\Aut^{\alg}(G)$ acts sufficiently transitively on $G$:

\begin{prop}
	\label{prop.characterization_suff_trans_algebraic_groups}
	Let $G$ be an algebraic group. Then $\Aut^{\alg}(G)$ acts sufficiently transitively on $G$
	if and only if $G$ is connected and characterless.
\end{prop}

\begin{proof}
	According to Example~\ref{exa.Enough_transitivity_alg_group} we only have to show the `only if'-part.
	
	Let $G$ be an algebraic group such that $\Aut^{\alg}(G)$ acts sufficiently transitively on it.
	
	Let $g, g' \in G$. Since $\Aut^{\alg}(G)$ acts transitively on $G$, there exist connected
	algebraic subgroups $H_1, \ldots, H_r$ in $\Aut(G)$ such that $g'$ lies
	in the image of the morphism
	\[
		H_1 \times \cdots \times H_r \to G \, , \quad (h_1, \ldots, h_r) \mapsto (h_1 \circ \ldots \circ h_r)(g) \, .
	\]
	Hence, $g, g' \in G$ lie in an irreducible closed subset of $G$. Since $g, g'$ were arbitrary elements
	of $G$, it follows that $G$ is connected.
	
	Denote by $G^u$ the algebraic subgroup of $G$ that is generated by all unipotent elements
	in $G$. Then $G^u$ is closed and normal in $G$ and each invertible function on $G^u$
	is constant.
 	There exists an algebraic torus $T \subseteq G$ (i.e.~$T$ is a product of finitely many copies of the underlying multiplicative group of the ground field) such that $G = G^u \rtimes T$;
 	see~e.g.~\cite[Lemma 8.2]{FeSa2019Uniqueness-of-embe}.
	
	Let $\pi \colon G \to T$ be the canonical projection. Take an arbitrary 
	algebraic action $\rho \colon H \times G \to G$ of an arbitrary connected algebraic group $H$.
	Since each invertible function on each fiber of $\pi$ is constant, the morphism
	\[
		H \times G \stackrel{\rho}{\to} G \stackrel{\pi}{\to} T
	\]
	is invariant under the algebraic action $N \times (H \times G) \to H \times G$ that is given
	by $n \cdot (h, g) = (h, n g)$. Hence, the morphism $\pi \circ \rho$
	factors through $\id_H \times \pi$, i.e.~there is a commutative diagram
	\begin{equation}
		\label{Eq.Induced_action}
		\begin{gathered}
		\xymatrix{
			H \times G \ar[r]^-{\rho} \ar[d]_-{\id_H \times \pi} & G \ar[d]^-{\pi} \\
			H \times T \ar[r]^-{\rho_T} & T
		}
		\end{gathered}
	\end{equation}
	for a unique morphism $\rho_T \colon H \times T \to T$. As $\rho$ is an action,
	$\rho_T$ is an action as well. Since $\Aut^{\alg}(G)$ acts
	$2$-transitively on $G$ and since each action $\rho$ of a connected algebraic group 
	on $G$ induces an action $\rho_T$ on 
	$T$ such that~\eqref{Eq.Induced_action} commutes, 
	we get that $\Aut^{\alg}(T)$ acts $2$-transitively on $T$.
	By Lemma~\ref{lem.torus-alg_autos} below, we find that $T$ is trivial, and
	thus $G = G^u$ is characterless.
\end{proof}

The following lemma is certainly well-known to the specialists. However, for lack of a reference
we give a proof of it.

\begin{lemma}
	\label{lem.torus-alg_autos} 
	Let $T$ be an algebraic torus. Then 
	\[
		\Aut^{\alg}(T) = \set{ T \to T \, , \, \,  t \mapsto s t}{s \in T} \, .
	\]
\end{lemma}

\begin{proof}
	Let $H \subset \Aut(T)$ be an algebraic subgroup. Hence there exists a faithful algebraic
	$H$-action $\rho \colon H \times T \to T$ such that the image of the induced homomorphism
	in $\Aut(T)$ is $H$ (see Remark~\ref{rem.alg_subgroup}). 
	By~\cite[Theorem~2]{Ro1961Toroidal-algebraic} there exist 
	morphisms $\mu \colon H \to T$ and $\lambda \colon T \to T$ such that
	\[
		\rho(h, t) = \mu(h)\lambda(t) \quad
		\textrm{for each $h \in H$, $t \in T$} \, .
	\]
	After replacing $\mu$ and $\lambda$ by $t_0 \mu$ and $t_0^{-1} \lambda$, respectively, 
	for some
	$t_0 \in T$, we may assume
	that $\mu(e_H) = e_T$, where $e_H$ and $e_T$ denote the neutral elements of $H$ and $T$,
	respectively. Hence, $t = \rho(e_H, t) = \lambda(t)$ for each $t \in T$, and thus
	\[
		\rho(h, t) = \mu(h)t \quad \textrm{for each $h \in H$, $t \in T$} \, .
	\]
	This implies that $H$ lies inside $\set{ T \to T \, , \, \,  t \mapsto st}{s \in T}$, and thus the lemma follows.
\end{proof}

\subsection{Sufficiently transitive group actions on fibers}

In the next proposition, we provide a class of smooth morphisms $\pi \colon X \to P$
such that $\Aut_{P}^{\alg}(X)$ acts sufficiently transitively on each fiber of $\pi$.

\begin{prop}
	\label{prop.sufficiently_transitive_on_fibers_quotient}
	Let $G$ be a connected 
	algebraic group and $H \subseteq G$ be a connected characterless algebraic subgroup of
	dimension $\geq 2$. Then,
	the algebraic quotient $\pi \colon G \to G/H \eqqcolon P$ 
	is a smooth morphism such that $\Aut_{P}^{\alg}(G)$ acts
	sufficiently transitively on each fiber of $\pi$.
\end{prop}

We note that the dimension condition $\dim H \geq 2$ in Proposition~\ref{prop.sufficiently_transitive_on_fibers_quotient} is necessary, as the following example shows.

\begin{example}\label{exa.sufficiently_transitive_on_fibers_quotient:conterexampledim1}
	Denote by $\pi \colon \SL_2 \to P \coloneqq \SL_2 / H$ the 
	algebraic quotient, where $H \subset \SL_2$ denotes
	the subgroup of unipotent upper triangular matrices. In this case each automorphism $\varphi$
	in $\Aut_{P}(\SL_2)$ acts as a translation on $\pi^{-1}(p) \simeq \AA^1$ for each $p \in P$. 
	In particular, for each $p\in P$ we have that $\Aut_{P}^{\alg}(\SL_2)$ does not act sufficiently transitively
	on $\pi^{-1}(p)$ (while the group 
	$\Aut^{\alg}(\pi^{-1}(p))$ acts sufficiently transitively on $\pi^{-1}(p)$ by
	Example~\ref{exa.Enough_transitivity_alg_group}).
	
	That $\varphi$ acts as a translation on each fiber of $\pi$ can be checked explicitly 
	by writing $\varphi$ with respect to the following para\-metrizations
	\[
	\AA^1 \setminus \{0\} \times \AA^1 \times \AA^1 \to \SL_2 \, , \quad (x, z, y) \mapsto
	\begin{pmatrix}
	x & y \\
	z & \frac{yz+1}{x}
	\end{pmatrix}
	\]
	and
	\[
	\AA^1 \times \AA^1 \setminus \{0\} \times \AA^1 \to \SL_2 \, , \quad (x, z, w) \mapsto
	\begin{pmatrix}
	x & \frac{xw-1}{z} \\
	z & w
	\end{pmatrix} \, .
	\]
	
\end{example}

For the proof of Proposition~\ref{prop.sufficiently_transitive_on_fibers_quotient}, we need some preparation. First, we recall a more general version of
Theorem~\ref{thm.transitivity_and_describing_jets} stated in terms of the following definition.

\begin{definition}[{\cite[Definition~2.1]{ArFlKa2013Flexible-varieties}}]
	\label{def.saturated}
	Let $X$ be an affine variety and let $\mathcal{N}$ be a set of locally nilpotent derivations
	on the coordinate ring
	$\OO(X)$ and let $G(\mathcal{N})$ be the subgroup of $\SAut(X)$ that is generated
	by all automorphisms of $X$ that are induced by the locally nilpotent derivations in $\mathcal{N}$.
	Then $\mathcal{N}$ is called saturated, if
	\begin{enumerate}[leftmargin=*, label=(\roman*)]
		\item \label{Def.Saturated_1} $\mathcal{N}$ is closed under conjugation by elements from $G(\mathcal{N})$ and
		\item \label{Def.Saturated_2} for each
		$D \in \mathcal{N}$ and each $f \in \ker(D)$ we have $f D \in \mathcal{N}$.
	\end{enumerate}
\end{definition}

\begin{remark}
	\label{rem.saturated_set_LNDs}
	If $X$ is an affine variety and if $\mathcal{N}$ is a set of locally nilpotent derivations on $\OO(X)$
	that satisfies~\ref{Def.Saturated_2} from Definition~\ref{def.saturated}, then there exists
	a bigger set $\mathcal{N}'$ of locally nilpotent derivations on $\OO(X)$ that is saturated and
	satisfies $G(\mathcal{N}') = G(\mathcal{N})$; see~\cite[Lemma~4.6]{FlKaZa2017Cancellation-for-s}.
\end{remark}

We come now to the promised generalization of Theorem~\ref{thm.transitivity_and_describing_jets}.

\begin{theorem}[{\cite[Theorem~2.2, Theorem~4.14 and Remark~4.16]{ArFlKa2013Flexible-varieties}}]
	\label{thm.transitivity_and_describing_jets_general}
	Let $X$ be an irreducible smooth affine variety of dimension at least $2$
	and let $\mathcal{N}$ be a saturated set of locally nilpotent derivations on $\OO(X)$.
	If the subgroup $G(\mathcal{N})$ of $\SAut(X)$ acts transitively on $X$, then:
	\begin{enumerate}[leftmargin=*]
		\item \label{thm.flexible1} $G(\mathcal{N})$ acts $m$-transitively on $X$ for each $m \geq 1$;
		\item \label{thm.flexible2} for every finite subset $Z \subset X$ and
		every collection $\beta_z \in \SL(T_z X)$, $z \in Z$, there is an automorphism
		$\varphi \in G(\mathcal{N})$ that fixes $Z$ pointwise such that
		the differential satisfies $\textrm{d}_z \varphi = \beta_z$ for all $z \in Z$.\qed
	\end{enumerate}
\end{theorem}

\begin{lemma}
	\label{lem.Surjectivity_on_invariant_rings}
	Let $G$ be an algebraic groups and let $U \subseteq H \subseteq G$ be closed subgroups
	such that $H$ is characterless and $U$ is unipotent. 
	Then, the restriction map $\OO(G) \to \OO(H)$, $q \mapsto q|_H$
	induces a surjection on the $U$-invariant rings $\OO(G)^U \to \OO(H)^U$ where the $U$-actions
	are induced by right multiplication.
\end{lemma}

The following example shows, that the assumption that $H$ is characterless is necessary:

\begin{example}
	Let $G = \SL_2$, $H$ the subgroup of upper triangular matrices and let
	$U \subseteq H$ be the subgroup with $1$ on the diagonal. Denote the coordinates
	on $\SL_2$ by
	\[
		\begin{pmatrix}
			x & y \\
			z & w
		\end{pmatrix}
	\]
	Then $\OO(G) \to \OO(H)$ identifies with the homomorphism
	\[
		\kk[x, y, z, w] / (xw-yz-1)
		\xrightarrow{x \mapsto x \, , \ 
			y \mapsto y \, , \
			z \mapsto 0 \, , \
			w \mapsto w}
		\kk[x, y, w] / (xw - 1)
	\]
	and thus $\OO(G)^U \to \OO(H)^U$ identifies with  the non-surjective 
	homomorphism
	\[
		\kk[x, z]
		\xrightarrow{
			x \mapsto x \, , \ 
			z \mapsto 0}
		\kk[x, w] / (xw - 1) \simeq \kk[x, x^{-1}]  \, .
	\]
	
\end{example}

We first provide the proof of Proposition~\ref{prop.sufficiently_transitive_on_fibers_quotient} using Lemma~\ref{lem.Surjectivity_on_invariant_rings}. Afterwards, we provide the setup and the proof of Lemma~\ref{lem.Surjectivity_on_invariant_rings}.

\begin{proof}[Proof of Proposition~\ref{prop.sufficiently_transitive_on_fibers_quotient}]	
	We have to show that $\Aut^{\alg}_{G/H}(G)$ acts sufficiently transitively on each fiber of
	$\pi \colon G \to G/H$. Since $\pi$ is $G$-equivariant with respect to left multiplication by $G$,
	it is enough to show that $\Aut^{\alg}_{G/H}(G)$ acts sufficiently transitively on the closed subset 
	$H$ of $G$. Let
	\[
		\mathcal{N} = \Bigset{fD}{
									\begin{array}{l}
										\textrm{$D$ is a locally nilpotent derivation of $\OO(H)$ induced} \\
										\textrm{by right multiplication of a one-dimensional} \\ 
										\textrm{unipotent subgroup of $H$ and $f \in \ker(D)$}
									\end{array}
								} \, .
	\]
	Since $H$ is connected and characterless, $H$ is spanned by 
	all its one-dimen\-sion\-al unipotent subgroups; 
	see~\cite[Lemma~1.1]{Po2011On-the-Makar-Liman} Hence, the subgroup $G(\mathcal{N})$ 
	of $\SAut(H)$ generated by $\mathcal{N}$ acts transitively on $H$. By
	Remark~\ref{rem.saturated_set_LNDs}, there exists a saturated set of locally nilpotent derivations $\mathcal{N'}$ with $G(\mathcal{N})=G(\mathcal{N}')$; hence, Theorem~\ref{thm.transitivity_and_describing_jets_general}
	implies that $G(\mathcal{N})$ acts sufficiently transitively on $H$.
	Therefore, it suffices to show that every element of $G(\mathcal{N})$ can be extended to an 
	automorphism in $\Aut^{\alg}_{G/H}(G)$.
	
	Let $D \in \mathcal{N}$, $f \in \ker(D)$ and denote by
	$U \subseteq H$ the corresponding one-dimensional unipotent subgroup. 
	Note that $\ker(D)$ is equal to the invariant ring $\OO(H)^{U}$.
	By Lemma~\ref{lem.Surjectivity_on_invariant_rings}, there exists $q \in \OO(G)^U$
	such that  $q|_H = f$.
	
	Denote by $E$ the locally nilpotent derivation of $\OO(G)$
	induced by right multiplication of $U$ on $G$. Since $q \in \OO(G)^U = \ker(E)$,
	the derivation $q E$ of $\OO(G)$ is locally nilpotent. Since $U$ is a subgroup of $H$,
	the fibers $gH$, $g \in G$ of $\pi$ are stable under the induced $\GG_a$-action 
	of $q E$ and thus this $\GG_a$-action gives an algebraic subgroup of $\Aut^{\alg}_{G/H}(G)$.
	Note that we have a commutative diagram of the following form
	\[
	\xymatrix{
		\OO(G) \ar[r]^-{qE} \ar@{->>}[d] & \OO(G) \ar@{->>}[d] \\
		\OO(H) \ar[r]^-{fD} & \OO(H)
	}
	\] 
	where the vertical arrows are induced by the embedding $H \subseteq G$. Therefore
	we found our desired extension of the $\GG_a$-action induced by $fD$.
\end{proof}

\begin{proof}[Proof of Lemma~\ref{lem.Surjectivity_on_invariant_rings}]
	Since $H$ and $U$ are characterless, the quotients $G/U$, $H/U$ and $G/H$ are quasi-affine;
	see \cite[Example 3.10]{Ti2011Homogeneous-spaces}. Hence, the canonical morphisms
	\begin{itemize}[leftmargin=*]
		\item $\iota_{G/U} \colon G/U \to (G/U)_{\aff} \coloneqq \Spec(\OO(G)^U)$
		\item $\iota_{H/U} \colon H/U \to (H/U)_{\aff} \coloneqq \Spec(\OO(H)^U)$
		\item $\iota_{G/H} \colon G/H \to (G/H)_{\aff} \coloneqq \Spec(\OO(G)^H)$
	\end{itemize}
	are dominant open immersions;
	see \cite[\S5, Proposition 5.1.2]{Gr1961Elements-de-geomet-II}. 
	The targets of these open immersions are affine schemes that are, in general, 
	not of finite type over $\kk$.
	There are unique $G$-actions
	on $(G/U)_{\aff}$ and $(G/H)_{\aff}$ such that $\iota_{G/U}$ and $\iota_{G/H}$ are $G$-equivariant
	and a unique $H$-action on $(H/U)_{\aff}$ such that $\iota_{H/U}$ is $H$-equivariant; see \cite[Lemma~5]{KrReSa2019Is-the-Affine-Spac}. Moreover, the canonical $G$-equivariant morphism
	$\rho \colon G/U \to G/H$ induces a unique $G$-equivariant morphism 
	\[
		\rho_{\aff} \colon (G/U)_{\aff} \to (G/H)_{\aff}
	\]
	such that the following diagram commutes
	\[
		\xymatrix{
			G/U \ar[r]^-{\iota_{G/U}} \ar[d]^-{\rho} & (G/U)_{\aff} \ar[d]^-{\rho_{\aff}} \\
			G/H \ar[r]^-{\iota_{G/H}} & (G/H)_{\aff} \, .
		}
	\]
	
	Let $V \subseteq G/H$ be an open affine neighbourhood 
	of $q \coloneqq H \in G/H$. We may assume that there is an
	$s \in \OO(G/H) = \OO(G)^H$ 
	such that $V = (G/H)_s$, i.e.~$V$ consists
	of all points in $G/H$ where $s$ does not vanish. Further we may assume that the extension
	$s_{{\aff}} \colon (G/H)_{\aff} \to \AA^1$ of $s \colon G/H \to \AA^1$ via $\iota_{G/H}$
	vanishes on the complement of $\iota_{G/H}(G/H)$
	in $(G/H)_{\aff}$. Hence, $\iota_{G/H}(V) = ((G/H)_{\aff})_{s_{{\aff}}}$ and therefore
	\[
		\iota_{G/U}(\rho^{-1}(V)) \subseteq \rho_{\aff}^{-1}(\iota_{G/H}(V))
		= \rho_{\aff}^{-1}(((G/H)_{\aff})_{s_{{\aff}}}) = ((G/U)_{\aff})_{s_{{\aff}} \circ \rho_{\aff}} \, .
	\]
	By \cite[\href{https://stacks.math.columbia.edu/tag/01P7}{Lemma 01P7}]{stacks-project} we have $(\OO(G/U))_{s \circ \rho} = \OO((G/U)_{s \circ \rho})$ and thus
	\[
		((G/U)_{\aff})_{s_{{\aff}} \circ \rho_{\aff}} = ((G/U)_{s \circ \rho})_{\aff} = (\rho^{-1}(V))_{\aff} \, .
	\]
	Hence, we have the following commutative diagram
	\begin{equation}
	\label{eq.rho_inverse_of_V_aff}
	\tag{$\triangle$}
	\begin{gathered}
	\xymatrix{			
		\rho^{-1}(V) \ar@{=}[r]& (G/U)_{s \circ \rho} \ar[d]_-{\textrm{open}} \ar[r] & 
		((G/U)_{\aff})_{s_{{\aff}} \circ \rho_{\aff}} \ar[d]^-{\textrm{open}}
		\ar@{=}[r]& (\rho^{-1}(V))_{\aff} \\
		& G/U \ar[r]^-{\iota_{G/U}} & (G/U)_{\aff} \, .
	}
	\end{gathered}
	\end{equation}
	Furthermore, we may shrink $V$ such that there exists a finite Galois covering $\tau \colon V' \to V$
	for some finite group $\Gamma$ (i.e.~$\tau$ is a geometric
	quotient for a free $\Gamma$-action on $V'$) such that the pull-back map 
	$\rho'$ in the following pull-back diagram
	\[
	\xymatrix{
		V' \times_V \rho^{-1}(V) \ar[r]^-{\tau'} \ar[d]^-{\rho'} & \rho^{-1}(V) \ar[d]^-{\rho|_{\rho^{-1}(V)}} \\
		V' \ar[r]^-{\tau} & V
	}
	\]
	is a trivial $H/U$-bundle;
	see~\cite[\S1.5 and Proposition~3]{Se1958Espaces-fibres-alg}. In particular, 
	there exists an isomorphism 
	$\varphi \colon V' \times (H/U) \to V' \times_V \rho^{-1}(V)$ such that 
	$\rho' \circ \varphi \colon V' \times (H/U) \to V'$ is the projection onto the first factor. 
	As $\tau \colon V' \to V$ is finite and $V$ is affine, $V'$ is affine as well.
	Note further, that the $\Gamma$-action on $V'$ induces a natural free $\Gamma$-action
	on $V' \times_V \rho^{-1}(V)$ such that $\rho'$ is $\Gamma$-equivariant
	and $\tau'$ is a geometric quotient for this $\Gamma$-action.
	Choose $q' \in V'$ such that $\tau(q') = q$.
	
	Let $f \in \OO(H/U) = \OO(H)^U$. The goal is to extend $f$ to an element in $\OO(G)^U$.
	Consider the morphism
	\[
		f' \colon \Gamma q' \times (H/U) \xrightarrow{\varphi |_{\Gamma q' \times (H/U)}} 
		(\rho')^{-1}(\Gamma q') 
		\xrightarrow{\tau'|_{(\rho')^{-1}(\Gamma q')}} \rho^{-1}(q) = H/U \xrightarrow{f} \AA^1 \, . 
	\]
	Then the extension $f'_{\aff} \colon \Gamma q' \times (H/U)_{\aff} = (\Gamma q' \times (H/U))_{\aff} \to \AA^1$ of $f'$ can be extended to a morphism 
	\[
		\label{eq.Extension}
		\tag{$\ast$}
		V' \times (H/U)_{\aff} \to \AA^1 \, ,
	\]
	as $\Gamma q' \times (H/U)_{\aff}$ is a closed subscheme in the affine scheme $V' \times (H/U)_{\aff}$.
	Let $F' \colon V' \times (H/U) \to \AA^1$ be the composition of
	$\id_{V'} \times \iota_{H/U}$ with the morphism~\eqref{eq.Extension}.
	By construction we have that
	$F' |_{\Gamma q' \times (H/U)} = f'$. Now, let
	\[
		G' \coloneqq  \frac{1}{|\Gamma|}\sum_{\gamma \in \Gamma} \gamma \cdot (F' \circ \varphi^{-1})
		\colon V' \times_V \rho^{-1}(V) \to \AA^1
	\]
	be the average of $F' \circ \varphi^{-1}$ over $\Gamma$. Since 
	$f' \circ (\varphi|_{(\rho')^{-1}(\Gamma q')})^{-1} = f \circ \tau'|_{(\rho')^{-1}(\Gamma q')}$ is $\Gamma$-invariant, it follows that
	\[
		\label{eq.Extension2}
		\tag{$\ast\ast$}
		f \circ \tau'|_{(\rho')^{-1}(\Gamma q')} = G' |_{(\rho')^{-1}(\Gamma q')} \, .
	\]
	Since $G'$ is $\Gamma$-invariant and since $\tau'$ is a geometric quotient for the 
	$\Gamma$-action on $V' \times_V \rho^{-1}(V)$, 
	there exists a morphism $F \colon \rho^{-1}(V) \to \AA^1$
	such that $G' =  F \circ \tau'$. {Using~\eqref{eq.Extension2}, we find} $F |_{\rho^{-1}(q)} = f$.
	{The commutative diagramm~\eqref{eq.rho_inverse_of_V_aff} implies that} $F$ extends to a morphism
	$F_{\aff} \colon (\rho^{-1}(V))_{\aff} = \rho_{\aff}^{-1}(\iota_{G/H}(V))  \to \AA^1$ via $\iota_{G/U}$, 
	i.e.~
	\[
		F_{\aff}(\iota_{G/U}(gU)) = F(gU) \quad  \textrm{for all $gU \in \rho^{-1}(V)$} \, .
	\]
	Hence the restriction $F_{\aff} |_{\rho_{\aff}^{-1}(H)} \colon  \rho_{\aff}^{-1}(H) \to \AA^1$ satisfies
	\[
		F_{\aff}(\iota_{G/U}(hU)) = F(hU) = f(hU) \quad \textrm{for all $h \in H$} \, .
	\]
	Since $\rho_{\aff}^{-1}(H)$ is a closed subscheme of the affine scheme $(G/U)_{\aff}$,
	there exists an extension of $F_{\aff} |_{\rho_{\aff}^{-1}(H)}$ to a morphism $(G/U)_{\aff} \to \AA^1$.
	This is our desired element in $\OO(G)^U$.
\end{proof}

\subsection{The proof of Theorem~\ref{thm.Product}}
\label{subsec:Proof_embedd_into_tot_space_principal_bundle}
Throughout  this subsection we use the following notation.

\medskip

\textbf{Notation.} Let $f \colon X \to Z$ be a morphism of varieties, 
then we denote by $X_Z^{(2)}$
the complement of the diagonal in the fiber product $X \times_Z X$ and we denote by
$(\ker \textrm{d} f)^\circ$ the complement of the zero section in the kernel of the differential
$\textrm{d}f \colon TX \to TZ$.

\medskip

We start with the following rather technical result that will turn out to be the key.

\begin{prop}
	\label{prop.key}
	Let $\pi \colon X \to P$,  $\rho \colon X \to Q$ be smooth morphisms
	of smooth irreducible varieties such that
	there exists a morphism $\eta \colon Q \to P$ with $\pi = \eta \circ \rho$. 
	Assume that $\Aut_P^{\alg}(X)$ acts sufficiently transitively on each fiber of $\pi$.
	
	If $Z$ is a smooth variety and $f \colon Z \to X$ is a morphism such that each non-empty fiber of $\pi \circ f \colon Z \to P$ has the same dimension $k \geq 0$, then
	there exists a $\varphi \in \Aut_P^{\alg}(X)$ with
	\begin{align*}
		\label{enough.inj} 
		\tag{A}
		\dim((\varphi \circ f) \times (\varphi \circ f))^{-1}(X_Q^{(2)}) &\leq \dim Z + \dim P - \dim Q + k \\
		\label{enough.imm}
		\tag{B}
		\dim (\textrm{d} (\varphi \circ f))^{-1}(\ker \textrm{d} \rho)^\circ 
		&\leq \dim Z + \dim P - \dim Q + k \, .
	\end{align*}
\end{prop}

For the proof of the estimate~\eqref{enough.imm} in this key proposition, we need the following estimate:

\begin{lemma}
	\label{lem.estimate_kernel_differential}
	Let $f \colon X \to Y$ be a morphism of varieties
	such that $X$ is smooth and denote by $k$ the maximal dimension among the fibers of  $f$. 
	
	Then the kernel of
	the differential $\textrm{d} f \colon TX \to TY$, i.e.~the closed subvariety 
	\[
		\ker(\textrm{d} f)\coloneqq \bigcup_{x\in X}\ker(\textrm{d}_x f)\subseteq TX \, ,
	\] 
	satisfies 
	$\dim \ker(\textrm{d} f) \leq \dim X + k$.
\end{lemma}

\begin{proof}[Proof of Lemma~\ref{lem.estimate_kernel_differential}]
	Let $X = \bigcup_{i=1}^n X_i$ be a partition into smooth, irreducible, locally 
	closed subvarieties $X_1, \ldots, X_n$ in $X$ such that 
	\[
		f_i \coloneqq f|_{X_i} \colon X_i \to \overline{f(X_i)}
	\]
	is smooth for each $i=1,\ldots, n$ (see~\cite[Lemma~10.5, Ch.~III]{Ha1977Algebraic-geometry}).
	Note that $f(X_i)$ is an open subvariety of $\overline{f(X_i)}$ that is smooth,
	see~\cite[Proposition 3.1, Expos\'e II]{GrRa2003Revetements-etales}. Let $x \in X_i$.
	Thus the differential $\textrm{d}_x f_i \colon T_x X_i \to T_{f(x)} \overline{f(X_i)}$ is surjective
	and since $\dim f_i^{-1}(x) \leq k$, we get $\dim \ker (\textrm{d}_{x} f_i) \leq k$. 
	Then the kernel of 
	\[
		T_x X_i \hookrightarrow T_x X \stackrel{\textrm{d}_x f}{\longrightarrow} T_{f(x)} Y
	\]
	has dimension $\leq k$, which implies
	$\dim \ker(\textrm{d}_x f) \leq \dim T_x X - \dim T_x X_i + k$.
	Since $X$ is smooth, we have $\dim T_x X \leq \dim X$ (we did not assumed that $X$ is 
	equidimensional, hence we do not necessarily have an equality) and since $X_i$ is smooth and irreducible, we have $\dim T_x X_i = \dim X_i$. Thus we get
	\begin{align*}
		\dim \ker( \textrm{d} f)|_{X_i} &\leq \dim X_i + \max_{x \in X_i} \dim\ker(\textrm{d}_x f) \\
		&\leq \dim X_i + \dim X - \dim X_i + k = \dim X + k \, .
	\end{align*}
	Hence, $\dim \ker(\textrm{d} f) \leq \max_{1 \leq i \leq n} 
	\dim (\ker \textrm{d}_x f) \cap TX |_{X_i} \leq \dim X + k$.
\end{proof}

\begin{proof}[Proof of Proposition~\ref{prop.key}]	
	Let $G\coloneqq \Aut_P^{\alg}(X)$ and let $\mathcal{G}$ be the family of all
	connected algebraic subgroups of $\Aut(X)$ that lie in $\Aut_P(X)$. By definition 
	$G$ is generated by the subgroups inside $\mathcal{G}$ and $\mathcal{G}$ is
	closed under conjugation by elements of $G$.

	Since $\pi \colon X \to P$ is smooth and $G$ acts sufficiently transitively
	on each fiber of $\pi$, the morphisms
	\[
		\kappa \colon X_P^{(2)} \to P \, , \quad (x, x') \mapsto \pi(x) 
	\]
	and
	\[
		\kappa' \colon  (\ker \textrm{d} \pi )^\circ \to X \stackrel{\pi}{\longrightarrow} P
	\]
	are smooth and $G$ acts transitively on each fiber of
	$\kappa$ and $\kappa'$.
	
	Applying Proposition~\ref{prop.transitivity_and_big_enough_for_smoothness} to
	$\kappa$ and the image of $G$ in $\Aut_P(X_P^{(2)})$
	under $\varphi \mapsto \varphi \times_P \varphi$ gives 
	$H_1, \ldots, H_s \in \mathcal{G}$ such that
	$\mathcal{H} = (H_1, \ldots, H_s)$ is big enough for smoothness with respect to 
	$\kappa$. Likewise one gets
	$H'_1, \ldots, H'_{s'} \in \mathcal{G}$ such that
	$\mathcal{H'} = (H'_1, \ldots, H'_{s'})$ is big enough for smoothness with respect to 
	$\kappa'$. Using Proposition~\ref{prop.Big_enough_for_smoothness}\eqref{prop.Big_enough_for_smoothness2}, 
	$\mathcal{M} = (H_1, \ldots, H_s, H'_1, \ldots, H'_{s'})$ is big enough for smoothness with respect to 
	$\kappa$ and $\kappa'$. By
	Proposition~\ref{prop.Big_enough_for_smoothness}\eqref{prop.Big_enough_for_smoothness1},
	$\mathcal{M}$ is also big enough for proper intersection with respect 
	to $\kappa$ and $\kappa'$. Hence, there is an open dense subset 
	$U \subset H_1 \times \cdots \times H_s \times H'_1 \times \cdots \times H'_{s'}$ such that
	for each element in $U$ the estimate~\eqref{Eq.enough for proper intersection} in Definition~\ref{def.big_enough}
	is satisfied with respect to 
	\begin{itemize}
		\item the smooth morphism $\kappa \colon X_P^{(2)} \to P$,
		\item the morphism 
				$(f \times f)|_{(f \times f)^{-1}(X_P^{(2)})} \colon (f \times f)^{-1}(X_P^{(2)}) \to X_P^{(2)}$ and
		\item the closed subset $X_Q^{(2)}$ in $X_P^{(2)}$
	\end{itemize}
	and
	\begin{itemize}
		\item the smooth morphism $\kappa' \colon (\ker \textrm{d} \pi)^\circ \to P$,
		\item the morphism 
				$\textrm{d} f|_{(\textrm{d} f)^{-1}(\ker \textrm{d} \pi)^\circ} \colon (\textrm{d} f)^{-1}(\ker \textrm{d} \pi)^\circ \to (\ker \textrm{d} \pi)^\circ$ and
		\item the closed subset $(\ker \textrm{d} \rho)^\circ$ in $(\ker \textrm{d} \pi)^\circ$.
	\end{itemize}
	That means that, if we choose an element $(h_1, \ldots, h_s, h'_1, \ldots, h_{s'}') \in U$, then the automorphism 
	$\varphi = (h_1 \cdots h_s \cdot h_1' \cdots h_{s'}')^{-1} \in G$ satisfies the following
	estimates:
	\begin{align*}
			&\dim ((\varphi \circ f) \times (\varphi \circ f))^{-1}(X_Q^{(2)}) \\
			&= \dim ( f \times f)^{-1}(X_P^{(2)}) \times_{X_P^{(2)}} (\varphi \times \varphi)^{-1}(X_Q^{(2)}) \\ 
			&\overset{\eqref{Eq.enough for proper intersection}}{\leq}  \dim (f \times f)^{-1}(X_P^{(2)}) \times_P X_Q^{(2)} + \dim P - \dim X_P^{(2)}
	\end{align*}
	and
	\begin{align*}
			& \dim (\textrm{d} (\varphi \circ f))^{-1}(\ker \textrm{d} \rho)^\circ \\
			&= \dim (\textrm{d} f)^{-1}(\ker \textrm{d} \pi)^\circ \times_{(\ker \textrm{d} \pi)^\circ} 
			 (\textrm{d} \varphi)^{-1}(\ker \textrm{d} \rho)^\circ  \\
			&\overset{\eqref{Eq.enough for proper intersection}}{\leq} \dim (\textrm{d} f)^{-1}(\ker \textrm{d} \pi)^\circ \times_P (\ker \textrm{d} \rho)^\circ
			+ \dim P - \dim (\ker \textrm{d} \pi)^\circ \, .
	\end{align*}
	
	Since $\pi \colon X \to P$ and $\kappa \colon X \to Q$ are both smooth morphisms
	of smooth irreducible varieties, we get
	\begin{itemize}
		\item $\dim X_P^{(2)} = 2 \dim X - \dim P$
		\item $\dim X_Q^{(2)} = 2 \dim X - \dim Q$
		\item $\dim (\ker \textrm{d} \pi)^\circ = 2 \dim X - \dim P$.
	\end{itemize}
	Hence, it is enough to show the following estimates:
	\begin{enumerate}
		\item  \label{Eq.estimate1} 
		$\dim(f \times f)^{-1}(X_P^{(2)}) \times_P X_Q^{(2)} 
		\leq 2 \dim X + \dim Z - \dim Q - \dim P + k$
		\item \label{Eq.estimate2}
		$\dim(\textrm{d} f)^{-1}(\ker \textrm{d} \pi)^\circ \times_P (\ker \textrm{d} \rho)^\circ \leq 
		2 \dim X + \dim Z - \dim Q - \dim P + k$
	\end{enumerate}

	We establish~\eqref{Eq.estimate1}:
	Consider the following pull-back diagram
	\begin{equation}
	\begin{gathered}
	\label{eq.diagram_double_points}
	\xymatrix@=10pt{
	    (f \times f)^{-1}(X_P^{(2)}) \times_P X_Q^{(2)} \ar[r] \ar[d]  
		& X_Q^{(2)} \ar[d]^{\varepsilon} \\
		(f \times f)^{-1}(X_P^{(2)}) \ar[r] & P
	}
	\end{gathered}
	\end{equation}
	Let $Q_0 \subset Q$ be the image of $\rho \colon X \to Q$. Since $\rho$ is smooth,
	$Q_0$ is an open dense subset of $Q$. Hence
	$\eta|_{Q_0} \colon Q_0 \to P$ is a morphism of smooth irreducible varieties.
	Since $\pi = \eta|_{Q_0} \circ \rho \colon X \to P$ is smooth, it follows
	that $\eta|_{Q_0}$ is smooth. Thus 
	\[
		\varepsilon \colon X_Q^{(2)} = X_{Q_0}^{(2)} \to Q_0 \stackrel{\eta|_{Q_0}}{\longrightarrow} P
	\]
	is smooth as well of relative dimension $2 \dim X - \dim Q - \dim P$.
	Since each non-empty fiber of $\pi \circ f \colon Z \to P$ has dimension $k$, 
	the image of $Z \times_P Z \to P$ is contained in $\pi(f(Z))$
	and each non-empty fiber of it has dimension $\leq 2k$. Thus the same holds for
	\[
		(f \times f)^{-1}(X_P^{(2)}) \to P \, .
	\]
	Hence $\dim (f \times f)^{-1}(X_P^{(2)}) \leq \dim \overline{\pi(f(Z))} + 2k = \dim Z + k$
	and the estimate~\eqref{Eq.estimate1} follows from the pull-back diagram~\eqref{eq.diagram_double_points}.
	
	We establish~\eqref{Eq.estimate2}: Consider the following fiber product:
	\begin{equation}
	\begin{gathered}
		\label{eq.diagram_kernel}
		\xymatrix@=10pt{
			(\textrm{d} f)^{-1}(\ker \textrm{d} \pi)^\circ \times_P (\ker \textrm{d} \rho)^\circ 
			\ar[r] \ar[d] & (\ker \textrm{d} \rho)^\circ \ar[d] \\
			(\textrm{d} f)^{-1}(\ker \textrm{d} \pi)^\circ \ar[r] & P 
		}
	\end{gathered}
	\end{equation}
	Since $\rho \colon X \to Q$ is smooth, we get $\dim (\ker \textrm{d} \rho)^\circ = 2 \dim X - \dim Q$.
	Hence $(\ker \textrm{d} \rho)^\circ \to P$ is smooth of relative dimension $2 \dim X - \dim Q - \dim P$
	(since $(\ker \textrm{d} \rho)^\circ \to X$ and $\pi \colon X \to P$ are smooth). 
	Moreover,
	\[
		\dim (\textrm{d} f)^{-1}(\ker \textrm{d} \pi)^\circ \leq \dim \ker \textrm{d} (\pi \circ f) \leq \dim Z + k
	\]
	where the second inequality follows from Lemma~\ref{lem.estimate_kernel_differential}, 
	since each non-empty fiber of $\pi \circ f \colon Z \to P$ has dimension $k$ and $Z$ is smooth. 
	Thus the desired estimate~\eqref{Eq.estimate2} follows from the pull-back diagram~\eqref{eq.diagram_kernel}.
\end{proof}


\begin{lemma}
	\label{lem.stick_together}
	Let $f \colon Z \to X$ and $\rho \colon X \to Q$ be morphisms of varieties. Then we have the following:
	\begin{align*}
		\dim Z_Q^{(2)} &= \max\left\{ \dim (f \times f)^{-1}(X_Q^{(2)}), \dim Z^{(2)}_X \right\} \\
		\dim \ker \textrm{d}(\rho \circ f)^\circ  &= \max\left\{\dim (\textrm{d}f)^{-1}(\ker \textrm{d} \rho)^\circ,
		\dim (\ker \textrm{d}f)^\circ \right\} \, .
	\end{align*}
\end{lemma}

\begin{proof}
	The first equality follows, since the underlying set of
	$Z_Q^{(2)}$ is the disjoint union of
	\[
		\set{(z_1, z_2) \in Z \times Z}{ \rho(f(z_1)) = \rho(f(z_2)) \, , \ f(z_1) \neq f(z_2)} 
		= (f \times f)^{-1}(X_Q^{(2)})
	\]
	and the underlying subset of $Z^{(2)}_X$ in $Z \times Z$. 
	The second equality follows, since the underlying set of 
	$\ker \textrm{d}(\rho \circ f)^\circ$ is the disjoint union of
	\[
		\set{v \in TZ}{\textrm{d}(\rho \circ f)(v) = 0 \, , \  (\textrm{d}f)(v) \neq 0} = 
		(\textrm{d}f)^{-1}(\ker \textrm{d} \rho)^\circ
	\]
	and the underlying subset of $(\ker \textrm{d}f)^\circ$ in $TZ$.
\end{proof}

In order to construct embeddings, 
we use the following characterization of them:

\begin{prop}
	\label{prop.char_closed_embeddings_text}
	A morphism $f \colon Z \to X$ of varieties is an embedding if and only if 
	the following conditions are satisfied
	\begin{itemize}[leftmargin=*]
		\item $f$ is proper
		\item $f$ is injective
		\item for each $z \in Z$, the differential $\textrm{d}_z f \colon T_z Z \to T_{f(z)} X$ is injective. 
	\end{itemize}
\end{prop}

We prove this proposition in the Appendix~\ref{AppendixB} 
for the lack of a reference to an elementary proof; 
see Proposition~\ref{prop.crit_embedding}.
From Proposition~\ref{prop.key} and Lemma~\ref{lem.stick_together} we get now immediately the following consequence:

\begin{corollary}
	\label{cor:finite_etale}
	Let $X$ be a smooth irreducible 
	variety such that $\Aut^{\alg}(X)$ acts sufficiently transitively on $X$. 
	If $\rho \colon X \to Q$
	is a finite \'etale surjection and $Z \subset X$ is a smooth closed subvariety with
	$\dim X \geq 2 \dim Z + 1$, then there exists $\varphi \in \Aut^{\alg}(X)$ 
	such that $\rho \circ \varphi \colon X \to Q$
	restricts to an isomorphism $Z \to \rho(\varphi(Z))$.
\end{corollary}
\begin{proof}
	We apply Proposition~\ref{prop.key} to $\pi \colon X \to P \coloneqq \{\textrm{pt}\}$, $\rho \colon X \to Q$
	(note that $Q$ is irreducible and smooth by~\cite[Proposition 3.1, 
	Expos\'e II]{GrRa2003Revetements-etales}),
	and the inclusion $f \colon Z \hookrightarrow X$ in order to get a $\varphi \in \Aut^{\alg}(X)$ such
	that 
	\begin{align*}
		\dim((\varphi \circ f) \times (\varphi \circ f))^{-1}(X_Q^{(2)}) &
		\leq 2\dim Z - \dim Q \leq \dim X - 1 - \dim Q < 0 \\
		\dim (\textrm{d} (\varphi \circ f))^{-1}(\ker \textrm{d} \rho)^\circ 
		&\leq 2\dim Z - \dim Q \leq \dim X - 1 - \dim Q < 0
	\end{align*}
	where we used the assumption $\dim X \geq  2 \dim Z+1$.
	Applying Lemma~\ref{lem.stick_together} to $\varphi \circ f \colon Z \to X$ and $\rho \colon X \to Q$ yields,
	that the composition
	\[
		Z \stackrel{f}{\hookrightarrow} X \stackrel{\varphi}{\longrightarrow} X \stackrel{\rho}{\longrightarrow} Q
	\]
	is injective and the differential 
	$\textrm{d}_z(\rho \circ \varphi \circ f) \colon T_z Z \to T_{\rho(\varphi(z))} Q$ is injective
	for each $z \in Z$. As the composition $\rho \circ \varphi \circ f \colon Z \to Q$ is also proper, 
	the statement follows from
	Proposition~\ref{prop.char_closed_embeddings_text}. 
\end{proof}

The following number associated to each morphism 
will be crucial for the proof of Theorem~\ref{thm.Product}:

\begin{definition}
	\label{def.theta}
	For each morphism $f \colon Z \to X$ of varieties we define the
	\emph{$\theta$-invariant} by
	\[
		\theta_f \coloneqq \max\{ \, \dim Z^{(2)}_X \, , \ \dim (\ker \textrm{d} f)^\circ \, \} \, .
	\]
	In case $W \subseteq Z$ is locally closed, we define the \emph{restricted $\theta$-invariant} by
	\[
		\theta_f |_W \coloneqq \max\{ \, \dim W^{(2)}_X \, , \ 
											\dim (\ker \textrm{d} f)^\circ |_W \, \} \, .
	\]
\end{definition}

Note that $\theta_f$ stays the same if we replace $f$ with $\varphi \circ f$ for an automorphism
$\varphi \in \Aut(X)$. Moreover, the following remarks hold.

\begin{remark}
	\label{rem.theta-invariant}
	If $f \colon Z \to X$ is a proper morphism, then $f$
	is an embedding if and only if $\theta_f < 0$.  
	This follows directly from Proposition~\ref{prop.char_closed_embeddings_text}.
\end{remark}

\begin{remark}
	\label{rem.restricted_theta}
	If $f \colon Z \to X$ is a morphism and if $X_1, \ldots, X_r \subseteq X$ are locally closed subsets
	with $\bigcup_i X_i = X$, then we have
	\[
		\theta_f = \max_i \theta_f |_{f^{-1}(X_i)} \, .
	\]
\end{remark}


The next result will enable us to inductively lower the $\theta$-invariant in the proof
of Theorem~\ref{thm.Product}. We formulate it first in a general version suitable
for the applications, and we formulate it afterwards in the special case needed for the proof of Theorem~\ref{thm.Product}.

\begin{prop}
	\label{prop.lower_theta}
	Let $\rho \colon X \to Q$ be a principal $\GG_a$-bundle,
	$Z$ an affine variety and $r \colon Z \to Q$ a finite morphism.
	Moreover, let $A \subseteq Z$ be a closed subset, 
	let $g_A \colon A \to X$ be a morphism with $\rho \circ g_A = r |_A$ and let
	$Z_1, \ldots, Z_s \subseteq Z \setminus A$
	be locally closed subsets. 
	
	Then there is a morphism $g \colon Z \to X$ with
	$\rho \circ g = r$, $g |_A = g_A$ and such that the restricted $\theta$-invariants satisfy
	$\theta_g |_{Z_i} \leq \theta_r |_{Z_i} - 1$ for all $i$.
\end{prop}

Part of Proposition~\ref{prop.lower_theta} can be illustrated by the following commutative
diagram with filler $g$:
\[
	\xymatrix@=3pt{
		A \ar@{}[dddd]|-{\rotatebox[origin=c]{-90}{$\subset$}} \ar[rrrr]^-{g_A} &&&&  X \ar[dddd]_-{\rho} 
		\\
		\\
		\\
		\\
		Z \ar[rrrr]^-{r} \ar@{.>}[rrrruuuu]^-{\exists g} &&&& Q & .
	}
\]

\begin{proof}
	Let $W \coloneqq r(Z) \subset Q$. Since $r \colon Z \to Q$ is finite and $Z$ is affine, 
	$W$ is a closed affine subvariety of $Q$ 
	by Chevalley's Theorem, \cite[Theorem~12.39]{GoWe2010Algebraic-geometry}. 
	The restriction $\rho^{-1}(W) \to W$ of $\rho$ is locally trivial with respect to the Zariski topology
	(see~\cite[Example, \S2.3]{Se1958Espaces-fibres-alg}) and since $W$ is affine, it is
	a trivial principal $\GG_a$-bundle (see e.g.~\cite[Proposition~1, \S1]{Gr1958Torsion-homologiqu}); 
	this means, there exists a $W$-isomorphism
	$\iota \colon W \times \GG_a \to \rho^{-1}(W)$.
	
	For $i \in \{1, \ldots, s\}$, we choose finite subsets
	\[
		R_i \subseteq (Z_i)_{Q}^{(2)} \quad \textrm{and} \quad
		S_i \subseteq \left(\ker \textrm{d}r \right)^\circ |_{Z_i}
	\]
	such that each irreducible component of
	$(Z_i)_{Q}^{(2)}$ and of $\left(\ker \textrm{d}r \right)^\circ |_{Z_i}$ 
	contains a point of $R_i$ and of $S_i$, respectively.
	Let $\pr_1, \pr_2 \colon Z \times Z \to Z$ 
	be the projection onto the first and second factor, respectively.
	As $Z$ is affine and $Z_i \subset Z \setminus A$ for all $i$, 
	there exists a morphism $q \colon Z \to \GG_a$ such that
	\begin{itemize}[leftmargin=*]
		\item $q$ restricted to $A$ is equal to $\pr_{\GG_a} \circ \iota^{-1} \circ g_A$ where
				$\pr_{\GG_a} \colon W \times \GG_a \to \GG_a$ denotes the natural projection onto $\GG_a$,
		\item $q$ restricted to $\pr_1(R_i) \cup \pr_2(R_i)$ is injective for all $i$ and
		\item $\textrm{d} q \colon TZ \to T\GG_a$ restricted to $S_i$ never vanishes for all $i$.
	\end{itemize}
	Now, we define
	\[
	\xymatrix@=0.1pt{
		g \colon &  Z \ar[rrrr] &&&& W \times \GG_a \ar[rrrr]^-{\iota}_-{\simeq} &&&& \rho^{-1}(W) \subset X \, . \\
		& z \ar@{|->}[rrrr] &&&& (r(z), q(z))
	}
	\]
	Since $\iota \colon W \times \GG_a \to \rho^{-1}(W)$ is a $W$-isomorphism, $\rho \circ g = r$.
	Moreover, by construction we have $g |_{A} = g_A$. Now, we claim that
	\begin{align}
			\label{eq.a}  \dim (Z_i)^{(2)}_X &\leq \dim (Z_i)^{(2)}_Q - 1 \quad \textrm{for all $i$} \, , \\
			\label{eq.b} 
			\dim (\ker \textrm{d} g)^\circ |_{Z_i} &\leq \dim (\ker \textrm{d} r)^\circ |_{Z_i}-1 \quad
			\textrm{for all $i$} \, .
	\end{align}
	
	For proving~\eqref{eq.a}, take an irreducible component $V$ of $(Z_i)_{Q}^{(2)}$. Then
	\[
		V^\circ \coloneqq \set{(v_1, v_2) \in V}{g(v_1) \neq g(v_2)}
	\]
	is an open subset of $V$. By construction, there exists
	$(z_1, z_2) \in R_i \cap V$ with $g(z_1) \neq g(z_2)$. Hence, $V^\circ$ is non-empty.
	This implies that $V \cap (Z_i)_{X}^{(2)}$ is properly contained in $V$. Since 
	$(Z_i)_{X}^{(2)}$ is a closed subset of $(Z_i)_{Q}^{(2)}$,
	we get~\eqref{eq.a}. Similarly, we get~\eqref{eq.b} by
	using that $\textrm{d} g$ restricted to $S_i$ never vanishes. Together, 
	the estimates~\eqref{eq.a} and~\eqref{eq.b} imply that 
	$\theta_g |_{Z_i} \leq \theta_r |_{Z_i} - 1$ for all $i$.
\end{proof}

By choosing $A$ as the empty set, $s=1$, and $Z_1$ equal to $Z$, 
Proposition~\ref{prop.lower_theta} becomes the following.
\begin{corollary}
		\label{cor.lower_theta}
		Let $\rho \colon X \to Q$ be a principal $\GG_a$-bundle,
		$Z$ an affine variety and 
		$r \colon Z \to Q$ a finite morphism.
		Then there exists a morphism $g \colon Z \to X$ such that
		$\rho \circ g = r$ and $\theta_g \leq \theta_r-1$. \qed
\end{corollary}

%
%

We prove Theorem~\ref{thm.Product} by inductively applying Corollary~\ref{cor.lower_theta}.

\begin{proof}[Proof of Theorem~\ref{thm.Product}]
	Let $Z$ be a smooth affine variety such that $\dim X \geq 2 \dim Z + 1$ 
	and such that condition~\ref{thm.Product_d} is satisfied. Let $n \coloneqq \dim P = \dim Z$. 
	
	The following claim will enable us to lower the $\theta$-invariant. \\
	
	\noindent
	\begin{tabular}{ll}
		\textbf{Claim:} & $\exists \ f \colon Z \to X$ 
			such that $\pi \circ f \colon Z \to P$ is finite and $\theta_f \geq 0$ \\
			& $\implies$ $\exists \ g \colon Z \to X$ 
			such that $\pi \circ g \colon Z \to P$ is finite and $\theta_g < \theta_f$
	\end{tabular}
	\bigskip

	\begin{proof}[Proof of Claim] 
	Let $f \colon Z \to X$ be a morphism such that $\pi \circ f \colon Z \to P$ is finite
	and $\theta_f \geq 0$.
	By condition~\ref{thm.Product_d},
	$\eta \colon Q \to P$ is surjective and since $\rho \colon X \to Q$ is surjective,
	we get that $\pi \colon X \to P$ is surjective as well. Since $\rho$ and 
	$\pi$ are smooth surjections 
	and since $X$ is smooth and irreducible, it follows that $P$ and $Q$ are smooth and irreducible;
	see~\cite[Proposition 3.1, Expos\'e II]{GrRa2003Revetements-etales}. By condition~\ref{thm.Product_b},
	$\Aut^{\alg}_P(X)$ acts sufficiently transitively on each fiber of $\pi$.
	Thus we may apply Proposition~\ref{prop.key} to $f \colon Z \to X$ 
	and may choose a $\varphi \in \Aut^{\alg}_P(X)$ such that 
	$f' \coloneqq \varphi \circ f$ satisfies
	\begin{align*}
	&\max \{ \, \dim (f' \times f')^{-1}(X_{Q}^{(2)}) \, , \ \dim (\textrm{d}f')^{-1}(\ker \textrm{d}\rho)^\circ \, \} \\
	&\leq \dim Z + \dim P - \dim Q \, ,
	\end{align*}
	since $\pi \circ f \colon Z \to P$ is finite (see condition~\ref{thm.Product_d}). Note that
	\[
		\dim Z + \dim P - \dim Q = 2n - \dim Q \leq \dim X -1 - \dim Q = 0 \, ,
	\]
	since $\rho \colon X \to Q$ is a principal $\GG_a$-bundle.
	Thus by
	Lemma~\ref{lem.stick_together}:
	\begin{align*}
		\dim Z_{Q, \rho \circ f'}^{(2)} &\leq \max\{\, 0 \, , \ \dim Z_{X, f'}^{(2)} \, \} \\
		\dim \ker \textrm{d}(\rho \circ f')^\circ &\leq \max\{\, 0 \, , \ \dim \ker (\textrm{d}f')^\circ \, \}
	\end{align*}
	where we compute $Z_{Q, \rho \circ f'}^{(2)}$ and $Z_{X, f'}^{(2)}$ with respect
	to $\rho \circ f'$ and $f'$, respectively. 
	Thus $\theta_{f'} \leq \theta_{\rho \circ f'} \leq \max\{0, \theta_{f'}\}$, 
	which implies (as $\theta_f = \theta_{f'} \geq 0$)
	\begin{equation}
		\label{eq_theta}
		\theta_f = \theta_{f'} = \theta_{\rho \circ f'} \, .
	\end{equation}
	Note that $\rho \circ f' \colon Z \to Q$ is finite, since $\pi \circ f' = \pi \circ f$ is finite.
	Hence, applying Corollary~\ref{cor.lower_theta} to $\rho \circ f' \colon Z \to Q$ yields 
	a morphism $g \colon Z \to X$ such that $\rho \circ g = \rho \circ f'$
	and $\theta_g < \theta_{\rho \circ f'}$. Thus, we get $\theta_g < \theta_f$ by~\eqref{eq_theta}.
	Since $\pi \circ g = \pi \circ f'$ is finite, this completes the proof of the claim.
	\end{proof}
	
	By condition~\ref{thm.Product_d}, the composition $\eta \circ r \colon Z \to P$ is finite.
	In particular, $r \colon Z \to Q$ is finite and since $Z$ is affine, there exists a morphism
	$f \colon Z \to X$ such that $\rho \circ f = r$; see
	Corollary~\ref{cor.lower_theta}. By the finiteness of $\pi \circ f = \eta \circ r$,
	we can iteratively apply the claim in order to get a morphism $g \colon Z \to X$
	such that $\pi \circ g \colon Z \to P$ is finite and $\theta_g < 0$. In particular,
	$g \colon Z \to X$ is proper, and, thus, $g \colon Z \to X$ is an embedding
	by Remark~\ref{rem.theta-invariant}.
\end{proof}

\section{Applications: Embeddings into algebraic groups}
\label{sec:appl:groups}

In this section we apply the results from Section~\ref{sec:EmbIntoPriBund} 
in order to construct embeddings
of smooth affine varieties into characterless algebraic groups.

In the entire section, we use the language of and results about algebraic groups, with more notions showing up in later subsections. For the basic results on algebraic groups we refer
to~\cite{Hu1975Linear-algebraic-g} and for the basic results about Lie algebras and root systems
we refer to~\cite{Hu1978Introduction-to-Li}.

\subsection{Embeddings into a product of  the form $\AA^m \times H$}

In this subsection, we study embeddings of smooth affine varieties
into varieties of the from $\AA^m \times H$ where $H$ is a characterless algebraic group.
While this is of independent interest, for us it is also a preparation to establish Theorem~\ref{thmintro:main}; compare with the outline of the proof in the introduction. 

\begin{corollary}
	\label{cor.product}
	Let $H$ be a characterless algebraic group and let $Z$ be a smooth affine variety with
	\begin{equation}
		\label{Eq.Estimate_embbed_prod}
		\tag{$\ast$}
		2 \dim Z + 1 \leq m+ \dim H \, .
	\end{equation}
	If $\dim Z \leq m$, then $Z$ admits an embedding into $\AA^m \times H$.
\end{corollary}

\begin{proof}
	We may and do assume that $H$ is connected. 
	We set $d \coloneqq \dim Z \leq m$ and $G\coloneqq \AA^{m-d} \times H$. Since $G$ is a connected
	characterless algebraic group, 
	$\Aut^{\alg}(G)$ acts sufficiently transitively
	on $G$ by Example~\ref{exa.Enough_transitivity_alg_group}.

	Let $X = \AA^d \times G \simeq \AA^m \times H$. Since
	$\dim G = m + \dim H - d \geq d+1 \geq 1$ due to~\eqref{Eq.Estimate_embbed_prod}
	and since $G$ is characterless, we may and do choose a one-dimensional 
	unipotent subgroup $U \subseteq G$. Let $Q = \AA^d \times G/U$.
	We apply Theorem~\ref{thm.Product} and Remark~\ref{rem.thm.Product}
	to the natural projections
	\[
		\pi \colon X \to \AA^d \, , \quad
		\rho \colon X \to Q \quad \textrm{and} \quad
		\eta \colon Q \to \AA^d 
	\]
	and get our desired embedding $Z \to X$.
\end{proof}

\begin{remark}\label{rem:reproveHKS}
	Corollary~\ref{cor.product} gives us back the Holme-Kaliman-Srinivas embedding theorem, 
	when we take for $H$ the trivial group.
\end{remark}

\subsection{Embeddings into a product of the form $\AA^m \times (\SL_2)^s$}
In this subsection we study the special case $\AA^m \times (\SL_2)^s$.
The main result of the subsection is 
Proposition~\ref{prop.product_SL_2}, which is an analog of Corollary~\ref{cor.product} with a weaker dimension condition. 
This result will be used in order to get
optimal dimension conditions for embeddings into characterless algebraic groups of low dimension in Subsection~\ref{sec.lowdimenion}.

\begin{prop}
	\label{prop.product_SL_2}
	Let $s, m \geq 0$ be integers and let $Z$ be a smooth affine variety with
	\begin{equation}
		\label{Eq.Estimate_embbed_prod_SL2}
		\tag{$\ast\ast$}
		2 \dim Z + 1 \leq m  + \dim \left( (\SL_2)^s \right) \, .
	\end{equation}
	If $\dim Z \leq m+s$, then $Z$
	admits an embedding into $\AA^m \times (\SL_2)^s$.
\end{prop}

\begin{remark}
	\label{rem.product_SL_2}
	In Proposition~\ref{prop.product_SL_2} we may replace the condition $\dim Z \leq m$ by
	$s-1 \leq m$ in case $m + 3s$ is odd and by $s-2 \leq m$ in case $m+3s$ is even. Indeed,
	if $m+3s$ is odd, then $s-1 \leq m$ implies that
	\[
		\dim Z \stackrel{\eqref{Eq.Estimate_embbed_prod_SL2}}{\leq} \frac{m + 3s-1}{2}
		= \frac{m + (s-1) + 2s}{2} \leq \frac{m + m +2s}{2} = m+s  
	\]
	and if $m + 3s$ is even, then $s-2 \leq m$ implies that
	\[
		\dim Z \stackrel{\eqref{Eq.Estimate_embbed_prod_SL2}}{\leq}
		\frac{m + 3s - 2}{2} = \frac{m + (s-2) + 2s}{2} \leq \frac{m + m +2s}{2} = m+s \, .
	\]
\end{remark}

\begin{proof}[Proof of Proposition~\ref{prop.product_SL_2}]
	We may and do assume that $\dim Z \geq 0$.
	If $\dim Z < s$, we may replace $Z$ with $Z \times \AA^{s- \dim Z}$
	and the assumed dimension estimates are still satisfied; thus, we may and do assume
	that $\dim Z \geq s$. We set $d \coloneqq \dim Z$.
	
	Choose a finite morphism $r \colon Z \to \AA^{d}$ 
	(which exists due to Noether Normalization). 
	For any subset $I \subseteq \{1, \ldots, s\}$, let
	\[
		H_I \coloneqq \set{(x_1, \ldots, x_d) \in \AA^d}
					    {\textrm{$x_i = 0$ for each $i \in I$ and $x_i \neq 0$ for each $i \not\in I$}} \, .
	\]
	Moreover, we denote for $k \in \{0, \ldots, d\}$
	\[
		Z_k \coloneqq \set{z \in Z}{\rank \textrm{d}_z r = k} \, ,
	\]
	which is a locally closed subset of $Z$. 
	Note that $\dim Z_k \leq k$. (Indeed, since
	$r |_{Z_k} \colon Z_k \to r(Z_k)$ is a finite morphism, there
	exists $z \in Z_k$ with $\dim Z_k = \rank \textrm{d}_z(r |_{Z_k}) \leq \rank \textrm{d}_z r = k$.)
	Using Kleiman's Transversality Theorem~\cite[2. Theorem]{Kl1974The-transversality} there exists
	an affine linear automorphism $\varphi$ of $\AA^d$ such that
	\[
		\dim Z_k \cap r^{-1}(\varphi^{-1}(H_I)) \leq \dim Z_k + \dim H_I - d \leq k - |I| \, .
	\]
	Hence, after replacing $r$ by $\varphi \circ r$, we may assume that the dimension of the locally 
	closed subset
	\[
		Z_{k, I} \coloneqq  Z_k \cap r^{-1}(H_I) \subseteq Z
	\]
	is less than or equal to $k - |I|$. Since $\rank \textrm{d}_z r = k$ for each $z \in Z_{k, I}$ we get
	\begin{equation}
		\label{eq.estimate2}
		\dim (\ker \textrm{d} r )^\circ |_{Z_{k, I}} 
		\leq \dim (\ker \textrm{d} r ) |_{Z_{k, I}} = \dim Z_{k, I} + (d-k) \leq d - |I| \, .
	\end{equation}
	Now, for $I \subseteq \{1, \ldots, s\}$, let
	\[
		Z_I \coloneqq r^{-1}(H_I) = \bigcup_{k=0}^d Z_{k, I} \subset Z \, .
	\]
	Since $\dim Z_{k, I} \leq k - |I|$ for all $k$, we get 
	$\dim Z_I \leq d- |I|$. Since $r |_{Z_I} \colon Z_I \to \AA^d$ is finite, 
	the projection $\dim (Z_I)^{(2)}_{\AA^d} \to Z_I$ 
	to one of the factors is quasi-finite. Hence, 
	$\dim (Z_I)^{(2)}_{\AA^d} \leq \dim Z_I \leq d - |I|$,
	and by the estimate~\eqref{eq.estimate2} we get $\dim (\ker \textrm{d} r )^\circ |_{Z_I} \leq d - |I|$.
	In total the restricted $\theta$-invariants of $r$ satisfy
	\begin{equation}
		\label{eq.theta_r_restricted}
		\theta_r |_{Z_I} \leq d - |I| \quad \textrm{for all $I \subseteq\{1, \ldots, s\}$} \, .
	\end{equation}

	We set
	\[
		X_l	 \coloneqq (\AA^2 \setminus \{(0,0)\})^{l-1} \times \AA^2 \times \AA^{d-l}   \, ,
		\quad
		Q_l  \coloneqq (\AA^2 \setminus \{(0,0)\})^{l} \times \AA^{d-l}\, ,
	\]
	and
	\[
		\rho_l \coloneqq \id_{(\AA^2 \setminus \{(0,0)\})^{l-1}} \times \pr_1 \times \id_{\AA^{d-l}}  
			\colon X_l \to Q_{l-1}\, ,
	\]
	where $\pr_1 \colon \AA^2 \to \AA^1$ denotes 
	the projection onto the first factor. 
	Next, for $l \in \{0, \ldots, s\}$, we
	construct inductively finite
	morphisms $g_l \colon Z \to Q_l$
	such that we have $\rho_{l} \circ g_l = g_{l-1}$, $\theta_{g_l} |_{Z_I} \leq \theta_{g_{l-1}} |_{Z_I}$ 
	for all $I \subseteq \{1, \ldots, s \}$, and $\theta_{g_l}|_{Z_I} \leq \theta_{g_{l-1}}|_{Z_I} -1$ for 
	$I\subset \{1, \ldots, s\}$ with $l\notin I$.
	
	Let $g_0 \colon Z \to Q_0$ be the finite morphism $r \colon Z \to \AA^d$.
	By induction, we assume that the finite morphism
	\[
		g_{l-1} = \left(g_{l-1}^{(1)}, \ldots, g_{l-1}^{(l+d-1)}\right) \colon Z \to Q_{l-1}
	\]
	is already constructed for some $1 \leq l \leq s$. We apply Proposition~\ref{prop.lower_theta}
	to the trivial $\GG_a$-bundle $\rho_l \colon X_l \to Q_{l-1}$, the closed subset
	\[
		A \coloneqq 
		r^{-1}\left(\set{(x_1, \ldots, x_d) \in \AA^d}{x_l = 0}\right) =
		\bigcup_{I \subseteq \{1, \ldots, s\} \colon l \in I} Z_I \subseteq Z \, ,
	\]
	and the morphism
	\[
		g_A \colon A \to Q_l \, , \quad a \mapsto 
		\left(g_{l-1}^{(1)}(a), \ldots, g_{l-1}^{(2l-1)}(a), 1, 
		g_{l-1}^{(2l)}(a), \ldots, g_{l-1}^{(l+d-1)}(a)\right) 
	\]
	in order to get a morphism 
	$g_l \colon Z \to X_l$ with $\rho_l \circ g_l = g_{l-1}$, $g_l |_A = g_A$
	and 
	\begin{equation}
		\label{eq.lower_theta_piece}
		\theta_{g_l} |_{Z_I} \leq \theta_{g_{l-1}} |_{Z_I}-1 \quad 
		\textrm{for all $I \subseteq \{1, \ldots, s\}$ with $l \not\in I$}
	\end{equation}
	(here we used that $Z_I \subseteq Z \setminus A$ for each $I$ with $l \not\in I$).
	Since $\rho_l \circ g_l = g_{l-1}$, we get
	$\theta_{g_l} |_{Z_I} \leq \theta_{g_{l-1}} |_{Z_I}$
	for all $I \subseteq \{1, \ldots, s \}$. Since $g_l|_A=g_A$ and since $g_{l-1}^{(2l-1)}$
	is equal to the $l$-th coordinate function of $r$, we get that
	the image of $g_l$ is contained in $Q_l$. Thus, we may consider $g_l$ as a morphism
	$Z \to Q_l$.
	
	Now, 
	\begin{align*}
		\theta_{g_s} |_{Z_I} \stackrel{\eqref{eq.lower_theta_piece}}{\leq} \theta_r |_{Z_I} - (s-|I|) 
									  \stackrel{\eqref{eq.theta_r_restricted}}{\leq} d-s \quad 
		\textrm{for all $I \subseteq \{1, \ldots, s\}$} \, .
	\end{align*}
	Since $r \colon Z \to \AA^d$ factorizes through $g_s \colon Z \to Q_s$,
	$\AA^d = \bigcup_I H_I$, and $Z_I = r^{-1}(H_I)$,
	Remark~\ref{rem.restricted_theta} implies that
	\begin{equation}
		\label{eq.theta_g_s}
		\theta_{g_s} = 
		\max_{I \subseteq \{1, \ldots, s\}} \theta_{g_s} |_{Z_I} \leq d-s \, .
	\end{equation}
	
	Finally, let 
	$\rho \coloneqq \eta^s \times \pr \colon (\SL_2)^s \times \AA^m \to Q_s = 
	(\AA^2 \setminus \{(0, 0)\})^s \times \AA^{d-s}$, where
	$\eta \colon \SL_2 \to \AA^2 \setminus \{(0, 0)\}$ denotes the projection to the first column
	and $\pr \colon \AA^m \to \AA^{d-s}$ is a surjective linear map (such a map exists, since $d \leq s+m$).
	Since $\eta$ is a $\GG_a$-bundle, $\rho$ is the composition of $2s+m-d$ many $\GG_a$-bundles.
	Thus, Corollary~\ref{cor.lower_theta} gives us a morphism $g \colon Z \to (\SL_2)^s \times \AA^m$
	such that $\rho \circ g = g_s$ and $\theta_g \leq \theta_{g_s} - (2s+m-d)$. Using the
	estimate~\eqref{eq.theta_g_s} gives us $\theta_g \leq 2d - (3s+m)$. By 
	\eqref{Eq.Estimate_embbed_prod_SL2}, we have
	$2d - (3s+m) < 0$. Since $g$ is proper, Remark~\ref{rem.theta-invariant} implies that
	$g$ is an embedding.
	\end{proof}

\subsection{Embeddings into (semi)simple algebraic groups}
In this sub\-section, we consider arbitrary (semi)simple algebraic groups 
$G$ as targets of embeddings of smooth
affine varieties $Z$. 
However, while doing so the price we have to pay is to relax
the dimension condition $2 \dim Z + 1 \leq \dim G$ in order to get an embedding of $Z$ into $G$.

From the point of view of the outline of the proof of Theorem~\ref{thmintro:main} in the introduction, the content of this subsection can be summarized as follows. Fixing a semisimple algebraic group $G$, we start with two lemmas (Lemma~\ref{lem:max_parabolic} and Lemma~\ref{lem:G_red_commutator}) that yield closed subvarieties $X_P\subseteq G$ with $X_P \simeq \AA^m\times H$ based on a choice of a parabolic subgroup $P\subseteq G$. 
We then formulate a version of Theorem~\ref{thmintro:main} for semisimple algebraic groups where the dimension assumption on $Z$ depends on dimension estimates for a chosen parabolic subgroup $P$ and its subgroups $P^u$ and $R_u(P)$ defined below. Finally, we provide dimension estimates for $P^u$ and $R_u(P)$ for good choices of $P\subset G$ for simple algebraic groups based on the classification of simple Lie algebras (Proposition~\ref{prop.estimate}). This suffices to yield Theorem~\ref{thmintro:main} by applying Corollary~\ref{cor.product} to $X_P$ for a good choice of $P$ (Theorem~\ref{thm.simple}).

We recall a few notions. If $G$ is an algebraic group, we denote by $R(G)$ the radical, by $R_u(G)$ its unipotent radical, 
and by $G^u$ the closed subgroup of $G$ that is generated by all unipotent elements of $G$.
Recall that a connected algebraic group $G$ is called semisimple if $G$ is non-trivial and $R(G)$ is trivial,
and it is called simple if $G$ is non-commutative and contains no non-trivial proper closed connected
normal subgroup. Moreover, a non-trivial algebraic group $G$ is called reductive if $R_u(G)$ is trivial.

For lack of a reference, we insert a proof of the following classical facts:

\begin{lemma}
	\label{lem:max_parabolic}
	Let $G$ be a semisimple algebraic group and let $P \subset G$ be a parabolic subgroup.
	Then the following holds:
	\begin{enumerate}[leftmargin=*]
		\item \label{lem:max_parabolic1} If $L$ is a Levi factor of $P$, then $R_u(P) \rtimes L^u =  P^u$.
		\item \label{lem:max_parabolic1.5} If $P^- \subset G$ is an opposite parabolic subgroup to $P$, 
															then we have $\dim G = \dim R_u(P) + \dim P$ and   
															the product morphism
															\[
																R_u(P^{-}) \times R_u(P) \times (P \cap P^-)^u \to G
															\] 
															is an embedding\footnote{By convention, for us embeddings are closed. In contrast, the product morphism $G \times G \times G \to G$
															restricts to an isomorphism
															$R_u(P^{-}) \times R_u(P) \times (P \cap P^-) \to W$, 
															where $W$ is
															\emph{open} in $G$.}.
		\item \label{lem:max_parabolic2} If $P$ is a maximal parabolic subgroup of $G$
		(i.e. $P$ is a maximal proper subgroup of $G$ that contains a Borel subgroup),
		then $\dim P^u = \dim P - 1$.
	\end{enumerate}
\end{lemma}
For the proof of Lemma~\ref{lem:max_parabolic} we need the following Lemma.

\begin{lemma}
	\label{lem:G_red_commutator}
	Let $G$ be a connected reductive algebraic group. Then 
	\begin{enumerate}[label=\alph*), leftmargin=*]
		\item $G^u = [G, G]$,
		\item $G = G^u \cdot R(G)$, $G^u \cap R(G)$ is finite and $G^u$ is trivial or semisimple. 
	\end{enumerate}
\end{lemma}
\begin{proof}[Proof of Lemma~\ref{lem:G_red_commutator}]
	Note that $G/G^u$ is an algebraic torus, as it is connected and contains only semisimple elements; 
	see \cite[Proposition 21.4B and Theorem 19.3]{Hu1975Linear-algebraic-g}. In particular,
	$G^u$ contains the commutator subgroup $[G, G]$.
	
	On the other hand, for every non-trivial character $\alpha$ of 
	a maximal algebraic torus $T \subset G$, $[G, G]$ contains the root subgroup $U_\alpha \subset G$ with 
	respect to $T$, since for each isomorphism 
	$\lambda \colon \GG_a \to U_\alpha$ we have
	\[
		\lambda(\alpha(t)-1) = \lambda(\alpha(t))\lambda(1)^{-1} =
		t \lambda(1) t^{-1} \lambda(1)^{-1} \in [G, G] \quad
		\textrm{for every $t \in T$} \, .
	\] 
	Hence $[G, G]$ contains $G^u$ and thus we get the first statement.
	
	The second statement follows from the first statement and 
	from~\cite[Proposition 14.2, Ch.~IV]{Bo1991Linear-algebraic-g}.
\end{proof}

\begin{proof}[Proof of Lemma~\ref{lem:max_parabolic}]
	\eqref{lem:max_parabolic1}: By definition we have $R_u(P) \ltimes L = P$. Hence, we get
	an inclusion $R_u(P) \rtimes L^u \subset P^u$. On the other hand, the inclusion $P^u \subset P$
	induces an inclusion $P^u/R_u(P) \subset P / R_u(P)$ and $\pi \colon P \to P / R_u(P)$
	restricts to an isomorphism $\pi |_L \colon L \to  P / R_u(P)$. Hence, 
	\[
		L^u \xlongrightarrow[\simeq]{\pi |_{L^u}} (P / R_u(P))^u = P^u/R_u(P) \, ,
	\]
	which implies~\eqref{lem:max_parabolic1}.
	
	\eqref{lem:max_parabolic1.5}: By~\cite[Example 3.10]{Ti2011Homogeneous-spaces},
	the algebraic quotient $G/P^u$ is quasi-affine. Let $P^-$ be an opposite parabolic subgroup to $P$.
	The orbit in $G /P^u$ 	through the class of the neutral 
	element under the natural action of the unipotent radical $R_u(P^-)$ is therefore closed in $G /P^u$. 
	This implies that $R_u(P^-)P^u$ is closed in $G$. 
	
	By definition, $L \coloneqq P \cap P^-$ is a Levi factor of $P$ (and also of $P^-$). 
	The product morphism induces an isomorphism of varieties
	\[
		R_u(P^-) \times R_u(P) \times L \stackrel{\simeq}{\longrightarrow} R_u(P^-) \times P 	\stackrel{\simeq}{\longrightarrow} R_u(P^-) P
	\]
	and $R_u(P^-)P$ is an open dense subset of $G$ (see 
	\cite[Proposition 14.21]{Bo1991Linear-algebraic-g} or \cite[Appendix~B.2]{FeSa2019Uniqueness-of-embe}). 
	This gives the first statement. Due to 
	\eqref{lem:max_parabolic1}, we have $P^u = R_u(P) \rtimes L^u$.
	Hence, the above isomorphism restricts to an
	isomorphism:
	\[
	R_u(P^-) \times R_u(P) \times L^u \stackrel{\simeq}{\longrightarrow} R_u(P^-) P^u \, .
	\]
	
	\eqref{lem:max_parabolic2}:
	By construction $G / R_u(G)$ is a reductive or trivial algebraic group. 
	In the second case, $G$ contains no maximal 
	parabolic subgroup and thus we may assume that $G/R_u(G)$
	is reductive. Since $R_u(G)$ is contained 
	in every Borel subgroup of $G$, it follows that $R_u(G)$ is contained in $P$. Thus $P / R_u(G)$
	is a maximal parabolic subgroup of $G/R_u(G)$. Since $R_u(G) \subset P^u$, we get an isomorphism
	\[
		P / P^u \simeq (P/R_u(G)) / (P^u/R_u(G)) \, .
	\]
	Thus, it is enough to show \eqref{lem:max_parabolic2} in case $G$ is reductive
	(and by definition it is connected).
	
	Let $B \subset G$ be a Borel subgroup, $T \subset B$ a maximal algebraic torus,
	$r = \dim T$, $r$ is the rank of $G$, and let $\frak{X}(T)$ be the group of characters of $T$.
	We may choose simple roots $\alpha_1, \ldots, \alpha_r \in  \frak{X}(T)$ 
	such that $P$ is the parabolic subgroup
	with respect to $\alpha_1, \ldots, \alpha_{r-1}$; see \cite[Theorem in \S29.3]{Hu1975Linear-algebraic-g}.
	Let
	\[
		Z_i = \left(\bigcap_{j=1}^i \ker(\alpha_j)\right)^\circ \subset T \quad
		\textrm{for each $i =1, \ldots, r$}
	\]
	where $H^\circ$ denotes the identity component of a closed subgroup $H \subset G$.
	Since by definition $\alpha_1, \ldots, \alpha_r$ form a basis of $\frak{X}(T) \otimes_\ZZ \RR$, it
	follows that over $\ZZ$ the elements $\alpha_1, \ldots, \alpha_r$ are linearly independent.
	Hence, the dimension of $Z_i$ is $r-i$. From~\cite[\S30.2]{Hu1975Linear-algebraic-g}, it follows that
	\[
		R(P) = R_u(P) \rtimes Z_{r-1}  \, .
	\]
	Now, let $Q\coloneqq P / R_u(P)$. Thus $Q$ is a connected reductive algebraic group.
	Since $P^u$ is the preimage of $Q^u$ under the canonical projection
	$\pi \colon P \to Q$, we get
	\[
		P / P^u \simeq Q / Q^u \, .
	\]
	Note that $\pi(R(P))$ is a normal solvable connected subgroup of $Q$ and thus
	$\pi(R(P)) \subset R(Q)$. On the other hand, $\pi^{-1}(R(Q))$ is a normal, connected subgroup
	and it is solvable, as $R_u(P) = \ker(\pi)$ and $R(Q)$ are solvable. The latter two statements
	together imply that $\pi^{-1}(R(Q)) = R(P)$ and thus 
	\[
		 R(Q) \simeq Z_{r-1} \, .
	\]
	By Lemma~\ref{lem:G_red_commutator}, $Q = Q^u \cdot R(Q)$ and $R(Q) \cap Q^u$ is finite. Thus, 
	the canonical projection $Q \to Q / R(Q)$ restricts to an isogeny $Q^u \to Q/R(Q)$. In total:
	\[
		1 = \dim Z_{r-1} = \dim R(Q) = \dim Q - \dim Q^u = \dim Q /Q^u = \dim P / P^u \, .\qedhere
	\]
\end{proof}

\begin{theorem}
	\label{thm.simple}
	Let $G$ be a simple algebraic group, let $k \geq 0$ be an integer,  
	and let $Z$ be a smooth affine variety. If $\dim G +k > 2 \dim Z + 1$,
	then $Z$ admits an embedding into $G\times \AA^k$.
\end{theorem}

%


For the proof of Theorem~\ref{thm.simple} we will use the two next propositions.

\begin{prop}
	\label{prop.embedding_general}
	Let $G$ be a semisimple algebraic group and let $k \geq 0$ be an integer. 
	If there exists a parabolic subgroup $P \subset G$
	with $\dim P^u-1 \leq 3 \dim R_u(P)$, then for every smooth affine variety $Z$ with
	\[
	2 \dim Z + \dim P - \dim P^u < \dim G + k
	\]
	there exists an embedding of $Z$ into $G \times \AA^k$.
\end{prop}

\begin{prop}
	\label{prop.estimate}
	Let $G$ be a simple algebraic group. Then there exists a maximal
	parabolic subgroup $P \subset G$ such that $\dim P^u \leq 3 \dim R_u(P)$.
\end{prop}

\begin{proof}[Proof of Theorem~\ref{thm.simple}]
	Let $P \subset G$ be a maximal parabolic subgroup as in Proposition~\ref{prop.estimate}.
	By Lemma~\ref{lem:max_parabolic}\eqref{lem:max_parabolic2} we have 
	$\dim P - \dim P^u = 1$.
	Thus the theorem follows from Proposition~\ref{prop.embedding_general}.
\end{proof}

\begin{proof}[Proof of Proposition~\ref{prop.embedding_general}]
	By Lemma~\ref{lem:max_parabolic}\eqref{lem:max_parabolic1.5} 
	there exists an embedding 
	of $\AA^m \times H$ into $G \times \AA^k$, where $m = 2 \dim R_u(P) +k$ and 
	$H = (P \cap P^-)^u$ for
	an opposite parabolic subgroup $P^- \subset G$ of $P$. By Lemma~\ref{lem:max_parabolic}\eqref{lem:max_parabolic1} we have $\dim H = \dim P^u - \dim R_u(P)$.
	Now, we get 
	\begin{equation}
		\label{eq.estimate_H}
		\dim H -1 = \dim P^u - \dim R_u(P) -1 \leq 2 \dim R_u(P)\leq m \, . 
	\end{equation}
	By Lemma~\ref{lem:max_parabolic}\eqref{lem:max_parabolic1.5} we get 
	$\dim G = \dim R_u(P)+ \dim P$. Hence
	\begin{align*}
		2 \dim Z + 1 &\leq \dim G - \dim P + \dim P^u +k \\
						 &= \dim P^u + \dim R_u(P)+k \\
						 &= \dim H + m \, .
	\end{align*}
	Thus, we get $\dim Z \leq \frac{\dim H - 1 + m}{2} \leq m$ by~\eqref{eq.estimate_H}.
	Hence, the proposition follows from Corollary~\ref{cor.product}.
\end{proof}

\begin{proof}[Proof of Proposition~\ref{prop.estimate}]
	Let $P \subset G$ be a maximal parabolic subgroup.
	By Lemma~\ref{lem:max_parabolic}\eqref{lem:max_parabolic1.5},\eqref{lem:max_parabolic2}
	we get $\dim R_u(P) + \dim P^u + 1 = \dim G$. Let $L$ be a Levi factor of $P$. Then,
	by Lemma~\ref{lem:max_parabolic}\eqref{lem:max_parabolic1} 
	$\dim P^u = \dim L^u + \dim R_u(P)$.
	Now, if we find a maximal parabolic subgroup $P$ in $G$ such that
	\begin{equation}
		\label{Eq.Condition_for_max_parabolic}
		\dim G \geq 2 \dim L^u + 1 \, ,
	\end{equation}
	then we are done, as in this case we would get 
	\begin{align*}
		3\dim R_u(P) &= \dim R_u(P) + \dim P^u + 1 - 1 + 2\dim R_u(P) - \dim P^u \\
							&= \dim G -1 + 2 \dim R_u(P) - \dim P^u \\
							&\geq 2 \dim L^u + 2 \dim R_u(P) - \dim P^u \\
							&=\dim P^u \, .
	\end{align*}
	
	We treat first the case, when $G$ is one of 
	the classical Lie-types $A_n, B_n, C_n$ or $D_n$.
	For $n \geq 1$, 
	we denote by $a_n, b_n, c_n, d_n$ the dimension of the Lie algebra of type $A_n, B_n, C_n$
	and $D_n$, respectively. By \cite[\S1.2]{Hu1978Introduction-to-Li}, we get
	\[
		a_n = n^2 + 2n \, , \quad b_n = c_n = 2 n^2 + n \, , \quad d_n = 2n^2 - n \, .
	\]
	Now we choose $s \in \mathbb{N}_0$ according to the Lie-type as follows
	\begin{center}
		\begin{tabular}{l|l|l}
			Lie-type & Dynkin diagram & s \\
			\hline
			$A_n$, $n \geq 1$ & \dynkin A{}  &  $\floor*{(n+ 1)/2}$ \\
			$B_n$, $n \geq 2$ & \dynkin B{} & $\floor*{(4n + 1)/6}$ \\
			$C_n$, $n \geq 3$ & \dynkin C{} & $\floor*{(4n + 1)/6}$ \\
			$D_n$, $n \geq 4$ & \dynkin D{} & $\floor*{(4n - 1)/6}$ \\
		\end{tabular}
	\end{center}
	where $\floor*{x}$ means the largest integer that is smaller or equal than $x$. 
	In order to specify the maximal parabolic subgroup $P$ of $G$, let $I$ be the set of all simple roots in the Dynkin diagram
	of $G$, except the simple root at position $s$, when we count from the 
	left in the Dynkin diagram. We let $P$ be the standard parabolic subgroup with respect to $I$
	and some fixed chosen Borel subgroup of $G$ and we let (as above) $L \subset P$ be a Levi factor.
	Then $L^u$ is semisimple or trivial (by Lemma~\ref{lem:G_red_commutator}) 
	and the corresponding Dynkin diagram is the  Dynkin diagram of $G$
	with the vertex $s$ (counted from the left) deleted;
	see~\cite[\S30.2]{Hu1975Linear-algebraic-g}. 
	For example, if the Lie type of $G$ is $B_4$,
	then $s = \floor*{17/6} = 2$ and we have the following Dynkin diagrams
	(the cross \dynkin[parabolic=1,x/.style={gray,very thick}] A1 means to delete the corresponding simple root):
	\[
		\textrm{$G$:} \ \dynkin B4 \qquad \qquad
		\textrm{$P$:} \ \dynkin[parabolic=2,x/.style={gray,very thick}] B4
		\quad \implies \quad \dim L^u = a_1 + b_2 = 13 \, .
	\]
	By considering the Dynkin diagrams for the classical types
	$A_n, B_n, C_n$ and $D_n$ and by using that $a_1 = b_1 = c_1$, $d_2 = 2 a_1$ and $d_3 = a_3$,
	we get
	\begin{center}
		\begin{tabular}{l|l|rl}
			Lie-type & s & $\dim L^u$ & \\
			\hline
			$A_n$, $n \geq 1$ & $\floor*{\frac{n+ 1}{2}} \geq 1$ &  $a_{s-1} + a_{n-s}$
			= &$2s^2-(2n+2)s+n^2+2n-1$ \\[0.1cm]
			$B_n$, $n \geq 2$ & $\floor*{\frac{4n + 1}{6}} \geq 1$ & $a_{s-1} + b_{n-s}$
			= &$3s^2-(4n+1)s+2n^2+n-1$\\[0.1cm]
			$C_n$, $n \geq 3$ & $\floor*{\frac{4n + 1}{6}} \geq 2$ & $a_{s-1} + c_{n-s}$
			= &$3s^2-(4n+1)s+2n^2+n-1$ \\[0.1cm]
			$D_n$, $n \geq 4$  & $\floor*{\frac{4n - 1}{6}} \geq 2$ & $a_{s-1} + d_{n-s}$
			= &$3(s+1)^2-(4n+5)(s+1)$ \\
			& & & $+2n^2+3n+1$.
		\end{tabular}
	\end{center}
From this table we conclude $\dim G -2\dim L^u \geq 0$ as desired. 
We provide the detailed calculation. For $A_n$ with $n\geq 1$, we note 
	\begin{align*}
   		\dim G -2\dim L^u-1=\;&-n^2 + 4(n + 1)\floor*{\frac{n+1}{2}} - 4\floor*{\frac{n+1}{2}}^2 - 2n + 1
   		\\=\;&
		\left\{\begin{array}{l}
			-n^2 + 4(n + 1)\frac{n+1}{2} - 4\left(\frac{n+1}{2}\right)^2 - 2n + 1 \\
			\hspace{0.5cm} \text{if $n$ is odd}\\
        	-n^2 + 4(n + 1)\frac{n}{2} - 4\left(\frac{n}{2}\right)^2 - 2n + 1 \\
        	\hspace{0.5cm} \text{if $n$ is even}
		\end{array}\right. \\
		=\;&
		\left\{\begin{array}{cl}
			2 & \text{if $n$ is odd}\\
			1 & \text{if $n$ is even}
		\end{array}\right. \\
		\geq\;&0 \, .
	\end{align*}
		For $B_n$ and $C_n$ with $n\geq2$ and $n\geq3$, respectively, and $x\in\{0,-2,-4\}$ such that $6$ divides $4n+x$, we calculate
	\begin{align*}
	 \dim G -2\dim L^u -1=\;&-2n^2 + 2(4n + 1)\floor*{\frac{4n+1}{6}} - 6\floor*{\frac{4n+1}{6}}^2 - n + 1\\
	 =\;&-2n^2 + 2(4n + 1)\frac{4n+x}{6}-\frac{(4n+x)^2}{6} - n + 1 \\
	 =\;&\frac{2n^2+n}{3}+1+\frac{2x-x^2}{6}\geq\frac{2n^2+n}{3}+1-4 \\
	 \geq\;&0 \, .
 \end{align*}
 For $D_n$ with $n\geq4$ and $x\in\{0,2,4\}$ such that $6$ divides $4n+x$, we calculate
 \begin{align*}
	 \dim G -2\dim L^u-1=\;&-2n^2+2(4n + 5)\floor*{\frac{4n+5}{6}} - 6\floor*{\frac{4n+5}{6}}^2\\
	 &-7n - 3\\
	 =\;&-2n^2+2(4n + 5)\frac{4n+x}{6} - \frac{(4n+x)^2}{6} \\
	 &-7n - 3\\
	=\;&\frac{2n^2-n}{3}+\frac{10x-x^2}{6}-3 \\
	\geq\;&\frac{2 n^2-n}{3}-3 \\
	\geq\;&0 \, .
  \end{align*}
  
  Now, for the exceptional Lie-types, we choose $P$ as in the table below  and
  the estimate~\eqref{Eq.Condition_for_max_parabolic} follows from the same table
  (again the cross \dynkin[parabolic=1,x/.style={gray,very thick}] A1 
  in the dynkin diagram of $P$ means, to remove the corresponding simple root):
  
  \begin{center}
		\begin{tabular}{l|l|l|l|l}
			Lie-type & Dynkin diagram & $\dim G$ & Dynkin diagram  & $\dim L^u$ \\
						  & of $G$ & & of $P$ \\
			\hline
			$E_6$ & \dynkin E6 & 78 & \dynkin[parabolic=8,x/.style={gray,very thick}] E6 &  $a_1 + a_2 + a_2 = 19$ \\
			$E_7$ & \dynkin E7 & 133 & \dynkin[parabolic=8,x/.style={gray,very thick}] E7 & $a_1 + a_2 + a_3 = 26$ \\
			$E_8$ & \dynkin E8 & 248 & \dynkin[parabolic=8,x/.style={gray,very thick}] E8 & $a_1 + a_2 + a_4 = 35$ \\
			$F_4$ & \dynkin F4 & 52 & \dynkin[parabolic=8,x/.style={gray,very thick}] F4 &  $b_3 = 21$ \\
			$G_2$ & \dynkin G2 & 14 & \dynkin[parabolic=2,x/.style={gray,very thick}] G2 &  $a_1 = 3$ \, .\qedhere
		\end{tabular}
  \end{center}  
\end{proof}

Having settled the case for simple algebraic groups, we go on to semisimple algebraic groups.
The following result generalizes Theorem~\ref{thm.simple}.

\begin{theorem}
	\label{thm.semisimple}
	Let $G$ be a semisimple algebraic group and let $k\geq 0$ be an integer. 
	Let $r \geq 1$ be the number of minimal normal
	closed  connected 
	subgroups of $G$. If $Z$ is a smooth affine variety with $\dim G +k > 2 \dim Z + r$, 
	then there exists an embedding of $Z$ into $G \times \AA^k$. 
\end{theorem}

\begin{proof}
	Let $G_1, \ldots, G_r$ be the minimal normal closed connected subgroups of $G$. 
	By~\cite[Theorem in \S27.5]{Hu1975Linear-algebraic-g}, the product morphism
	$G_1 \times \cdots \times G_r \to G$ is a finite \'etale surjection. In the light of
	Corollary~\ref{cor:finite_etale} we may thus assume $G = G_1 \times \cdots \times G_r$.
	Since $G_i$ is a simple algebraic group, there exists a maximal parabolic subgroup 
	$P_i \subset G_i$ such that $3 \dim R_u(P_i) \geq \dim P_i^u$, by Proposition~\ref{prop.estimate}.
	Let 
	\[
		P \coloneqq P_1 \times P_2 \times \cdots \times P_r \subset G_1 \times G_2 \times \cdots \times G_r \, .
	\]
	Then we get $P^u = P_1^u \times \cdots \times P_r^u$ and 
	$R_u(P) = R_u(P_1) \times \cdots \times R_u(P_r)$ and therefore $3 \dim R_u(P) \geq \dim P^u$.
	Since $\dim P_i - \dim P_i^u = 1$ for each $i \in \{1, \ldots, r\}$ 
	(Lemma~\ref{lem:max_parabolic}\eqref{lem:max_parabolic2}), we get $\dim P - \dim P^u = r$.
	Thus, the theorem follows from Proposition~\ref{prop.embedding_general}.
\end{proof}

\subsection{Embeddings into algebraic groups of low dimension}
\label{sec.lowdimenion}

	Our main result concerning characterless algebraic groups of low dimension 
	is the following.

	\begin{prop}
		\label{prop.lowdimension}
		Let $G$ be a characterless algebraic group with $\dim G \leq 10$
		and let $Z$ be a smooth affine variety with $2 \dim Z + 1 \leq \dim G$.
		If the Lie algebra of $G$ is non-isomorphic to
		$\mathfrak{sl}_2 \times \mathfrak{sl}_2 \times \mathfrak{sl}_2$ and non-isomorphic to
		$\mathfrak{sl}_3 \times \k$, then $Z$ admits an embedding into $G$.
	\end{prop}

	Before giving the proof, let us shortly comment on the above result. Proposition~\ref{prop.lowdimension} implies that
	for any characterless algebraic group $G$ with $\dim G \leq 8$ the condition
	$2 \dim Z + 1 \leq \dim G$ suffices to get an embedding of $Z$ into $G$. 

	\begin{question}
		Does every $4$-dimensional smooth affine variety embed into the algebraic group
		$\SL_2 \times \SL_2 \times \SL_2$ or into $\SL_3 \times \GG_a$?
	\end{question}

	\begin{proof}[Proof of Proposition~\ref{prop.lowdimension}]
		Let $G$ be a characterless algebraic group of dimension $\leq 10$ such that its Lie algebra 
		is neither isomorphic to 
		$\mathfrak{sl}_2 \times \mathfrak{sl}_2 \times \mathfrak{sl}_2$ nor to $\mathfrak{sl}_3 \times \k$.
		We may and will assume that 
		$G$ is connected. Using a Levi decomposition \cite[Theorem 4, Ch.~6]{OnVi1990Lie-groups-and-alg}, $G$ is isomorphic as a variety
		to $\AA^m \times H$ where $H$ is a connected reductive characterless algebraic group.
		In particular, $H$ is semisimple or trivial; see~\cite[Remark 8.3]{FeSa2019Uniqueness-of-embe} and Lemma~\ref{lem:G_red_commutator}.
		In case $H$ is trivial, the result follows from the Holme-Kaliman-Srinivas embedding theorem.
		Thus we may assume that $H$ is semisimple. Since every semisimple algebraic group
		is the target of a finite homomorphism of a product of simple algebraic groups 
		(see \cite[Theorem in \S27.5]{Hu1975Linear-algebraic-g}), 
		we may
		assume that $H$ is the product of simple algebraic groups by Corollary~\ref{cor:finite_etale}.
		From the classification of simple Lie algebras it follows that a simple algebraic group
		of dimension $\leq 10$ has Lie algebra equal to $\mathfrak{sl}_2$, $\mathfrak{sl}_3$ or $\mathfrak{so}_5 = \mathfrak{sp}_4$. Again using Corollary~\ref{cor:finite_etale}, we may assume that the factors
		of $H$ are simple algebraic groups that are not targets of non-trivial finite homomorphisms.
		Hence, $H$ is a product of the groups
		\[
			\SL_2 \, , \quad \SL_3 \quad \textrm{and} \quad \textrm{Sp}_4 \, .
		\]
		If $H$ has a factor equal to $\SL_3$ or $\textrm{Sp}_4$, then the statement follows from 
		Theorem~\ref{thm.simple} (note we excluded the case $\AA^1 \times \SL_3$). 
		Hence, we are left with the case
		\[
			G \simeq \AA^m \times (\SL_2)^s \, .		
		\]
		for some $s \geq 1$.
		We distinguish two cases:
		\begin{itemize}[leftmargin=*]
			\item $m + 3s$ is odd: In case $s-1 \leq m$, the statement follows from Remark~\ref{rem.product_SL_2} and Proposition~\ref{prop.product_SL_2}. Thus we assume
			that $s-1 > m$. Since $m+3s \leq 10$ by assumption, we get 
			$0 \leq m \leq \min\{10-3s, s-2\}$. Since $m + 3s$ is odd, this implies that $(s, m) = (3, 0)$,
			which contradicts the assumption that the Lie algebra of $G$ is non-isomorphic
			to $\mathfrak{sl}_2 \times \mathfrak{sl}_2 \times \mathfrak{sl}_2$.
			\item $m + 3s$ is even: Again using Remark~\ref{rem.product_SL_2} and Proposition~\ref{prop.product_SL_2} we may assume that $s-2 > m$. Similarly as
			above we get $0 \leq m \leq \min\{10-3s, s-3\}$. Hence, $(s, m) = (3, 0)$, and since 
			$m + 3s$ is even, we arrive at a contradiction.\qedhere
		\end{itemize}
	\end{proof}

\section{Non-embedability results for algebraic groups}
\label{sec:non-embed}

Recall from the last section that, for all simple algebraic groups $G$ and smooth affine
varieties $Z$ such that $\dim G \geq 2 \dim Z + 2$, there exists an embedding of $Z$ into $G$
(see Theorem~\ref{thm.simple}).
In this section, for every algebraic group $G$ and every integer
$d$ such that $\dim G \leq 2 d$, we construct a smooth affine variety $Z$ of dimension $d$ such that $Z$
does not allow an embedding into $G$ (see Corollary~\ref{cor.non-emeddability} below).
Thus, for a simple algebraic group $G$ this gives optimality of 
our embedding result (Theorem~\ref{thm.simple}) in case $\dim G$ is even, and optimality
up to one dimension in case $\dim G$ is odd. We will focus more on this last case in Section~\ref{chp.limits_of_our_methods}. 

\smallskip

We recall some facts of the Segre- and Chern class operations. For this we use the excellent book
of Fulton~\cite{Fu1998Intersection-theor} as a reference.
For a smooth irreducible 
variety $X$ of dimension $d$ we denote
by $\CH_i(X)$ its $i$-th Chow group, i.e.~the group of $i$-cycles modulo linear equivalence
for each $0 \leq i \leq d$. For $i > d$ and $i < 0$ we set 
$\CH_i(X) = 0$.
For each vector bundle $E \to X$  and each $i \geq 0$, we get the so-called 
\emph{Segre class operations}
\[
s_i(E) \colon \CH_{k}(X) \to \CH_{k-i}(X) \, , \quad
\alpha \mapsto s_i(E) \cap \alpha
\]
and thus endomorphisms $s_i(E)$ on $\CH(X) = \bigoplus_{i=0}^k \CH_i(X)$
(see \cite[\S3.1]{Fu1998Intersection-theor}). 
By~\cite[Propsition 3.1(a)]{Fu1998Intersection-theor}
we have that $s_0(E) = 1$ is the identity in $\End(\CH(X))$.
Following \cite[\S3.2]{Fu1998Intersection-theor} we consider the formal
power series $s_t(E) = \sum_{i=0}^\infty s_i(E) t^i$ and define $c_t(E) = \sum_{i=0}^\infty c_i(E) t^i$
as the inverse of $s_t(E)$ inside the formal power series ring $\End(\CH(X))[[t]]$.
This makes sense since the endomorphisms $s_i(E)$, $i \geq 0$ commute pairwise
\cite[Proposition 3.1(b)]{Fu1998Intersection-theor}. It follows that $c_i(E)$ maps 
$\CH_k(X)$ into $\CH_{k-i}(X)$ and we denote the image of $\alpha \in \CH_k(X)$
under $c_i(E)$ by $c_i(E) \cap \alpha \in \CH_{k-i}(X)$.
The operations $c_i(E)$, $i \geq 0$ are called \emph{Chern class operations}.
 Moreover, by 
\cite[Example~8.1.6]{Fu1998Intersection-theor} we have
\[
	c_i(E) \cap (c_j(E) \cap [X]) = (c_i(E) \cap [X]) \cdot (c_j(E) \cap [X]) \quad
	\textrm{for all $i, j$} \, ,
\]
where `$\cdot$' denotes the intersection product; see~\cite[\S8.1]{Fu1998Intersection-theor}.
In the sequel we denote by $T^\ast X \to X$  the cotangent bundle of $X$.

\begin{prop}
	\label{prop.existence_of_var_with s_d_neq_zero}
	For $d \geq 1$, there exists an irreducible smooth affine variety $Z$ of dimension $d$ 
	such that $s_d(T^*Z) \neq 0$.
\end{prop}

\begin{proof}
	By the proof of \cite[Theorem~5.8]{BlMuSz1989Zero-cycles-and-th}, 
	there exists a smooth irreducible affine variety $Z$ of dimension $d$ 
	such that the component
	in $\CH_0(Z)$ of the total Segre class of $T^\ast Z \to Z$
	is non-vanishing, $s_d(T^\ast Z) \cap [Z] \neq 0$ in $\CH_0(Z)$. This implies
	that $s_d(T^\ast Z) \neq 0$ inside $\End(\CH(Z))$.
\end{proof}


From a Theorem of Grothendieck, \cite[Remarque p.21]{Gr1958Torsion-homologiqu} or \cite[Proposition~2.8]{Br2011On-the-geometry-of} we get the following result:

\begin{prop}
	\label{prop.Chow_alg_groups}
	Let $G$ be a connected algebraic group of dimension $n$. 
	Then $\CH_i(G)$ is a torsion group for $0 \leq i \leq n-1$ and
	$\CH_n(G) = \ZZ$. \qed
\end{prop}

\begin{lemma}
	\label{lem.embedding_cond}
	Let $Z$ be an irreducible smooth affine variety of dimension $d \geq 1$.
	If there is a connected algebraic group $G$ 
	of dimension $2d$ such that
	there is an embedding $\iota \colon Z \to G$, then $s_d(T^\ast Z) = 0$.
\end{lemma}

\begin{proof}
	Since $d \geq 1$, by Proposition~\ref{prop.Chow_alg_groups}, we get that $\iota_*([Z]) \in \CH_d(G)$
	is a torsion element where $[Z] \in \CH_d(Z)$ denotes the class associated to $Z$.
	By \cite[Corollary 6.3]{Fu1998Intersection-theor} we have
	\[
	\iota^*(\iota_*([Z])) = c_d(N^\ast) \cap [Z] \in \CH_0(Z) 
	\]
	where $N^*$ denotes the conormal bundle of $Z$ in $G$.
	Hence $\iota^*(\iota_*([Z]))$ is a torsion element in $\CH_0(Z)$. In case
	$d = 1$, we have $\dim G = 2$ and thus $G$ is solvable. In particular $\CH_1(G) = 0$.
	In case $d \geq 2$, it follows from \cite[Proposition~2.1]{BlMuSz1989Zero-cycles-and-th}
	that $\CH_0(Z)$ is torsion free. Thus in both cases $\iota^*(\iota_*([Z]))$ is zero.
	Moreover $c_d(N^*) \cap \alpha = 0$ for each $\alpha \in \CH_k(Z)$ if $k < d$.
	This implies that $c_d(N^\ast) = 0$, it is the zero endomorphism of $\CH(Z)$.
	
	Since $G$ is an algebraic group, the cotangent bundle $T^*G \to G$ is trivial.
	Moreover, we have a short exact sequence of vector bundles over $Z$:
	\[
	0 \to N^\ast \to \iota^\ast(T^*G) \to T^*Z \to 0 \, .
	\]
	Then we get
	\[
	1 = c_t(\iota^\ast(T^*G)) = c_t(N^\ast) c_t(T^*Z) \quad \textrm{inside $\End(\CH(Z)$)[[t]]}
	\]
	by \cite[Theorem~3.2(e)]{Fu1998Intersection-theor}. By definition we get $s_t(T^*Z) = c_t(N^*)$
	and thus $s_d(T^*Z) = c_d(N^*) = 0$. 
\end{proof}

Now, we apply the above results in order to get irreducible smooth affine varieties that
do not admit an embedding into algebraic groups for appropriate dimensions.

\begin{corollary}
	\label{cor.non-emeddability}
	Let $G$ be an algebraic group of dimension $n > 0$. Then, for each integer $d \geq \frac{n}{2}$
	there exists a smooth irreducible affine variety $Z$ of dimension $d$ that does not admit
	an embedding into $G$.
\end{corollary}

\begin{proof}
	By assumption $2d \geq n$. Let $k \coloneqq 2d -n \geq 0$.
	By Proposition~\ref{prop.existence_of_var_with s_d_neq_zero}
	there exists a smooth irreducible affine variety $Z$ of dimension $d$ such that $s_d(T^* Z) \neq 0$.
	Towards a contradiction, assume that $Z$ allows an embedding into $G$. 
	As $Z$ is irreducible,
	there exists an embedding of $Z$ into the identity
	component  $G^\circ$ of $G$ and hence also 
	into $G^\circ \times (\GG_a)^k$. Since $\dim G^\circ + k = n + 2d - n = 2d$,
	by Lemma~~\ref{lem.embedding_cond} we get $s_d(T^*Z) = 0$, contradiction.
\end{proof}


\section{Limits of our methods for odd dimensional simple groups}
\label{chp.limits_of_our_methods}
In Section~\ref{sec:non-embed} we proved that Theorem~\ref{thm.simple} is optimal for even dimensional simple algebraic groups $G$.
Moreover, by Proposition~\ref{prop.lowdimension} 
we also get optimality in case $\dim G \leq 8$.
In this section we will explain, why we are not able to apply our
method to an odd dimensional simple algebraic group $G$ and smooth affine varieties $Z$
with $\dim G = 2 \dim Z + 1$ and $\dim Z > 1$.

Concretely, let $G$ be an odd dimensional simple algebraic group. In order to apply our 
method (Theorem~\ref{thm.Product}) 
to a smooth affine variety $Z$ with $\dim G = 2 \dim Z + 1$ we need at least the following:
a smooth morphism
\[
	\pi \colon G \to P \quad \textrm{with} \quad \textrm{$\dim P = \dim Z$}
\]
that factors through a principal $\GG_a$-bundle,
$\Aut^{\alg}_P(G)$ acts sufficiently transitively on each fiber of $\pi$, and a finite surjective morphism $Z \to P$. 

The only way to construct such
a $\pi \colon G \to P$ seems to be forming the algebraic quotient by 
some proper connected characterless algebraic subgroup $H \subset G$ of the right
dimension; see Proposition~\ref{prop.characterization_suff_trans_algebraic_groups} 
and Proposition~\ref{prop.sufficiently_transitive_on_fibers_quotient}. 
However, in this
section we prove Proposition~\ref{prop:intronomapsZtoG/H}  which yields an 
obstruction to the existence of proper surjective
morphisms $Z \to G/H$; see also the discussion in the introduction. 

\bigskip

Since the obstruction comes from algebraic topology, in this section we work 
with varieties  over the complex numbers, i.e.~our ground field will be $\CC$. However, using 
an appropriate
Lefschetz principle, we promote a version of Proposition~\ref{prop:intronomapsZtoG/H} 
back to every algebraically closed field of characteristic zero;
see Appendix~\ref{Appendix.Lefschetz}.

In order to avoid confusion with the category of complex manifolds, below we write \emph{algebraic
morphism} instead of just \emph{morphism}. We restate Proposition~\ref{prop:intronomapsZtoG/H}:

\begin{prop}
	\label{prop.no-finite-morph}
	Let $Z$ be a {simply-connected} complex smooth algebraic variety with the rational homology of a point. If $G/H$ is a $\dim Z$-dimensional complex homogeneous space of a complex simple algebraic group $G$, then there is no proper surjective algebraic
	morphism from $Z$ to $G/H$. 
\end{prop}

In Appendix~\ref{Appendix.Lefschetz} we prove that Proposition~\ref{prop.no-finite-morph} holds
for $Z = \AA^{\dim G - \dim H}$
over any algebraically closed field of characteristic zero; see Proposition~\ref{prop:intronomapsZtoG/H_over_any_k}.

%

\begin{proof}[Proof of Proposition~\ref{prop.no-finite-morph}]
	Let $\tilde{G}$ be the universal cover of $G$.
		Then $p \colon \tilde{G} \to G$ is a homomorphism of
		simple complex algebraic groups
		\cite[Th\'eor\`eme 5.1,  Expos\'e XII]{GrRa2003Revetements-etales}. 
		Since $\tilde{G} / p^{-1}(H)$ and $G/ H$ are isomorphic as algebraic varieties, we
		may assume that $G$ is simply connected.
	
	Assume that there exists a proper surjective algebraic morphism $Z \to G/H$. 
	Let $H^\circ$ be the identity 
	component of $H$.  Denote by 
	$p \colon G/H^\circ \to G/H$ the canonical projection, 
	which is a finite algebraic \'etale surjection.	
	As $Z$ is simply connected, there exists a holomorphic map $f \colon Z  \to G / H^\circ$
	such that $p \circ f \colon Z \to G / H$ is the original proper surjective algebraic 
	morphism. By \cite[Proposition~20]{Se1958Espaces-fibres-alg}, it follows that 
	$f \colon Z \to G / H^\circ$ is an algebraic morphism, and it is also proper and surjective. 
	Thus, by replacing
	$H$ by $H^\circ$, we may assume without loss of generality that $H$ is connected.

	Since $G$ is simply connected and $H$ is connected, the 
	long exact homotopy sequence assocaited to $H \hookrightarrow G \twoheadrightarrow G/H$
	yields the exact sequence
	\[
		1 = \pi_1(G) \to \pi_1(G/H) \to \pi_0(H) = 1 \, .
	\]
	Thus, since $G$ is connected, we get that $G/H$ is simply connected. Let
	\[
		i_0 \coloneqq \inf\set{i\geq 1}{
		\textrm{ $\pi_i(G/H) \otimes_{\ZZ} \QQ$ is non-vanishing}} \, .
	\]
	By Proposition~\ref{prop.homotopy_non-vanishing} below, it follows that $1 < i_0 < \infty$.
	As $G/H$ is simply connected, we may apply a rational version of the
	Hurewicz Theorem \cite[Theorem 1.1]{KlKr2004A-quick-proof-of-t} and get
	\[
		0 \neq \pi_{i_0}(G/H) \otimes_{\ZZ} \QQ \simeq H_{i_0}(G/H; \QQ)
	\]
	where $H_\ast(\cdot; \QQ)$ denotes singular homology with rational coefficients. 
 	Since $f \colon Z \to G/H$ is a proper surjective algebraic morphism, 
 	Theorem~\ref{mthm:surjectiononhomology} applies, 
	and we get that
	$f_{\ast} \colon H_{i_0}(Z; \QQ) \to H_{i_0}(G/H; \QQ)$ is surjective. However, this contradicts
	$H_{i_0}(Z; \QQ) = 0$. 
\end{proof}

\begin{prop}
	\label{prop.homotopy_non-vanishing}
	Let $G$ be a simple complex algebraic group. 
	Then, for each proper closed complex subgroup
	$H \subset G$, there exists $i > 1$ such that 
	\[
	\pi_i(G / H) \otimes_{\ZZ} \QQ \neq 0 \, .
	\]
\end{prop}

For the proof of this proposition, we use facts about the rational homotopy
groups of all simply connected simple complex algebraic groups. We recall those facts next.

Denote by $G$ a simply connected semisimple complex algebraic group. 
Recall that there exists a maximal compact connected real Lie
subgroup $K \subset G$ such that $G$ and $K$ are homotopy equivalent
\cite[Theorem 2.2, Ch.~VI]{He1978Differential-geome}. In particular, 
$K$ is simply connected, and thus
we may apply \cite[Theorem 6.27, Ch.~IV]{MiTo1991Topology-of-Lie-gr} to get
a continuous map of a product of odd dimensional spheres into $K$
\[
	f \colon S^{2n_1-1} \times \cdots \times S^{2n_l-1} \to K
\]
that induces an isomorphism between the singular cohomology rings with rational coefficients
\[
	H^\ast(K; \QQ) \simeq H^\ast(S^{2n_1-1} \times \cdots \times S^{2n_l-1}; \QQ) \, .
\]
By the universal coefficient theorem for cohomology, 
$f$ induces an isomorphism between singular homology groups with rational coefficients.
Since $K$ is simply connected, we get by K\"unneth's formula
\[
H_1(S^{2n_1-1}; \QQ) \oplus \cdots \oplus H_1(S^{2n_l-1}; \QQ) \simeq H_1(K; \QQ) = 0 \, .
\]
This implies $n_i \geq 2$ for each $i \in \{1, \ldots, l\}$. In particular, the product of spheres
$S^{2n_1-1} \times \cdots \times S^{2n_l-1}$ is simply connected as well.
Now, by the Whitehead-Serre Theorem \cite[Theorem 8.6]{FeHaTh2001Rational-homotopy-}, 
$f$ induces for each $i \geq 0$ an isomorphism of rational homotopy groups
\begin{equation}
\label{Eq.homotpy_semisimple}
\pi_i(S^{2n_1-1}) \otimes_\ZZ \QQ \times \cdots \times
\pi_i(S^{2n_l-1}) \otimes_\ZZ \QQ \simeq \pi_i(K) \otimes_\ZZ \QQ = \pi_i(G) \otimes_\ZZ \QQ  \, .
\end{equation}
Note that by a Theorem of Serre 
(\cite[Example 1 in \S15(d)]{FeHaTh2001Rational-homotopy-} or \cite[Theorem~1.3]{KlKr2004A-quick-proof-of-t}, for
odd positive integers $k$, the group
$\pi_i(S^k) \otimes_\ZZ \QQ$ is isomorphic to $\QQ$ if $i =k$
and otherwise it vanishes.

\begin{definition}
	For a simply connected semisimple complex algebraic group $G$, we call the above
	constructed unordered $l$-tuple $\{2n_1-1, \ldots, 2n_l-1\}$ the \emph{rational homotopy type} of $G$.
\end{definition}

In the following table we list the complex dimension and rational
homotopy type for each Lie type 
(the statements follow from 
\cite[Theorem 6.5, Ch.~III and Theorem 5.10, Ch.~VI]{MiTo1991Topology-of-Lie-gr}):

\begin{center} 
	\captionof{table}{Rational homotopy types} 
	\label{table1}
	\begin{tabular}{r|r|r}
		Lie-Type  & Complex dimension  & Rational homotopy type of the \\
		& & simply connected simple \\
		& & complex algebraic group \\
		\hline
		$A_m$, $m \geq 1$ & $m^2 + 2m$ & $\{3, 5, \ldots, 2m+1\}$ \\
		$B_m$, $m \geq 2$ & $2m^2 +m$ & $\{3, 7, \ldots, 4m-1 \}$ \\
		$C_m$, $m \geq 3$ & $2m^2 +m$ & $\{3, 7, \ldots, 4m-1 \}$ \\
		$D_m$, $m \geq 4$ & $2m^2 -m$ & $\{3, 7, \ldots, 4m-5 \} \cup \{ 2m-1 \}$ \\
		$E_6$ & $78$ & $\{3, 9, 11, 15, 17, 23\}$ \\
		$E_7$ & $133$ & $\{3, 11, 15, 19, 23, 27, 35\}$ \\
		$E_8$ & $248$ & $\{3, 15, 23, 27, 35, 39, 47, 59\}$ \\
		$F_4$ & $52$ & $\{3, 11, 15, 23\}$ \\
		$G_2$ & $14$ & $\{3, 11\}$
	\end{tabular}
\end{center}

\begin{proof}[Proof of Proposition~\ref{prop.homotopy_non-vanishing}]
	With the same argument as in the beginning of the proof of Proposition~\ref{prop.no-finite-morph}, 
	we may assume that $G$ is simply connected.
	Let $H^\circ \subset H$ be the identity component of $H$.
	Since $G / H^\circ \to G/H$ is a finite \'etale surjection, we get for each
	$i > 1$ an isomorphism  $\pi_i(G/H^\circ) \simeq \pi_i(G/H)$. Hence, in addition
	we may assume that $H$ is connected.
	
	Let $R(H)$ be the radical of $H$.
	By definition $H / R(H)$ is a semisimple complex algebraic group.
	Let $S \to H/R(H)$ be the universal covering. As before, $S$ is a simply connected
	semisimple complex algebraic group.
	Since $R(H)$ is the product of an algebraic torus and a 
	unipotent algebraic group, it follows from the long exact homotopy sequence that
	\[
		\pi_i(H) = \pi_i(H/R(H)) = \pi_i(S) \quad \textrm{for each $i > 2$}
	\]
	and
	\[
		\pi_2(H) \hookrightarrow \pi_2(H/R(H)) = \pi_2(S)
	\] 
	is injective.
	From~\eqref{Eq.homotpy_semisimple} it follows that $\pi_2(S) \otimes_{\ZZ} \QQ = 0$.
	Hence we get
	\[
		\pi_i(H) \otimes_{\ZZ} \QQ = \pi_i(S) \otimes_{\ZZ} \QQ \quad
		\textrm{for each $i > 1$} \, .
	\]
	Let $S_1, \ldots, S_l$ be the connected normal minimal closed
	complex subgroups of $S$.
	Then each $S_i$ is a simple complex algebraic group and the product morphism
	$S_1 \times \cdots \times S_l \to S$ is a finite \'etale surjection 
	(see \cite[Theorem in \S27.5]{Hu1975Linear-algebraic-g}). Hence,
	we get
	\[
	\pi_i(S) = \pi_i(S_1) \times \cdots \times \pi_i(S_l) \quad \textrm{for each $i > 1$} \, .
	\]
	
	Now, assume towards a contradiction that $\pi_i(G/ H) \otimes_{\ZZ} \QQ = 0$ 
	for each $i > 1$. By tensoring the long exact homotopy sequence associated to 
	$H \hookrightarrow G \twoheadrightarrow G / H$ with $\QQ$, we get isomorphisms
	\[
	\pi_i(H) \otimes_\ZZ \QQ \simeq \pi_i(G) \otimes_\ZZ \QQ 
	\quad \textrm{for each $i > 1$} \, .
	\] 
	In particular,
	\begin{equation}
	\label{eq:homotopy_iso_1}
	\pi_3(S_1) \otimes_\ZZ \QQ \times \cdots \times \pi_3(S_l) \otimes_{\ZZ} \QQ  \simeq
	\pi_3(G) \otimes_\ZZ \QQ \, .
	\end{equation}
	According to Table~\ref{table1}, we have 
	$\pi_3(S_i) \otimes_\ZZ \QQ \simeq \pi_3(G) \otimes_\ZZ \QQ \simeq \QQ$ for each
	$i \in \{1, \ldots, l\}$. Hence, due to~\eqref{eq:homotopy_iso_1}, we get $l = 1$,
	$S$ is already simple, and
	\[
	\pi_i(S) \otimes_\ZZ \QQ \simeq \pi_i(G) \otimes_\ZZ \QQ 
	\quad \textrm{for each $i > 1$} \, .
	\]
	Since $S$ and $G$ are both simply connected we get even
	\[
	\pi_i(S) \otimes_\ZZ \QQ \simeq \pi_i(G) \otimes_\ZZ \QQ 
	\quad \textrm{for each $i \geq 0$} \, .
	\]
	This implies that $S$ and $G$ have the same rational homotopy type.
	However, according to Table~\ref{table1} this can only happen if the Lie types of
	$S$ and $G$ coincide or the Lie types of $S$ and $G$ are $B_m$ and $C_m$, respectively,
	for some $m \geq 3$ (note that the \nth{5} rational homotopy group
	is non-vanishing only for $A_m$ with $m \geq 2$ 
	and the \nth{7} rational homotopy group
	is non-vanishing only for $B_m$, $C_m$ and $D_m$). In both cases
	the complex dimension of 
	$S$ and $G$ coincide, which contradicts the fact that $H$ is a proper
	closed complex subgroup of $G$.
\end{proof}

\appendix

\addtocontents{toc}{\protect\setcounter{tocdepth}{1}}

\section{Hopf's Umkehrungshomomorphismus theorem
}
\label{AppendixA}
In this chapter we use a version of Hopf's Umkehrungshomomorphismus theorem to prove Theorem~\ref{mthm:surjectiononhomology}.
Apart from the proof of Theorem~\ref{mthm:surjectiononhomology}, in this section we consider the Euclidean topology on (topological, smooth and complex) manifolds and subsets thereof.
For lack of reference, we provide a proof of the following version for (in general) non-closed manifolds of a result going back to the work of Hopf in the case of closed (smooth) manifolds~\cite{Hopf_30}. While we will apply the result only for smooth maps, we take the opportunity to formulate the statements for topological manifolds and proper continuous maps between them. Aspects of our proof are written with smooth concepts in mind (definition of degree, exhaustion of manifolds by full-dimensional compact manifolds with boundary), even if the proficient topologist might have worked differently, e.g.~to avoid topological transversality in Lemma~\ref{lem:nicecompacts}. An advantage is that this proof works very naturally in the smooth setup as well, and it seems like the fastest path from citable literature to the theorem.

The reader may read what follows for the ring $R$ being $\Z$ or $\Q$ without loss for the application in this paper. Recall that an orientation on a manifold is a $\Z$-orientation. The notions used in the result will be explained afterwards.

\begin{theorem}\label{thm:Hopfdegree}Let $R$ be a commutative unital ring, $M$ and $N$ be $R$-oriented non-empty topological manifolds of the same dimension where $N$ is connected, and let $f\colon M\to N$ be a proper continuous map. 
Denote by $d\in R$ the degree of $f$, by $f_k\colon H_k(M;R)\to H_k(N;R)$ the induced map in $k$-th homology, and by $f_{!,k}: H_k(N;R)\to H_k(M;R)$ the Umkehrungshomomorphismus.
Then, for all non-negative integers $k$ and all $c\in H_k(N;R)$, we have $f_k\circ f_{!,k}(c)=dc$.
\end{theorem}

We use Theorem~\ref{thm:Hopfdegree} to prove Theorem~\ref{mthm:surjectiononhomology}. In fact we prove the following.

\begin{theorem}\label{thm:THMCgeneral} Let $f\colon X\to Y$ be a proper surjective
	holomorphic map between complex $n$-dimensional manifolds. 
	Assume that $Y$ is connected and let the integer $d\geq 1$ be the number of preimages of a regular value of $f$. Then the following hold. 
\begin{enumerate}[leftmargin=*, label=(\alph*)]
\item \label{thm:THMCgeneral_a} The image of the induced map on $k$-th homology $H_k(X;R)\to H_k(Y;R)$ contains $dH_k(Y;R)$ for all
integers $k \geq 0$.
\item \label{thm:THMCgeneral_b} Assume that $X$ is connected. 
Then for all $x$ in $X$, the image of the induced homomorphism 
on the fundamental groups $f_*\colon \pi_1(X,x)\to \pi_1(Y,f(x))$ has finite index
in $\pi_1(Y,f(x))$ and this index divides $d$. 
In case $X$ has the rational homology of a point, then $f_*$ is a surjection.
\end{enumerate}
\end{theorem}

%

\begin{proof}[Proof of Theorem~\ref{mthm:surjectiononhomology}]
Since $f$ is proper (in the sense of algebraic geometry), it is proper as a map when 
$X$ and $Y$ are endowed with the Euclidean topology (i.e.~their topology as (complex) 
differentiable manifolds); see~\cite[Proposition 3.2, Exp. XII]{GrRa2003Revetements-etales}, \cite[Proposition~6, \S10]{Bo1971Elements-de-mathem}.
From here on we consider $X$ and $Y$ with their Euclidean topology. 
W.l.o.G.~$X$ and $Y$ are connected.
Since $d\neq 0$ is a unit in $\Q$, the statement follows by applying
Theorem~\ref{thm:THMCgeneral}\ref{thm:THMCgeneral_a}
for $R=\Q$.
\end{proof}
\begin{proof}[Proof of Theorem~\ref{thm:THMCgeneral}]


Since $X$ and $Y$ are complex manifolds, they are canonically oriented, and since $f$ is
holomorphic, for every regular point $x\in X$, $f$ maps a neighborhood orientation-preservingly to $Y$.
Consequently, the local degree of $f$ at a regular value $y\in Y$ (see Remark~\ref{rem:locdegree} for a topological definition) equals the non-negative integer $d$ of the number of elements of $f^{-1}(y)$.

The degree of $f$ is well-defined as the local degree $d\geq 0$ of any regular value $y\in Y$, and,
since preimages of regular values are non-empty, the degree of $f$ is non-zero.

\ref{thm:THMCgeneral_a}: We apply Theorem~\ref{thm:Hopfdegree} to find that
\[
	f_k(H_k(X;R))
	\supseteq f_k(f_{!,k}(H_k(Y;R)))=dH_k(Y;R) \, .
\]

\ref{thm:THMCgeneral_b}: W.l.o.g., we fix $x\in X$ such that $y\coloneqq f(x)\in Y$ is a regular value of $f$.
Let $p\colon\widetilde{Y} \to Y$ be the covering 
of $Y$ corresponding to the subgroup $f_*(\pi_1(X,x))\subseteq \pi_1(Y,y)$. We show that this is a finite cover, whence $f_*(\pi_1(X,x))$ has finite index in $\pi_1(Y,y)$.

We pick $\widetilde{y}\in p^{-1}(y)$ and denote by $\widetilde{f}\colon X\to \widetilde{Y}$ a lift of $f$ with $\widetilde{f}(x)=\widetilde{y}$.

First we show that $p^{-1}(y)$ is contained in the image of $\widetilde{f}$.
 Indeed, take $\widetilde{z}\in p^{-1}(y)$ and let $\widetilde{\beta}\colon [0,1]\to \widetilde{Y}$ be a path connecting $\widetilde{y}$ and $\widetilde{z}$. We arrange for $\widetilde{\beta}$ to lie in $p^{-1}(Y^\mathrm{reg})$, where $Y^\mathrm{reg}\subseteq Y$ denotes the subset of regular values of $f$. This can, for example, be achieved by composing $\tilde{\beta}$ with $p$, homotoping the resulting path rel endpoints  into $Y^\mathrm{reg}$ (here we invoke that $f$ is 
 holomorphic\footnote{
We provide an argument for the claim that any smooth path $\beta\colon[0,1]\to Y$ with endpoints in $Y^{\mathrm{reg}}$ can be homotoped relative endpoints to a smooth 
path in $Y^{\mathrm{reg}}$.
 
 Note that $Y^\mathrm{sing}\coloneqq Y\setminus Y^\mathrm{reg}$ is the image $f(X^{\mathrm{sing}})$ of the closed analytic subset $X^{\mathrm{sing}}=\set{x\in X}{f\text{ is singular at } x}\subseteq X$,
 and thus $Y^\mathrm{sing}$ is a closed analytic subset of $Y$ by Remmert's Proper Mapping Theorem~\cite[Satz~23]{Re_57}. As such, $Y^\mathrm{sing}$ can be stratified as a finite union $M_1\dot\cup\cdots\dot\cup M_{k}$ of complex submanifolds $M_i \subseteq Y$
 of complex codimesion at least $1$; 
 see e.g.~\cite[\S5.5.~Stratifications]{Chi_89}.
Let $G \colon [0, 1] \times \RR^N \to Y$ be a smooth map for some $N\in\N$ with 
 $G |_{[0, 1] \times \{0\}} = \beta$ such that
 for each $t \in [0, 1]$, the map $\RR^N \to Y$, $v \mapsto G(t, v)$ is submersive (see e.g.~\cite[Corollary to the $\varepsilon$-Neighbourhood Theorem]{GuPo_74}). Let $F \colon [0, 1] \times \RR^N \to Y$
 be given by $F(t,v)\coloneqq G(t,t(1-t)v)$. Then, 
 $F$ and $\partial F \colon \partial [0, 1] \times \RR^N \to Y$ are both transversal to each submanifold $M_i$ of $Y$.
 By Thom's Transversality Theorem, 
 $\beta_v \colon [0, 1] \to Y$, $t \mapsto F(t, v)$ 
 is transversal to $Y^{\mathrm{sing}}$ for each $v \in \RR^N$ away from a nullset.
As the $M_i$ have real codimension at least $2$ in $Y$, 
this means that the image of $\beta_v$ is disjoint from each $M_i$. We set $\beta'=\beta_v$ for some $v$ not in that nullset.}),
 and lifting the resulting path.
 The loop $\beta\coloneqq p\circ \widetilde{\beta}$ can be lifted to a path $\alpha\colon[0,1]\to f^{-1}(Y^\mathrm{reg})$ starting at $x$ since $f$ restricts to a covering $f^{-1}(Y^\mathrm{reg})\to Y^\mathrm{reg}$ (recall that proper local homeomorphisms 
 are coverings). By construction, $\widetilde{f}\circ\alpha$ and $\widetilde{\beta}$ are lifts of $\beta$ starting at $\widetilde{y}$; in particular, $\widetilde{f}(\alpha(1))=\widetilde{z}$ as desired.
 
 We conclude that $p^{-1}(y)$, which is the index of 
 $f_*(\pi_1(X,x))$ in $\pi_1(Y,y)$, must be finite. In fact,
   \[\left|p^{-1}(y)\right|\left|\widetilde{f}^{-1}(\widetilde{y})\right|=\left|\left(p\circ \widetilde{f}\right)^{-1}(y)\right|=\left|f^{-1}(y)\right|
 =d\in\N,\]
 where the first equality follows since $\tilde{f}$ restricts to a covering $f^{-1}(Y^{\mathrm{reg}})\to p^{-1}(Y^{\mathrm{reg}})$ and, thus, $|\tilde{f}^{-1}(y')|=|\tilde{f}^{-1}(\widetilde{{y}})|$ for all $y'\in p^{-1}(y)$.
 
 Assume now, that $X$ has the rational homology of a point.
 Note that $\widetilde{f}$ is a holomorphic map between complex $n$-dimensional manifolds; hence, its degree $\widetilde{d}$ equals $\widetilde{f}^{-1}(\widetilde{y})$ by the same argument we used above to find $d=f^{-1}(y)$. Hence, since $f$ and $\widetilde{f}$ have non-zero-degree, Theorem~\ref{thm:Hopfdegree} implies that they both induce surjections on rational homology.
 Hence, both $Y$ and $\widetilde{Y}$ have the rational homology of a point and, in particular, they both have Euler characteristic 1. However, for a finite covering $p\colon\widetilde{Y}\to Y$ of index $k$, the Euler characteristic of $\widetilde{Y}$ is $k$-times that of $Y$, thus $k=1$.
 \end{proof}

\begin{remark}An $n$-dimensional manifold $M$ is said to \emph{dominate} an $n$-dimensional connected
manifold $N$, if there exists a proper continuous map $f\colon M\to N$ of non-zero degree. 
Using this term, the above proof of Theorem~\ref{thm:THMCgeneral}\ref{thm:THMCgeneral_a}
amounts to observing that  a proper surjective holomorphic map between complex $n$-dimensional 
manifolds is a map that establishes that the domain dominates the target and then applying Theorem~\ref{thm:Hopfdegree}.
%
\end{remark}

\begin{remark}\label{rmk:Gurjar}
Only after a preprint of this article appeared on the arXiv, the authors became aware of Gurjar's result~\cite{Gu1980Topology-of-affine}. This was the motivation 
to add~\ref{thm:THMCgeneral_b} 
to Theorem~\ref{thm:THMCgeneral}, so that Theorem~\ref{thm:THMCgeneral} specializes to Gurjar's result by setting $X=\C^n$. Our proof of part~\ref{thm:THMCgeneral_b} can also be understood as the natural generalization of the argument from~\cite{Gu1980Topology-of-affine}.
\end{remark}

Before providing the proof of Theorem~\ref{thm:Hopfdegree}, we recall orientations, dualities, the Umkehrungshomomorphismus, and the degree. We do this somewhat detailed and in a for us suitable way since we need all notions to work for non-compact manifolds. We take~\cite{hatcher_AT} as our reference for algebraic topology.

For readability, we will drop the coefficients from the notation of homology and cohomology. 

\subsection*{Manifold} A topological manifold, short manifold, of dimension $n$ is a second countable Hausdorff space locally homeomorphic to $\R^n$. In particular, manifolds have no boundary unless otherwise stated. A manifold is said to be closed if it is compact.
\subsection*{Orientation} An $R$-orientation is a map $o\colon M\to \bigcup_{x\in M} H_n(M, M\setminus\{x\})$ such that $o(x)\in H_n(M, M\setminus\{x\})\simeq R$ is a generator 
(i.e.~$Ro(x)=H_n(M, M\setminus\{x\})$) and $o$ is continuous.
Here, $\bigcup_{x\in M} H_n(M, M\setminus\{x\})$ is endowed with the following topology, which turns the canonical projection to $M$ into a covering map and $o$ into a section of this covering map. The topology is the inductive limit topology
	with respect to the maps
	\[
		\xymatrix{
			\RR^n \times R \ar[rr]^-{(x, r) \mapsto r \varepsilon_x(\mu)} && 
			\bigcup_{y\in \RR^n} H_n(\RR^n, \RR^n\setminus\{y\}) \ar[r]^-{\phi_\ast} &
			\bigcup_{x\in M} H_n(M, M\setminus\{x\})
		}
	\]
	for all local charts $\phi \colon \RR^n \to U$ where $R$ carries the discrete topology,
	$\mu \in H_n(\RR^n, \RR^n \setminus \{0\})$
	is a fixed generator and 
	$\varepsilon_x \colon H_n(\RR^n, \RR^n \setminus \{0\}) \to H_n(\RR^n, \RR^n \setminus \{x\})$
	is induced by the translation $\RR^n \to \RR^n$, $y \mapsto y+x$; see \cite[$R$-orientation]{hatcher_AT}).

For every compact $K\subset M$, we denote by $o_K\in H_n(M,M\setminus K)$ the unique element in $H_n(M,M\setminus K)$ that maps to $o(x)$ under the map induced by inclusion of pairs $(M,M\setminus K)\subset (M,M\setminus \{x\})$ for all $x\in K$; see \cite[Lemma~3.27]{hatcher_AT}.

For context, recall that for a closed oriented manifold $M$, $o_M$ 
is the fundamental class in $H_n(M)$.

\subsection*{Cohomology groups with compact supports as a limit}
For any topological space $X$, we denote by $H^k_{\c}(X)$ \emph{cohomology groups with compact supports}; that is, the limit group of the directed system of groups given by the groups and maps
\[
	\left\{H^k(X,X\setminus K)\right\}_{K\subset X\text{, $K$ compact}}\et H^k(X,X\setminus K)\overset{\iota^*}{\to}H^k(X,X\setminus L),
\] 
where $K\subseteq L\subseteq X$ are compacts and $i:(X,X\setminus L)\to (X,X\setminus K)$ denotes the inclusions of pairs, respectively; see \cite[Paragraph after Prop~3.33]{hatcher_AT}.

This yields a functor from the category of topological space with morphisms given by \emph{proper} continuous maps to the category of $R$-modules for each non-negative integer $k$: to a proper continuous map $f\colon X\to Y$ we associate
\[
	f_{\c}^k\colon H^k_\c(Y)\to H^k_\c(X) \, , \quad [\phi] \mapsto [f^k(\phi)\in H^k(X,X\setminus f^{-1}(J))] \, ,
\] 
where $\phi\in H^k(Y,Y\setminus J)$ for some compact subset $J\subseteq Y$, and we denote by $f^k\colon H^k(Y,Y\setminus J)\to H^k(X,X\setminus f^{-1}(J))$ the homomorphism induced by $f\colon (X,X\setminus f^{-1}(J))\to (Y,Y\setminus J)$.

\subsection*{Poincar\'e duality and the Umkehrungshomomorphismus}
We recall that for an $R$-oriented topological manifolds $M$ we have the Poincare duality isomorphism.
One can write the \emph{Poincar\'e duality} map
\[
	PD_k(M)\colon H^{n-k}_\c(M)\to H_{k}(M)
\]
as the homomorphism induced by
\[
	H^{n-k}(M, M \setminus K) \to H_k(M) \, , \quad  \psi \mapsto o_K\cap \psi
\]
for all compact subsets $K$ in $X$; see~\cite[Theorem~3.35]{hatcher_AT}.

Correspondingly, for all non-negative integers $k$, one defines the \emph{Um\-kehr\-ungs\-homo\-morphismus} in homology of a proper continuous map $f\colon M\to N$ between $R$-oriented $n$-manifolds as
\[f_{!,k}\coloneqq PD_k(M)\circ f_\c^{n-k}\circ (PD_k(N))^{-1}\colon H_k(N)\to H_k(M).\]

\subsection*{Alexander duality}
For an $R$-orientable manifold $M$ and a locally contractible compact path-connected subset $K\subset M$, one has the following
\begin{equation}\label{eq:Kconnectedandloccotractible}
H_l(M,M\setminus K)\simeq H^{n-l}(K) \text{ for all $l\in\{0,\ldots, n\}$},
\end{equation}
which we only use for $l=n$ and $K$ path-connected, so that $H^{n-l}(K)\simeq R$. 

\begin{proof}[Proof of~\eqref{eq:Kconnectedandloccotractible}]
Let $K$ be a compact in an $R$-orientable $n$-dimensional manifold~$M$.
If $M$ is closed, see~\cite[Theorem 3.44]{hatcher_AT} for a proof (the proof given there works as stated for every $R$).

If instead $M$ is not closed (i.e.~$M$ is not compact), we find a compact $n$-dimensional submanifold $M_0\subset M$ with boundary such that $K$ is contained in the interior of $M_0^\circ$ of $M_0$; see Lemma~\ref{lem:nicecompacts} below. Now~\eqref{eq:Kconnectedandloccotractible} follows
from the case above since by excision
\[
	H_l(M,M\setminus K)\simeq H_l(M_0^\circ, M_0^\circ
	\setminus K) \simeq H_l\left(M_0\cup\textrm{id}_{\partial M_0}{M_0},M_0\cup\textrm{id}_{\partial M_0}{ \overline{M_0}}\setminus K\right),
\] 
where $M_0\cup\textrm{id}_{\partial M_0}{M_0}$ denotes the doubling of $M_0$, i.e.~the closed $R$-orientable $n$-manifold obtained by gluing $M_0$ to a copy of itself along their boundary via the identity.
\end{proof}

\subsection*{Every compact sits in a compact submanifold}
The following lemma was used above to assure that Alexander duality holds for non-compact manifolds. We will also use it below for degree calculations.
\begin{lemma}
\label{lem:nicecompacts}
Let $M$ be an $n$-dimensional manifold. If $K\subset M$ is a compact subset, then there exists a compact $n$-dimensional manifold $M_0$ (with boundary if $K$ has non-empty intersection with at least one non-compact component of $M$) such that the interior of $M_0$ contains $K$. 
If $M$ is connected, then $M_0$ can be chosen to be path-connected.
\end{lemma}
\begin{proof}
If $K$ has empty intersection with all non-compact components of $M$, set $M_0$ to be the union of connected components of $M$ that have non-empty intersection with $K$. Hence, we consider the case that $K$ has non-empty intersection with at least one non-compact component of $M$ (in particular, $M$ is non-compact).

Pick a proper continuous map $f\colon M \to \R$. (For example, exhaust $M$ by a countable union of compacts $K_1\subset K_2\subset\cdots$ with $K_i\subseteq K_{i+1}^\circ$ (possible by second countability), define $f$ to be $i$ on the compacts $K_{i}\setminus K_i^\circ$, and extend it to map $K_{i+1}\setminus K_i^\circ$ to $[i,i+1]$ by the Tietze extension theorem.)

Let $a,b\in \R$ be such that $a+1<f(x)<b-1$ for all $x$ in $K$. Up to changing $f$ by a homotopy that is constant outside of the compact $f^{-1}([a-1,a+1]\cup [b-1,b+1]$ (in particular, the resulting $f$ stays proper), we may and do assume that $f$ is transversal to $a$ and $b$, which in particular implies that $M_0\coloneqq f^{-1}([a,b])$ is a compact manifold  with boundary $f^{-1}(a)\cup f^{-1}(b)$; see~\cite[Definition 10.7 and Theorem 10.8]{FNOP} for necessary definitions and statements.\footnote{We abstain from providing the details of topological transversality (details and further references can be found in~\cite{FNOP}). We note that in the rest of the paper we use this appendix only for smooth manifolds, and the proof is written such that replacing $f$ by a smooth map the argument works with the notion of transversality and the corresponding transversality theorems in smooth manifold theory.}

In case $M$ is connected, one can easily arrange for $M_0$ to be connected.
Indeed, let $L$ be the union of $M_0$ with the image of (finitely many) paths in $M$ between components of $M_0$. Thus $L$ is a path-connected compact subset of $M$ that contains the original $K$. 
Find a compact submanifold (with boundary) of dimension $n$ of $M$ that contains $L$ (as done in the previous paragraph) and take its connected component that contains $L$.
\end{proof}



\subsection*{Degree} Let $f\colon M\to N$ be a proper continuous map, where $M$ and $N$ are $R$-oriented $n$-manifolds. 

For $y\in N$, we set $K\coloneqq f^{-1}(y)$ and consider the induced map 
\[
	f_n\colon H_n(M, M \backslash f^{-1}(y))\to H_n(N, N \backslash \{y \})=Ro(y) \simeq R \, .
\] 
We define the \emph{local degree} $d_y$ of $f$ at a point $y\in N$ as the unique $d_y\in R$ such that $f_n(o_K)=d_yo(y)$.

\begin{remark}[Local degree for $y$ with finite preimage]\label{rem:locdegree}
In the special case that $K$ is finite, say given by pairwise distinct points $x_1,\cdots x_l$, we have that
\[d_y=\sum_{i=0}^lr(x_i),
\]
where $r(x_i)\in R$ is such that for an open neighborhood $U$ of $x_i$ with $U\cap K=\{x_i\}$ the induced map of pairs $f_n\colon H_n(M,M\setminus \{x_i\}) \simeq H_n(U,U\setminus \{x_i\})\to H_n(N,N\setminus \{y\})$ satisfies $f_n(o(x_i))=r(x_i)o(y)$.
\end{remark}

If $y_1\neq y_2$ are in the same connected component of $N$, then $d_{y_1}=d_{y_2}$. This follows from the following lemma, which is immediate from naturality of induced maps in homology of pairs.
\begin{lemma}\label{lem:degreeviagoodcompacts}
Let $f\colon M\to N$ be a proper continuous map, where $M$ and $N$ are $R$-oriented $n$-manifolds.

If $J$ is a compact subset of $N$ such that $H_n(N,N\setminus J) \simeq R$, e.g.~$J$ is path connected and locally contractible (Alexander duality; see~\eqref{eq:Kconnectedandloccotractible}), then the unique $d\in R$ such that $f_n(o_{f^{-1}(J)})=do_J$ satisfies $d=d_y$ for all $y\in J$.\qed
\end{lemma}
And, indeed, it follows that $d_{y_1}=d_{y_2}$: let $J$ be a closed arc embedded in $N$ with endpoints $y_1$ and $y_2$ (such an arc exists since connected components of manifolds are arc-connected), hence  $d_{y_1}=d_{y_2}$ by Lemma~\ref{lem:degreeviagoodcompacts}.

Hence, if $N$ is connected, the \emph{degree} $d$ of $f$ is defined to be the local degree of $f$ at a $y\in N$.




\subsection*{The proof}
\begin{proof}[Proof of Theorem~\ref{thm:Hopfdegree}] Let $f\colon M\to N$ be a proper continuous map between $R$-oriented manifolds $M$ and $N$, where $N$ is connected. Fix a non-negative integer $k$ and $c_Y\in H_k(N)$. We choose a compact $J\subset N$ such that $PD_k(N)([\psi])=c_Y$ for some $\psi\in H^{n-k}(Y,Y\setminus J)$. In fact, by increasing $J$ if necessary (and changing $\psi$ to the corresponding class given by the inclusion map), we may and do choose $J$ to be connected and locally contractible (indeed, we may choose it as a connected submanifold with boundary by Lemma~\ref{lem:nicecompacts}). We set $K\coloneqq f^{-1}(J)\subset M$, which is compact since $f$ is proper. With this setup we calculate
\begin{align}
\label{eq:deff!}f_{k}\left(f_{!,k}(c_Y)\right)&=f_k\left(PD_k(M)\circ f_\c^{n-k}\circ (PD_k(N))^{-1}(c_Y)\right)&\\
\label{eq:defphi}&=f_k \left(PD_k(M)(f_\c^{n-k}([\psi]))\right)&\\
\label{eq:homologyaslimit}&=f_k \left(PD_k(M)([f^{n-k}(\psi)])\right)&\\
\label{eq:defPD}&=f_k \left(o_K\cap f^{n-k}(\psi)\right)&\\
\label{eq:naturalityofcap}&=f_n(o_K)\cap \psi&\\
\label{eq:orientation}&=do_J\cap\psi&\\
\label{eq:defPDJ}&=dPD_k(N)([\psi])=dc_Y&,
\end{align}
where we use the following.
\eqref{eq:deff!} holds by the definition of the Um\-kehr\-ungs-homomorphism.
\eqref{eq:defphi} holds by our choice of $\psi$.
\eqref{eq:homologyaslimit} follows by the definition of the induced map on cohomology with compact support.
\eqref{eq:defPD} holds by the definition of $PD_k(M)$.
\eqref{eq:naturalityofcap} is an application of the naturality of the cap product
\[\cap\colon H_n(M,M\setminus K)\times H^{n-k}(M, M\setminus K)\to H_k(M);\] see~\cite[more general relative cap product, The Duality Theorem, p.~240]{hatcher_AT}.
For~\eqref{eq:orientation}, note that $do_J=f_n(o_K)$ by Lemma~\ref{lem:degreeviagoodcompacts}
since $K = f^{-1}(J)$ and by Alexander duality (see~\eqref{eq:Kconnectedandloccotractible})
we have $H_n(N,N\setminus J)\simeq R$ by our choice of $J$.
Finally, \eqref{eq:defPDJ} holds by the definition of $PD_k(N)$ and since $PD_k(N)([\psi])=c_Y$.
\end{proof}

\section{A characterization of embeddings}
\label{AppendixB}

For the lack of a reference to an elementary proof of the following characterization
of embeddings, we provide a proof here.

\begin{prop}
	\label{prop.crit_embedding}
	Let $f \colon X \to Y$ be a morphism of varieties. Then the following are
	equivalent:
	\begin{enumerate}[label=\alph*), leftmargin=*]
		\item \label{prop.crit_embedding_a} $f$ is an embedding;
		\item \label{prop.crit_embedding_b}
		$f$ is proper, injective and for each $x \in X$ the differential 
		$\textrm{d}_x  f \colon T_x X \to T_{f(x)} Y$
		is injective.
	\end{enumerate}
\end{prop}

For the proof, we use the following two lemmas from commutative algebra.

\begin{lemma}
	\label{lem.algebra}
	Let $B$ be a ring and $S \subset B$ be a multiplicative set such that the localization
	$R \coloneqq S^{-1}B$ is a local ring. Denote by $\nn$ the maximal ideal of $R$, 
	by $\varphi \colon B \to R$ the canonical homomorphism and set 
	$\mm \coloneqq \varphi^{-1}(\nn)$. 
	
	Then there exists an isomorphism
	$\psi \colon R \to B_{\mm}$ such that $\psi \circ \varphi \colon B \to B_{\mm}$
	is the canonical homomorphism of the localization.
\end{lemma}

\begin{proof}[Proof of Lemma~\ref{lem.algebra}]
	As $\varphi(S)$ consists of units in $R$, we get $\varphi(S) \subset R \setminus \nn$, i.e.~$S \subset B \setminus \mm$. By the universal property of localizations there exists
	a homomorphism $\psi \colon R \to B_{\mm}$ such that $\psi \circ \varphi$ is equal to the
	canonical homomorphism $\iota \colon B \to B_{\mm}$. Thus it is enough to show that
	$\psi$ is an isomorphism.

	By definition $\varphi(B \setminus \mm) \subset R \setminus \nn$, i.e.~$\varphi(B \setminus \mm)$ consists of units in $R$.
	Hence $\varphi \colon B \to R$ factors through
	$\iota \colon B \to B_{\mm}$, there exists $\theta \colon B_{\mm} \to R$ such that
	$\theta \circ \iota = \varphi$. Thus the following commutative diagram exists:
	\[
	\xymatrix{
		&  B \ar[rd]^-{\varphi} \ar[ld]_-{\varphi} \ar[d]^-{\iota} \\
		R \ar[r]^{\psi} & B_{\mm} \ar[r]^{\theta} & R \, .
	}
	\]
	For $r \in R$ there exist $b \in B$ and $s \in S$ with $r = \frac{b}{s}$ in $R$ and we get
	\[
	(\theta \circ \psi)(r) = (\theta \circ \psi)(\varphi(b) \varphi(s)^{-1}) 
	= \varphi(b) \varphi(s)^{-1} = r \, .
	\]
	Hence $\theta \circ \psi$ is the identity on $R$ and in particular, $\psi$ is injective.
	On the other hand, let $\frac{b}{t} \in B_{\mm}$ where $b \in B$ and $t \in B \setminus \mm$.
	Since $\varphi(B \setminus \mm)$ consists of units in $R$, we get $\varphi(b) \varphi(t)^{-1} \in R$
	and thus $\psi(\varphi(b) \varphi(t)^{-1}) = \iota(b) \iota(t)^{-1} = \frac{b}{t}$. This shows that $\psi$
	is surjective.
\end{proof}

\begin{lemma}
	\label{lem.isom}	
	Let $A \subset B$ be a ring extension of Noetherian local rings
	where $\mm_A$ and $\mm_B$ denote the maximal ideals of $A$ and $B$, respectively. 
	If
	\begin{enumerate}[label=\alph*), leftmargin=*]
		\item \label{conda} $\mm_A \subset \mm_B$,
		\item \label{condb} the induced field extension $A / \mm_A \subset B / \mm_B$ is trivial,
		\item \label{condc} the induced homomorphism $\mm_A / \mm_A^2 \to \mm_B / \mm_B^2$
		is surjective,
		\item \label{condd} $B$ is a finite $A$-module,
	\end{enumerate} 
	then $A = B$.
\end{lemma}

\begin{proof}[Proof of Lemma~\ref{lem.isom}]
	We claim that $\mm_A B = \mm_B$. Indeed, by~\ref{conda} we know that $\mm_A B \subset \mm_B$.
	Since $\mm_A / \mm_A^2 \to  \mm_B / \mm_B^2$
	is surjective, we get $\mm_B = \mm_A + \mm_B^2$ and inductively
	\begin{equation}
	\label{eq.formula}
	\mm_B = \mm_A + \mm_B^n \quad  \textrm{for each $n \geq 2$} \, .
	\end{equation}
	Let $\pi \colon B \to B / \mm_A B$ be the canonical projection. 
	Since $B / \mm_A B$ is a Noetherian local ring and $\pi(\mm_B)$ is a proper ideal of $B / \mm_A B$, 
	Krull's intersection theorem implies the second equality below:
	\[
	\pi(\mm_B) \stackrel{\eqref{eq.formula}}{=} \bigcap_{n \geq 1} \pi(\mm_B)^n = (0) \, .
	\]
	This implies $\mm_B \subset \mm_A B$ and proves the claim.
	
	Since by~\ref{condb}, we have that the field extension
	$A / \mm_A \subset B /  \mm_B$ is trivial, the claim implies now that
	\[
	B = A + \mm_B = A + \mm_A B \, . 
	\]
	This in turn gives us $M = \mm_A M$ for $M = B / A$. Since $B$ is a finite $A$-module,
	$M$ is a finite $A$-module as well. Since $A$ is a local ring with maximal ideal $\mm_A$, 
	we conclude by Nakayama's lemma that $M = 0$, $A =B$.
\end{proof}

\begin{proof}[Proof of Proposition~\ref{prop.crit_embedding}]
	Clearly, \ref{prop.crit_embedding_a} implies~\ref{prop.crit_embedding_b}, hence we are left with
	the proof of the reverse implication and thus we assume~\ref{prop.crit_embedding_b} holds.
	
	Note that $f(X)$ is closed in $Y$, since $f$ is proper. 
	We may therefore replace $Y$ with $f(X)$ and assume in addition that $f$ is surjective.
	Now, we have to show that $f$ is locally an isomorphism. Since $f$ is proper and injective,
	it is finite; see~\cite[Corollary 12.89]{GoWe2010Algebraic-geometry}. Thus, for each $x \in X$ 
	there exists an open affine neighbourhood $U \subset Y$ of $f(x)$ 
	such that $f^{-1}(U)$ is affine and $\OO_X(f^{-1}(U))$ is a finite $\OO_Y(U)$-module
	via the induced homomorphism $f_U^\ast \colon \OO_Y(U) \to \OO_X(f^{-1}(U))$. 
	As $f$ is surjective, $f_U^\ast$ is injective.
	
	Let $A \coloneqq \OO_{Y}(U)$, $B \coloneqq \OO_{X}(f^{-1}(U))$ and denote by $\mm_A$, $\mm_B$
	the maximal ideals corresponding to the points $f(x)$, $x$, respectively. We identify
	$A$ with a subring of $B$ and then 
	$\mm_A = \mm_B \cap A$. 
	By the flatness of $A \to A_{\mm_A}$ the homomorphism
	$A_{\mm_A} \to A_{\mm_A} \otimes_A B$
	is injective and it is finite, since $A \subset B$ is finite. 
	Let $R \coloneqq A_{\mm_A} \otimes_A B$. Then $R$ is the localization of
	$B$ at the multiplicative set $A \setminus \mm_A$.
	Hence, we have a commutative push-out diagram
	\begin{equation}
	\begin{gathered}
	\label{eq.comm_diagram}
	\xymatrix@=15pt{
		A_{\mm_A} \ar@{}[r]|-{\subset} & R \\
		A \ar[u]^-{\iota_A} \ar@{}[r]|-{\subset} & B \ar[u]_-{\varphi}
	}
	\end{gathered}
	\end{equation}
	where $\iota_A$ and $\varphi$ denote the canonical homomorphisms into the corresponding
	localizations.
	
	Let $\nn$ be a maximal ideal in $R$. We claim that $\nn = \varphi(\mm_B)R$.
	Indeed, $\nn \cap A_{\mm_A}$ is a maximal ideal of $A_{\mm_A}$, since $A_{\mm_A} \subset R$
	is finite; see~\cite[Lemma~2, \S9]{Ma1986Commutative-ring-t}. This implies the first equality
	below and the second one follows from the commutativity of~\eqref{eq.comm_diagram}:
	\begin{equation}
	\label{eq.m_A}
	\mm_A =  \iota_A^{-1}(\nn \cap A_{\mm_A}) = \varphi^{-1}(\nn) \cap A \, .
	\end{equation}
	Since $\varphi^{-1}(\nn)$ is a prime ideal of $B$, $\varphi^{-1}(\nn) \cap A = \mm_A$ is a maximal
	ideal of $A$ and since $A \subset B$ is finite,
	it follows from~\cite[Lemma~2, \S9]{Ma1986Commutative-ring-t} that $\varphi^{-1}(\nn)$
	is a maximal ideal of $B$. Since $f \colon X \to Y$ is injective, $\mm_B$ is the only
	maximal ideal in $B$ with $\mm_B \cap A = \mm_A$. By~\eqref{eq.m_A}, we get now
	$\varphi^{-1}(\nn) = \mm_B$, which implies the claim.
	
	Using the claim, $\varphi(\mm_B) R$ is the unique maximal ideal in $R$.
	In particular $R$ is a local ring and $\mm_B = \varphi^{-1}(\varphi(\mm_B) R)$.
	By Lemma~\ref{lem.algebra} there is an isomorphism
	$\psi \colon R \to B_{\mm_B}$ such that $\psi \circ \varphi$ is equal to
	the canonical homomorphism $\iota_B \colon B \to B_{\mm_B}$ of the localization.
	Hence we may identify $R$ with $B_{\mm_B}$ and $\varphi$ with $\iota_B$ and we have to 
	show now that $A_{\mm_A} = B_{\mm_B}$. 
	However, this follows from Lemma~\ref{lem.isom} applied
	to the ring extension $A_{\mm_A} \subset B_{\mm_B}$ (condition~\ref{condc} in Lemma~\ref{lem.isom} follows from the injectivity of $\textrm{d}_x f \colon T_x X \to T_{f(x)} Y$
	and condition~\ref{condb}
	follows from the fact that $A / \mm_A = A_{\mm_A} / \mm_A A_{\mm_A}$, 
	$B / \mm_B = B_{\mm_B} / \mm_B B_{\mm_B}$
	and from the assumption that the ground field is algebraically closed).
\end{proof}

\
\section[Non-existence of proper surjective morphisms]
{Non-existence of proper surjective morphisms of affine spaces into homogeneous spaces}
\label{Appendix.Lefschetz}

In this last appendix, we prove a version of Proposition~\ref{prop:intronomapsZtoG/H} that 
works over an arbitrary algebraically closed field $\k$ of characteristic zero:

\begin{prop}
	\label{prop:intronomapsZtoG/H_over_any_k}
	Let $d \geq 1$. If $G/H$ is a $d$-dimensional homogeneous space of a simple algebraic group $G$, then there is no proper surjective
	morphism from $\AA^d$ to $G/H$.
\end{prop} 

The idea is simply to reduce the situation to the 
case of complex numbers and then to use Proposition~\ref{prop:intronomapsZtoG/H}.
In other words, we check that the Lefschetz principle holds for the specific statement we need. 

	
For the proof we make the following convention. If $X$ is a variety
over $\k$ and if $\k \subset K$ is a field extension such that $K$ is algebraically
closed as well, we denote by $X_K$ the fiber product $X \times_{\Spec \k} \Spec K$. 
In case $X$ is affine, we will denote the coordinate ring of $X$ by $\k[X]$; in particular we then have 
$K[X_K] = K \otimes_{\k} \k[X]$.
In the proof we will
use the following properties of $G_K$ for an algebraic group $G$ over $\k$:

\begin{lemma}
	\label{lem.extension_of_scalars}
	Let $\k \subset K$ be a field extension such that $K$ is algebraically closed and let $G$ be an 
	algebraic group over $\k$. Then the following holds:
	\begin{enumerate}[leftmargin=*]
		\item \label{lem.extension_of_scalars3} The algebraic group $G$ is connected if and only
		if $G_K$ is connected.
		\item \label{lem.extension_of_scalars0} 
		The group of $\k$-rational points $G(\k)$ is dense in $G_K$.
		\item \label{lem.extension_of_scalars1} Let $H$ be a closed subgroup over $\k$ of $G$.
		Then $G_K / H_K = (G/H)_K$.
		\item \label{lem.extension_of_scalars2} If $G^\circ$ denotes the
		identity component of $G$, then $(G^\circ)_K = (G_K)^\circ$.
		\item \label{lem.extension_of_scalars4} Assume that $G$ is connected. Then, 
		$G$ is simple (semisimple, reductive) 
		if and only if $G_K$ is simple (semisimple, reductive).
	\end{enumerate}
\end{lemma}

\begin{remark}
	\label{rem.reductive}
	Let $G$ be a non-trivial algebraic group $G$ over $\k$. Then $G$ 
	is reductive if and only if the identity component $G^\circ$
	is reductive or trivial. 
	Hence, for any field extension $\k \subset K$ where $K$ is algebraically closed, the algebraic
	group $G$ is reductive if and only if $G_K$ is (see Lemma~\ref{lem.extension_of_scalars}).
\end{remark}

\begin{proof}[Proof of Lemma~\ref{lem.extension_of_scalars}]
	\eqref{lem.extension_of_scalars3}: If $G$ is connected, then
			$\k[G]$ is an integral domain. 
			There is a canonical inclusion $K[G_K] = K \otimes_{\k} \k[G] \subset K \otimes_{\k} \k(G)$
			where $\k(G)$ denotes the field of rational functions on $G$. Since $\k$ is algebraically
			closed, by~\cite[Corollary~1, \S15, Ch.~III]{ZaSa1958Commutative-algebr}, we get that
			$K \otimes_{\k} \k(G)$ is an integral domain and thus $G_K$ is connected.
			
			If $G_K$ is connected, then $K[G_K] = K \otimes_{\k} \k[G]$
				is an integral domain. As $\k \subset K$ is an inclusion, it follows that
				$\k[G] \to K \otimes_{\k} \k[G]$ is an inclusion and thus $\k[G]$
				is an integral domain as well. This shows that $G$ is connected.

	\eqref{lem.extension_of_scalars0}: Note that a $\k$-rational point of $G$ corresponds to 
	a $\k$-algebra homomorphism $\k[G] \to \k$ which in turn induces a $K$-algebra
	homomorphism $K[G_K] = K \otimes_{\k} \k[G] \to K \otimes_{\k} \k = K$ and thus gives a (closed) point in
	$G_K$. In this way we see $G(\k)$ as a subgroup of $G_K$. 
	
	Denote by $G^\circ$ the identity component. Hence
	there exists a finite set $E \subset G(\k)$ such that $G = \coprod_{e \in E} e \cdot G^\circ$. Since
	$(G^\circ)_K$ is connected (see~\eqref{lem.extension_of_scalars3}), 
	it follows from \cite[18.3 Corollary]{Bo1991Linear-algebraic-g} that 
	$G^\circ(\k)$ is dense in $(G^\circ)_K$. Hence
	\[
		G(\k) = \coprod_{e \in E} e \cdot G^\circ(\k)	\quad \textrm{is dense in} \quad
		G_K = \coprod_{e \in E} e \cdot (G^\circ)_K \, .
	\]
	
	\eqref{lem.extension_of_scalars1}: Denote by $\pi_K \colon G_K \to (G/H)_K$ the pull-back
	of the natural projection $\pi \colon G \to G/H$. Let $\pr \colon H \times G \to G$ be the projection onto the
	second factor. Since $\pi$ is $H$-invariant, we get the commutativity of
	\[
		\xymatrix@=9pt{
			G \ar[dd]^-{\pr} \times H \ar[rrr]^-{(g, h) \mapsto g \cdot h} &&& G \ar[dd]_-{\pi}  \\ \\
			G \ar[rrr]^-{\pi} &&& G/H 
		}	
		\quad \xymatrix@=9pt{\\ \textrm{and thus of} \\} \quad
		\xymatrix@=8.5pt{
					G_K \ar[dd]^-{\pr_K} \times H_K \ar[rrr]^-{(g, h) \mapsto g \cdot h} &&& G_K \ar[dd]_-{\pi_K}  \\ \\
					G_K \ar[rrr]^-{\pi_K} &&& (G/H)_K \, .
				}	
	\] 
	This shows that $\pi_K$ is $H_K$-invariant. In particular,
	there exists a morphism $\theta \colon G_K / H_K \to (G/H)_K$ such that $\pi_K$ factors as
	\begin{equation}
		\label{eq.composition_def_theta}
		G_K \to G_K/H_K \stackrel{\theta}{\to} (G/H)_K
	\end{equation}
	where the first morphism denotes the canonical projection.
	
	Let $U \subset G/H$ be an open affine subvariety and let 
	$V \to U$ be a finite \'etale morphism 
	such that $V \times_U \pi^{-1}(U) \to V$ is a trivial principal $H$-bundle. In particular,
	$V \times_U \pi^{-1}(U) \simeq U \times H$ is affine and since 
	$V \times_U G \to \pi^{-1}(U)$ is finite,
	it follows that $\pi^{-1}(U)$ is affine by Chevalley's Theorem~\cite[Theorem 12.39]{GoWe2010Algebraic-geometry}. 
	Using~\eqref{eq.composition_def_theta}, we get that the restriction 
	$\pi_K |_{\pi^{-1}(U)_K} \colon \pi^{-1}(U)_K \to U_K$ factorizes as
	\[
		\pi^{-1}(U)_K \to \pi^{-1}(U)_K / H_K \stackrel{\theta_{U_K}}{\longrightarrow} U_K \, ,
	\]
	where $\theta_{U_K}$ denotes the restriction of $\theta$ to $\pi^{-1}(U)_K / H_K$.
	Since $\pi^{-1}(U)$ is affine, we get $K[\pi^{-1}(U)_K] = K \otimes_{\k} \k[\pi^{-1}(U)]$.
	
	We claim that $\theta_{U_K}$ is an isomorphism. To achieve this
	it is enough to show that the induced map of
	$\theta_{U_K}$ on global sections of the structure sheaves is a
	$K$-algebra isomorphism (since $U_K$ is affine). Since~$U_K = (\pi^{-1}(U) /H)_K$, 
	this amounts to showing that the invariant rings
	satisfy
	\[
		 (K \otimes_{\k} \k[\pi^{-1}(U)])^{H_K} = K \otimes_{\k} \k[\pi^{-1}(U)]^H \quad	
		 \textrm{inside $K \otimes_{\k} \k[\pi^{-1}(U)]$} \, .
	\]
	The inclusion `$\supseteq$' follows from the existence of $\theta_{U_K}$. For the reverse
	inclusion let $(e_i)_i$ be a $\k$-basis of the $\k$-vector space $K$ and let 
	$\sum_i e_i \otimes_{\k} f_i \in K \otimes_{\k} \k[\pi^{-1}(U)]$ be $H_K$-invariant
	(almost all $f_i \in \k[\pi^{-1}(U)]$ are zero). In particular, we get for all $h \in H(\k)$ and $g \in \pi^{-1}(U)$ that
	\[
		\sum_i e_i f_i(h \cdot g) = \sum_i e_i f_i(g) \quad \textrm{inside $K$} \, .
	\]
	As $(e_i)_i$ is a $\k$-basis for $K$, 
	we get $f_i(h \cdot g) = f_i(g)$ for each $h \in H(\k)$, each $g \in \pi^{-1}(U)$
	and each $i$. This implies $f_i \in \k[\pi^{-1}(U)]^H$ for each $i$ and shows `$\subseteq$'.
	Hence $\theta_{U_K}$ is an isomorphism.
	
	As we may cover $G/H$ by open affine subvarieties $U$ such that there is a finite \'etale morphism
	$V \to U$ that trivializes $\pi$ over $U$,
	it follows that $\theta$ is an isomorphism.
	
	\eqref{lem.extension_of_scalars2}: 
		The connectedness of $(G^\circ)_K$ follows from the connectedness of $G^\circ$; see~\eqref{lem.extension_of_scalars3}. 
		Since $(G/ G^\circ)_K = G_K/(G^\circ)_K$ is finite (see \eqref{lem.extension_of_scalars1}), 
		we get that $(G^\circ)_K$ is the identity component in $G_K$.

	\eqref{lem.extension_of_scalars4}: Let $T$ be a maximal algebraic torus of $G$, denote by $\mathfrak{X}(T)$
	the character lattice of $T$ and denote by
	$\g$ the Lie algebra of $G$. Moreover, let $R \subset \mathfrak{X}(T)$ be 
	the roots of $\g$ with respect to $T$ and for each $\alpha \in R$, let $\g^\alpha$ denote the corresponding 
	eigenspace. Hence we get
	\[
			\g = \g^0 \oplus \bigoplus_{\alpha \in R} \g^{\alpha} \, .
	\]
	Note that $K \otimes_{\k} \g$ is the Lie algebra of $G_K$,
	that we may naturally identify $\mathfrak{X}(T)$ with $\mathfrak{X}(T_K)$ and that the natural $T$-action on $\g$
	induces naturally a $T_K$-action on $K \otimes_{\k} \g$.
	Since 
	$(K \otimes_{\k} \g)^{\alpha} \supset K \otimes_{\k} \g^{\alpha}$ for each $\alpha \in R$
	and $(K \otimes_{\k} \g)^0 \supset K \otimes_{\k} \g^0$, we get
	\[
			(K \otimes_{\k} \g)^0 \oplus \bigoplus_{\alpha \in R} (K \otimes_{\k} \g)^{\alpha}  = 
			K \otimes_{\k} \g = K \otimes_{\k} \g^0 \oplus \bigoplus_{\alpha \in R} K \otimes_{\k} \g^{\alpha}
	\] 
	and 
	\[
			(K \otimes_{\k} \g)^0 = K \otimes_{\k} \g^0 \, , \quad
			(K \otimes_{\k} \g)^{\alpha} = K \otimes_{\k} \g^\alpha \ \textrm{for each $\alpha \in R$} \, .
	\]
	
	We assume first that $G$ is semisimple (reductive).
	Using that $G_K$ is connected, 
	it follows from~\cite[Proposition 1.12, Corollaire 1.13, Exp. XIX]{DeGr2011Schemas-en-groupes} that $T_K$ is a maximal algebraic torus in $G_K$ and
	$G_K$ is semisimple (reductive). If $G$ is simple, then $G_K$ is semisimple. Moreover, 
	the roots system of $G$ with respect to $T$ is irreducible, and thus
	the root system of $G_K$ with respect to $T_K$ is irreducible as well. Hence, if $G$ is simple, then $G_K$ is simple as well.
	
	Assume now that $G_K$ is simple. If there exists a proper connected normal
	subgroup $N$ over $\k$ of $G$, then $N_K$ is a proper connected normal subgroup
	of $G_K$, since $G(\k)$ is dense in $G_K$ and $N(\k)$ is dense in $N_K$; see~\eqref{lem.extension_of_scalars0}.
	Hence, $N_K$ contains only the identity and 
	thus $N$ as well. Moreover, as $G_K$ is non-commutative and $G(\k)$
	is dense in $G_K$, we get that $G$ is non-commutative.
	Analogously one shows that $G$ is semisimple (reductive) in case $G_K$ is semisimple (reductive).
\end{proof}

\begin{proof}[Proof of Proposition~\ref{prop:intronomapsZtoG/H_over_any_k}]
	We assume towards a contradiction that there exists a proper surjective morphism $\varphi \colon \AA^n \to G/H$. In particular,
	$\varphi$ is quasi-finite and since $\varphi$ is proper, we conclude that $\varphi$ is finite; 
	see~\cite[Corollary 12.89]{GoWe2010Algebraic-geometry}.
	By Chevalley's Theorem~\cite[Theorem 12.39]{GoWe2010Algebraic-geometry}, $G/H$ is affine. In particular, we get a finite ring extension
	\[
		\k[G/H] \subseteq \k[x_1, \ldots, x_n] \, ,
	\]
	where $x_1, \ldots, x_n$ are variables. 
	By the assumption there exist monic polynomials $f_1, \ldots, f_n \in \k[G/H][T]$
	such that $f_1(x_1) = \ldots = f_n(x_n) = 0$. 
	
	There exists an algebraically closed subfield $\k' \subset \k$ of finite transcendence 
	degree over $\QQ$, an algebraic group $G'$ over $\k'$, and a 
	proper subgroup $H'$ over $\k'$ of $G'$ such that $G = G'_{\k}$ and $H = H'_{\k}$.
	
	Since $G/H$ is affine and $G$ is reductive, $H$ is reductive or trivial 
	(see \cite[Theorem 3.8]{Ti2011Homogeneous-spaces}).
	By Remark~\ref{rem.reductive}, $H'$ is reductive or trivial, and thus $G'/H'$ is affine.
	Hence, $\k'[G'/H']$ is a finitely generated $\k'$-algebra, and thus there exists a surjective 
	$\k'$-algebra homomorphism $\eta' \colon \k'[y_1, \ldots, y_m] \to \k'[G'/H']$, where $y_1, \ldots, y_m$
	are new variables. By Lemma~\ref{lem.extension_of_scalars}\eqref{lem.extension_of_scalars1}
	\[
		\eta \coloneqq \k \otimes_{\k'} \eta \colon \k[y_1, \ldots, y_m] \to \k \otimes_{\k'} \k'[G'/H'] = \k[G/H]  
	\] 
	is a surjective $\k$-algebra homomorphism. For each $i=1, \ldots, n$, let $d_i \coloneqq \deg(f_i) > 0$ and let
	$p_{ij} \in \k[y_1, \ldots, y_m]$, where $j =0, \ldots, d_i-1$,
	such that 
	\[
		f_i = T^{d_i} + \sum_{j=0}^{d_i-1} \eta(p_{ij}) T^j \, .
	\]
	By enlarging $\k'$ we may assume in addition that the coefficients of all the 
	$p_{ij} \in \k[y_1, \ldots, y_m]$ and all the $\eta(y_i) \in \k[x_1, \ldots, x_n]$ 
	are contained in $\k'$. 
	In particular, the polynomial $f_i$ has coefficients in $\k'[G'/H']$ for each $i$ and
	\begin{equation}
		\label{eq.restricted_extension}
		\k'[G' / H'] \subseteq \k'[x_1, \ldots, x_n] \, .
	\end{equation}
	As $f_i(x_i) = 0$ for each $i$, we get that~\eqref{eq.restricted_extension} is a finite
	ring extension.
	
	Since the field extension $\QQ \subset \k'$ has finite transcendence degree, there exists
	an embedding of $\k'$ into the field of complex numbers $\CC$. Hence,
	\[
		\CC[(G'/H')_{\CC}] = \CC \otimes_{\k'} \k'[G' / H'] \subset \CC[x_1, \ldots, x_n]
	\]
	is a finite ring extension and $G'_{\CC} / H'_{\CC} = (G'/H')_{\CC}$ is affine.
	Thus, we get a finite surjective morphism $\AA^n_{\CC} \to G'_{\CC} / H'_{\CC}$. Since $G$ simple, we get that $G'$ is simple, and thus  also $G'_{\CC}$ is simple; see Lemma~\ref{lem.extension_of_scalars}\eqref{lem.extension_of_scalars4}. 
	This contradicts Proposition~\ref{prop:intronomapsZtoG/H}.
\end{proof}

\par\bigskip
\renewcommand{\MR}[1]{}
\bibliographystyle{amsalpha}
\bibliography{BIB}

\newcommand{\etalchar}[1]{$^{#1}$}
\providecommand{\bysame}{\leavevmode\hbox to3em{\hrulefill}\thinspace}
\providecommand{\MR}{\relax\ifhmode\unskip\space\fi MR }
\providecommand{\MRhref}[2]{%
  \href{http://www.ams.org/mathscinet-getitem?mr=#1}{#2}
}
\providecommand{\href}[2]{#2}
\begin{thebibliography}{AFRW16}

\bibitem[AFK{\etalchar{+}}13]{ArFlKa2013Flexible-varieties}
I.~Arzhantsev, H.~Flenner, S.~Kaliman, F.~Kutzschebauch, and M.~Zaidenberg,
  \emph{Flexible varieties and automorphism groups}, Duke Math. J. \textbf{162}
  (2013), no.~4, 767--823.

\bibitem[AFRW16]{AnFoRi2016Proper-holomorphic}
Rafael Andrist, Franc Forstneri\v{c}, Tyson Ritter, and Erlend~Forn\ae{ss}
  Wold, \emph{Proper holomorphic embeddings into {S}tein manifolds with the
  density property}, J. Anal. Math. \textbf{130} (2016), 135--150.

\bibitem[BMS89]{BlMuSz1989Zero-cycles-and-th}
Spencer Bloch, M.~Pavaman Murthy, and Lucien Szpiro, \emph{Zero cycles and the
  number of generators of an ideal}, no.~38, 1989, Colloque en l'honneur de
  Pierre Samuel (Orsay, 1987), pp.~51--74.

\bibitem[Bor91]{Bo1991Linear-algebraic-g}
Armand Borel, \emph{Linear algebraic groups}, second ed., Graduate Texts in
  Mathematics, vol. 126, Springer-Verlag, New York, 1991.

\bibitem[Bou71]{Bo1971Elements-de-mathem}
N.~Bourbaki, \emph{\'{E}l\'{e}ments de math\'{e}matique. {T}opologie
  g\'{e}n\'{e}rale. {C}hapitres 1 {\`a} 4}, Hermann, Paris, 1971.

\bibitem[Bri11]{Br2011On-the-geometry-of}
Michel Brion, \emph{On the geometry of algebraic groups and homogeneous
  spaces}, J. Algebra \textbf{329} (2011), 52--71.

\bibitem[Chi89]{Chi_89}
E.~M. Chirka, \emph{Complex analytic sets}, Mathematics and its Applications
  (Soviet Series), vol.~46, Kluwer Academic Publishers Group, Dordrecht, 1989,
  Translated from the Russian by R. A. M. Hoksbergen. \MR{1111477 (92b:32016)}

\bibitem[DDK10]{DoDvKa2010Algebraic-density-}
F.~Donzelli, A.~Dvorsky, and S.~Kaliman, \emph{Algebraic density property of
  homogeneous spaces}, Transform. Groups \textbf{15} (2010), no.~3, 551--576.
  \MR{2718937}

\bibitem[DG11]{DeGr2011Schemas-en-groupes}
Michel Demazure and Alexandre Grothendieck, \emph{Sch\'{e}mas en groupes ({SGA}
  3). {T}ome {III}. {S}tructure des sch\'{e}mas en groupes r\'{e}ductifs},
  Documents Math\'{e}matiques (Paris) [Mathematical Documents (Paris)], vol.~8,
  Soci\'{e}t\'{e} Math\'{e}matique de France, Paris, 2011, S\'{e}minaire de
  G\'{e}om\'{e}trie Alg\'{e}brique du Bois Marie 1962--64. [Algebraic Geometry
  Seminar of Bois Marie 1962--64], A seminar directed by M. Demazure and A.
  Grothendieck with the collaboration of M. Artin, J.-E. Bertin, P. Gabriel, M.
  Raynaud and J-P. Serre, Revised and annotated edition of the 1970 French
  original.

\bibitem[EG92]{ElGr1992Embeddings-of-Stei}
Yakov Eliashberg and Mikhael Gromov, \emph{Embeddings of {S}tein manifolds of
  dimension {$n$} into the affine space of dimension {$3n/2+1$}}, Ann. of Math.
  (2) \textbf{136} (1992), no.~1, 123--135.

\bibitem[FHT01]{FeHaTh2001Rational-homotopy-}
Yves F\'{e}lix, Stephen Halperin, and Jean-Claude Thomas, \emph{Rational
  homotopy theory}, Graduate Texts in Mathematics, vol. 205, Springer-Verlag,
  New York, 2001.

\bibitem[FKZ17]{FlKaZa2017Cancellation-for-s}
Hubert Flenner, Shulim Kaliman, and Mikhail Zaidenberg, \emph{Cancellation for
  surfaces revisited. i}, \url{https://arxiv.org/pdf/1610.01805.pdf}, 2017.

\bibitem[FNOP19]{FNOP}
Stefan Friedl, Matthias Nagel, Patrick Orson, and Mark Powell, \emph{A survey
  of the foundations of four-manifold theory in the topological category},
  \url{http://arxiv.org/abs/1910.07372}, 2019.

\bibitem[For70]{Fo1970Plongements-des-va}
Otto Forster, \emph{Plongements des vari\'{e}t\'{e}s de {S}tein}, Comment.
  Math. Helv. \textbf{45} (1970), 170--184. \MR{269880}

\bibitem[Ful98]{Fu1998Intersection-theor}
William Fulton, \emph{Intersection theory}, second ed., Ergebnisse der
  Mathematik und ihrer Grenzgebiete. 3. Folge. A Series of Modern Surveys in
  Mathematics [Results in Mathematics and Related Areas. 3rd Series. A Series
  of Modern Surveys in Mathematics], vol.~2, Springer-Verlag, Berlin, 1998.

\bibitem[FvS19]{FeSa2019Uniqueness-of-embe}
Peter Feller and Immanuel van Santen, \emph{Uniqueness of embeddings of the
  affine line into algebraic groups}, J. Algebraic Geom. \textbf{28} (2019),
  no.~4, 649--698.

\bibitem[GP74]{GuPo_74}
Victor Guillemin and Alan Pollack, \emph{Differential topology}, Prentice-Hall,
  Inc., Englewood Cliffs, N.J., 1974. \MR{0348781}

\bibitem[GR03]{GrRa2003Revetements-etales}
Alexander Grothendieck and Mich{\`e}le Raynaud, \emph{Rev\^{e}tements
  \'{e}tales et groupe fondamental ({SGA} 1)}, Documents Math\'{e}matiques
  (Paris) [Mathematical Documents (Paris)], vol.~3, Soci\'{e}t\'{e}
  Math\'{e}matique de France, Paris, 2003, S\'{e}minaire de g\'{e}om\'{e}trie
  alg\'{e}brique du Bois Marie 1960--61. [Algebraic Geometry Seminar of Bois
  Marie 1960-61], Directed by A. Grothendieck, With two papers by M. Raynaud,
  Updated and annotated reprint of the 1971 original.

\bibitem[Gro58]{Gr1958Torsion-homologiqu}
Alexander Grothendieck, \emph{Torsion homologique et sections rationnelles},
  S\'eminaire Claude Chevalley \textbf{3} (1958) (fr), talk:5.

\bibitem[Gro61]{Gr1961Elements-de-geomet-II}
\bysame, \emph{\'{E}l{\'e}ments de g{\'e}om{\'e}trie alg{\'e}brique. {II}.
  \'{E}tude globale {\'e}l{\'e}mentaire de quelques classes de morphismes},
  Inst. Hautes {\'E}tudes Sci. Publ. Math. (1961), no.~8, 222.

\bibitem[Gur80]{Gu1980Topology-of-affine}
R.~V. Gurjar, \emph{Topology of affine varieties dominated by an affine space},
  Invent. Math. \textbf{59} (1980), no.~3, 221--225. \MR{579701}

\bibitem[GW10]{GoWe2010Algebraic-geometry}
Ulrich G\"{o}rtz and Torsten Wedhorn, \emph{Algebraic geometry {I}}, Advanced
  Lectures in Mathematics, Vieweg + Teubner, Wiesbaden, 2010, Schemes with
  examples and exercises.

\bibitem[Har77]{Ha1977Algebraic-geometry}
Robin Hartshorne, \emph{Algebraic geometry}, Graduate Texts in Mathematics,
  vol.~52, Springer-Verlag, New York, 1977.

\bibitem[Hat02]{hatcher_AT}
Allen Hatcher, \emph{Algebraic topology}, Cambridge University Press,
  Cambridge, 2002.

\bibitem[Hel78]{He1978Differential-geome}
Sigurdur Helgason, \emph{Differential geometry, {L}ie groups, and symmetric
  spaces}, Pure and Applied Mathematics, vol.~80, Academic Press, Inc.
  [Harcourt Brace Jovanovich, Publishers], New York-London, 1978.

\bibitem[HH63]{HaHi1963On-the-existence-a}
Andr\'{e} Haefliger and Morris~W. Hirsch, \emph{On the existence and
  classification of differentiable embeddings}, Topology \textbf{2} (1963),
  129--135.

\bibitem[HM73]{HoMu1973A-rank-2-vector-bu}
G.~Horrocks and D.~Mumford, \emph{A rank {$2$} vector bundle on {${\bf P}^{4}$}
  with {$15,000$} symmetries}, Topology \textbf{12} (1973), 63--81.

\bibitem[Hol75]{Ho1975Embedding-obstruct}
Audun Holme, \emph{Embedding-obstruction for singular algebraic varieties in
  {${\bf P}^{N}$}}, Acta Math. \textbf{135} (1975), no.~3-4, 155--185.

\bibitem[Hop30]{Hopf_30}
Heinz Hopf, \emph{Zur {A}lgebra der {A}bbildungen von {M}annigfaltigkeiten}, J.
  Reine Angew. Math. \textbf{163} (1930), 71--88.

\bibitem[Hum75]{Hu1975Linear-algebraic-g}
James~E. Humphreys, \emph{Linear algebraic groups}, Springer-Verlag, New
  York-Heidelberg, 1975, Graduate Texts in Mathematics, No. 21.

\bibitem[Hum78]{Hu1978Introduction-to-Li}
\bysame, \emph{Introduction to {L}ie algebras and representation theory},
  Graduate Texts in Mathematics, vol.~9, Springer-Verlag, New York-Berlin,
  1978, Second printing, revised. \MR{499562}

\bibitem[Kal91]{Ka1991Extensions-of-isom}
Shulim Kaliman, \emph{Extensions of isomorphisms between affine algebraic
  subvarieties of {$k^n$} to automorphisms of {$k^n$}}, Proc. Amer. Math. Soc.
  \textbf{113} (1991), no.~2, 325--334. \MR{1076575}

\bibitem[Kal20]{Ka2020Extensions-of-isom}
\bysame, \emph{Extensions of isomorphisms of subvarieties in flexible
  varieties}, Transform. Groups \textbf{25} (2020), no.~2, 517--575.
  \MR{4098881}

\bibitem[Kal21]{Ka2021Holme-type-theorem}
Shulim Kaliman, \emph{Holme type theorem for special linear groups},
  \url{https://arxiv.org/pdf/2104.09550.pdf}, 04 2021.

\bibitem[KK04]{KlKr2004A-quick-proof-of-t}
Stephan Klaus and Matthias Kreck, \emph{{A quick proof of the rational Hurewicz
  theorem and a computation of the rational homotopy groups of spheres}},
  Mathematical Proceedings of the Cambridge Philosophical Society \textbf{136}
  (2004), no.~3, 617--623.

\bibitem[Kle74]{Kl1974The-transversality}
Steven~L. Kleiman, \emph{The transversality of a general translate}, Compositio
  Math. \textbf{28} (1974), 287--297. \MR{360616}

\bibitem[KRvS19]{KrReSa2019Is-the-Affine-Spac}
Hanspeter Kraft, Andriy Regeta, and Immanuel van Santen, \emph{{Is the Affine
  Space Determined by Its Automorphism Group?}}, International Mathematics
  Research Notices (2019), \url{https://doi.org/10.1093/imrn/rny281}.

\bibitem[Lan83]{La1983Abelian-varieties}
Serge Lang, \emph{Abelian varieties}, Springer-Verlag, New York-Berlin, 1983,
  Reprint of the 1959 original.

\bibitem[Llu55]{Ll1955Sur-limmersion-des}
Emilio Lluis, \emph{Sur l'immersion des vari\'{e}t\'{e}s alg\'{e}briques}, Ann.
  of Math. (2) \textbf{62} (1955), 120--127.

\bibitem[Mat86]{Ma1986Commutative-ring-t}
Hideyuki Matsumura, \emph{Commutative ring theory}, Cambridge Studies in
  Advanced Mathematics, vol.~8, Cambridge University Press, Cambridge, 1986.

\bibitem[MT91]{MiTo1991Topology-of-Lie-gr}
Mamoru Mimura and Hirosi Toda, \emph{Topology of {L}ie groups. {I}, {II}},
  Translations of Mathematical Monographs, vol.~91, American Mathematical
  Society, Providence, RI, 1991, Translated from the 1978 Japanese edition by
  the authors.

\bibitem[OV90]{OnVi1990Lie-groups-and-alg}
A.~L. Onishchik and {\`E}.~B. Vinberg, \emph{Lie groups and algebraic groups},
  Springer Series in Soviet Mathematics, Springer-Verlag, Berlin, 1990,
  Translated from the Russian and with a preface by D. A. Leites. \MR{1064110}

\bibitem[Pet57]{Pe1957Some-non-embedding}
Franklin~P. Peterson, \emph{Some non-embedding problems}, Bol. Soc. Mat.
  Mexicana (2) \textbf{2} (1957), 9--15. \MR{87940}

\bibitem[Pop11]{Po2011On-the-Makar-Liman}
Vladimir~L. Popov, \emph{On the {M}akar-{L}imanov, {D}erksen invariants, and
  finite automorphism groups of algebraic varieties}, Affine algebraic
  geometry, CRM Proc. Lecture Notes, vol.~54, Amer. Math. Soc., Providence, RI,
  2011, pp.~289--311. \MR{2768646}

\bibitem[Ram64]{Ra1964A-note-on-automorp}
C.~P. Ramanujam, \emph{A note on automorphism groups of algebraic varieties},
  Math. Ann. \textbf{156} (1964), 25--33.

\bibitem[Rem57]{Re_57}
Reinhold Remmert, \emph{Holomorphe und meromorphe {A}bbildungen komplexer
  {R}{\"a}ume}, Math. Ann. \textbf{133} (1957), 328--370. \MR{0092996
  (19,1193d)}

\bibitem[Ros61]{Ro1961Toroidal-algebraic}
Maxwell Rosenlicht, \emph{Toroidal algebraic groups}, Proc. Amer. Math. Soc.
  \textbf{12} (1961), 984--988.

\bibitem[Sch97]{Sc1997Embeddings-of-Stei}
J.~Sch{{\"u}}rmann, \emph{Embeddings of {S}tein spaces into affine spaces of
  minimal dimension}, Math. Ann. \textbf{307} (1997), no.~3, 381--399.

\bibitem[Ser58]{Se1958Espaces-fibres-alg}
Jean-Pierre Serre, \emph{Espaces fibr\'es alg\'ebriques}, S\'eminaire Claude
  Chevalley \textbf{3} (1958) (fr), talk:1.

\bibitem[Sri91]{Sr1991On-the-embedding-d}
V.~Srinivas, \emph{On the embedding dimension of an affine variety}, Math. Ann.
  \textbf{289} (1991), no.~1, 125--132.

\bibitem[{Sta}21]{stacks-project}
The {Stacks project authors}, \emph{The stacks project},
  \url{https://stacks.math.columbia.edu}, 2021.

\bibitem[Tim11]{Ti2011Homogeneous-spaces}
Dmitry~A. Timashev, \emph{Homogeneous spaces and equivariant embeddings},
  Encyclopaedia of Mathematical Sciences, vol. 138, Springer, Heidelberg, 2011,
  Invariant Theory and Algebraic Transformation Groups, 8.

\bibitem[VdV75]{Va1975On-the-embedding-o}
A.~Van~de Ven, \emph{On the embedding of abelian varieties in projective
  spaces}, Ann. Mat. Pura Appl. (4) \textbf{103} (1975), 127--129.

\bibitem[Whi36]{Wh1936Differentiable-man}
Hassler Whitney, \emph{Differentiable manifolds}, Ann. of Math. (2) \textbf{37}
  (1936), no.~3, 645--680.

\bibitem[Whi44]{Wh1944The-self-intersect}
\bysame, \emph{The self-intersections of a smooth {$n$}-manifold in
  {$2n$}-space}, Ann. of Math. (2) \textbf{45} (1944), 220--246.

\bibitem[ZS58]{ZaSa1958Commutative-algebr}
Oscar Zariski and Pierre Samuel, \emph{Commutative algebra, {V}olume {I}}, The
  University Series in Higher Mathematics, D. Van Nostrand Company, Inc.,
  Princeton, New Jersey, 1958, With the cooperation of I. S. Cohen.

\end{thebibliography}

\end{document}